%% file: Main.tex
\documentclass[11pt]{article}

\usepackage[margin = 1in, paperwidth =8.5in, paperheight=11in]{geometry}

\usepackage{enumitem}  
\usepackage{amsmath}
\usepackage{amsfonts}
\usepackage{amssymb}
\usepackage{tikz}
\usepackage{comment}
\usepackage{graphicx}
\usetikzlibrary{arrows}
\usepackage{mathtools}
\usepackage{amsthm}
\usetikzlibrary{backgrounds}
\usepackage{array}
\usepackage{bbm}
\usetikzlibrary{calc}
\usepackage[font=footnotesize]{caption}

\newcommand{\inZ}{\in \mathbb{Z}}
\newcommand{\inN}{\in \mathbb{N}}

\newcommand{\inR}{\in \mathbb{R}}
\newcommand{\inC}{\in \mathbb{C}}
\newcommand{\Z}{\mathbb{Z}}
\newcommand{\N}{\mathbb{N}}

\newcommand{\R}{\mathbb{R}}

\newcommand{\E}{\mathbb{E}}

\newcommand{\setofplanetrees}{\mathfrak{T}}
\newcommand{\treecount}{\mathcal{T}}
\newcommand{\setofcomp}{{\mathcal{C}}_k}
\newcommand{\setoftuples}{{\mathcal{V}}_{k}^*}
\newcommand{\comp}{\mathcal{C}}
\newcommand{\setoftreecomppairs}{{\mathcal{F}}_k^*}
\newcommand{\skele}{\mathcal{S}}
\newcommand{\normConst}{\mathcal{Z}}

\theoremstyle{plain}
\newtheorem{theorem}{Theorem}[subsection]
\newtheorem{corollary}[theorem]{Corollary}
\newtheorem{lemma}[theorem]{Lemma}
\newtheorem{observation}[theorem]{Observation}
\newtheorem{remark}[theorem]{Remark}

\theoremstyle{definition}
\newtheorem{example}[theorem]{Example}

\newcommand\lvl{p} 
\newcommand\lvi{q} 
\newcommand\lvr{y} 
\newcommand\gvl{\overline{p}} 
\newcommand\gvi{\overline{q}} 

\title{On the Asymptotic Distributions of Classes of Subtree Additive Properties of Plane Trees under the Nearest Neighbor Thermodynamic Model}
\author{Anna Kirkpatrick and Chidozie Onyeze}

\begin{document}
\maketitle
\input{Abstract}

\section{Background and Introduction}
\input{Introduction.tex}

\section{Mathematical Preliminaries}
\input{Preliminaries}


\section{Results} \label{Resuls Section}
\input{Results_Sec_1.tex}
\input{Results_Sec_2.tex}

\input{Results_Sec_3.tex}
\input{Other_Results.tex}


\section{Discussion}
\label{sec:discussion}
\input{Discussion.tex}

\section{Acknowledgements}
\input{Acknowledgements.tex}

\bibliographystyle{plain}
\bibliography{References}

\section*{Appendix}
\input{Appendix}
\end{document}

%% file: Abstract.tex
\begin{abstract}
We define a class of properties on random plane trees, which we call subtree additive properties, inspired by the combinatorics of certain biologically-interesting properties in a plane tree model of RNA secondary structure. The class of subtree additive properties includes the Wiener index and path length (total ladder distance and total ladder contact distance, respectively, in the biological context). We then investigate the asymptotic distribution of these subtree additive properties on a random plane tree distributed according to a Gibbs distribution arising from the Nearest Neighbor Thermodynamic Model for RNA secondary structure. We show that for any property in the class considered, there is a constant that translates the uniformly weighted random variable to the Gibbs distribution weighted random variable (and we provide the constant). 
We also relate the asymptotic distribution of another class of properties, which we call simple subtree additive properties, to the asymptotic distribution of the path length, both in the uniformly weighted case. 
The primary proof techniques in this paper come from analytic combinatorics, and most of our results follow from relating the moments of known and unknown distributions and showing that this is sufficient for convergence. 

\end{abstract}

%% file: Introduction.tex
This manuscript uses techniques from analytic combinatorics to explore probability distributions arising from questions in molecular biology. Specifically, the questions explored are inspired by the problem of RNA secondary structure prediction. 

\subsection{Overview of results}
This paper examines a model of RNA secondary structure in which secondary structures are modeled by plane trees. As defined more rigorously in Section~\ref{sec:preliminaries}, we consider the set of all plane trees with \(n\) edges under a Gibbs distribution, where the energy of each tree depends on its degree sequence and root degree. The energy function is also determined by 3 thermodynamic parameters, which we treat as fixed: \((\alpha , \beta , \gamma )\). In this paper, we treat these parameters as arbitrary real numbers. However, specific values of interest do arise from the thermodynamic models used in molecular biology; the biological motivation is discussed further in Section \ref{bioback}.  For a wide class of properties of plane trees, we show a relationship in their asymptotic distributions under different values of \((\alpha , \beta , \gamma )\).

One such property is the path length. The path length of a plane tree is the sum of the distances from each vertex of the tree to the root. It is known that the path length of all plane trees with \(n\) edges is Airy distributed asymptotically as \(n\rightarrow \infty \) (see Theorem \ref{contactdistance - thm}). This case, with the plane trees being uniformly distributed, corresponds to the thermodynamic parameters  \((0, 0, 0)\). Applying Theorem \ref{maintheorem3} to the path length property allows us to relate the asymptotic distribution of path length under arbitrary thermodynamic parameters \((\alpha , \beta , \gamma )\) to this known result. As shown in Corollary \ref{full dist corr CD}, we can explicitly state the asymptotic distribution of path length under arbitrary thermodynamic parameters.

For some properties, such as the path length, the asymptotic distribution in the uniform case is known. For many other properties, such as the sum of the distances from each leaf to the root, this distribution is not known. Therefore in addition to our main result relating asymptotic distributions of plane tree properties under different thermodynamic parameters, we also develop some tools which are useful for determining the asymptotic distributions of certain properties in the uniform case. Combining these tools with the main result discussed above allows us to obtain explicit forms for asymptotic distributions for a wide variety of plane tree properties under arbitrary thermodynamic parameters. More specifically, in Theorem \ref{mainSimpleAdditveTheorem1}, we relate the asymptotic distribution of a class of plane tree properties,  which we call simple subtree additive properties, to the asymptotic distribution of path length, both in the uniform case. As shown in Corollary \ref{corollary P^LRD and P^IRD dist}, this theorem allows us to deduce that the total leaf to root distance is also distributed asymptotically as an Airy random variable.

\subsection{Biological background}
\label{bioback}
We now proceed with some exposition on RNA secondary structure and prior work in this area. The reader interested only in the mathematics may skip this section.

\subsubsection{RNA secondary structure}
RNA is an important biological polymer with roles including information transfer, regulation of gene expression, and catalysis of chemical reactions.
The \emph{primary structure} of an RNA molecule is the sequence of nucleotides in the polymer.
RNA nucleotides are adenine, cytosine, guanine, and urasil, which we frequently abbreviate as A, C, G, and U, respectively.
The primary structure, therefore, may simply be understood as a string of A's, C's, G's, and U's.
Because RNA is single-stranded, it has the capacity to form nucleotide-nucleotide bonds with itself.
The set of such bonds is the \emph{secondary structure} of an RNA molecule.
The bonds A-U, C-G, and G-U are considered canonical, and our the only bonds considered by the model presented here.
The \emph{tertiary structure} of an RNA molecule is its three-dimensional shape.
Though tertiary structure ultimately is most relevant to the determination of function, it is also very difficult to deduce with current laboratory techniques.
Therefore, secondary structure is often used as a first step in the process of predicting tertiary structure \cite{Doudna2000structural, Tinoco99how}.
In fact, secondary structure is often an input to tertiary structure prediction algorithms \cite{massire1998manip,seetin2000automated, yunjie2011improvements, zhao2012automated}.

One of the main computational tools for predicting RNA secondary structures is thermodynamic free energy minimization using Nearest Neighbor Thermodynamics Modeling (NNTM) \cite{Turner_89, mathews1999, mathews2004}. Under the NNTM, the free energy of a structure is computed as the sum of the free energy of its various substructures. This free energy is in turn used in algorithms to predict secondary structure given an RNA sequence, see, e.g., \cite{Ding_Chan_Lawrence_04_Sfold,Hofacker94_fastFolding, mathuriya09_gtfold}. 
Though such algorithms perform reasonably well on short RNA sequences, performance rapidly degrades once sequence length exceeds a few hundred nucleotides.

\subsubsection{Multiloops and branching}

This paper investigates an aspect of RNA secondary structure that becomes more significant as sequence length grows: multiloops.
A \emph{multiloop} is a place where 3 or more helices meet in an RNA secondary structure. 
Multiloops are not predicted well by the current NNTM energy assignments \cite{Doshi_04}.
The number and type of multiloops determines the branching behavior of an RNA secondary structure.

We study a model for RNA secondary structure first presented by Hower and Heitsch \cite{hower2011}.
This model isolates the multiloops and branching properties of secondary structure, allowing their study without the necessity to consider the identity of individual base pairs.
Under the model, secondary structures are placed in bijection with plane trees. 
The minimum energy structures under the model were characterized by Hower and Heitsch in the original paper, but this leaves open the question of the full Gibbs distribution of possible structures, as well as the question of characterizing asymptotic behavior of the  distribution. 
Kirkpatrick et al. \cite{Kirkpatrick_et_al_20} have shown the existence of a polynomial-time Markov chain-based algorithm for sampling from the Gibbs distribution on structures of a fixed size.
Bakhtin and Heitsch~\cite{Bakhtin_Heitsch_09} analyzed a simplification of the model and determined degree sequence properties of the distribution of plane trees asymptotically.

Several properties are used to describe the overall branching behavior of an RNA secondary structure.
In particular, ladder distance and contact distance are used to characterize various aspects of a molecule's shape, size, and branching structure (see, e.g. \cite{borodavka2016sizes, Yoffe_08}).
As discussed in Section \ref{sec:examples}, ladder distance and contact distance correspond to Wiener index and path length, respectively, of plane trees.

We will examine the distribution of several plane tree properties asymptotically under the full version of the Hower and Heitsch model.
Under an assumption that the root degree is bounded, the theorems developed will allow us to characterize the asymptotic distribution of many properties of RNA secondary structures under the NNTM.
We will further show that altering the parameters \((\alpha, \beta, \gamma)\) only changes the property distribution by a constant multiple.
When the assumption that the root degree is bounded is removed, we will still obtain analogous results for parameter sets of the form \((\alpha , \beta , 0)\).

\subsection{Structure of this paper}
In Section~\ref{sec:preliminaries}, we give an overview the necessary mathematical preliminaries (including relevant known results). In Section \ref{Gen function count section}, we construct generating functions counting plane trees. In Section \ref{Distributions of Simple Subtree Additive Properties}, we relate the moments (and hence the distribution) of a class of subtree additive properties we call simple. In Section \ref{Dist of Subtree Additive Properties NNTM}, we apply the generating functions in Section \ref{Gen function count section} to relate the moments of a class of subtree additive properties under the uniformly weighted distribution on the plane trees to the same properties under the non-uniformly weighted distribution on the plane trees arising from the Nearest Neighbor Thermodynamic Model. We, hence, show that there exists a constant that translates the asymptotic random variable in the uniformly weighted case to the asymptotic random variable in the non-uniformly weighted case. In Section \ref{Counting trees - enumerations}, we provide some miscellaneous enumeration results on plane trees.

%% file: Preliminaries.tex
\label{sec:preliminaries}
\subsection{Plane Trees and their Properties}

A \emph{plane tree} is a rooted ordered tree. Let $\setofplanetrees_n$ denote the set of plane trees on $n$ edges. Let $\setofplanetrees_{\leq k} = \cup_{n\leq k} \setofplanetrees_{n}$. It is well-known that $|\setofplanetrees_n|$ is given by the $n$th Catalan number, $C_n  = \frac{1}{n+1}\binom{2n}{n}$. We define the \emph{down degree} of a vertex to be the degree of the vertex when considering the root and one less than the degree of the vertex for all other vertices. We define a \emph{leaf} to be a non-root vertex with down degree 0 and an \emph{internal node} to be a non-root vertex with down degree 1. For a plane tree $T$, let $v(T)$, $n(T)$, ${d_0}(T)$, ${d_1}(T)$ and $r(T)$ be the number of vertices, edges, leaves, internal nodes and root degree of $T$, respectively. Let $\mathcal{V}(T)$ be the vertex set of $T$ and let $\overline{\mathcal{V}}(T)$ be the vertex set excluding the root vertex. For $v \in \mathcal{V}(T)$, let $T_v$ be the subtree of $T$ that contains all descendant of $v$ (including $v$).

For plane trees $T_1 \in \setofplanetrees_n$ and $T_2 \in \setofplanetrees_m$ for some $m,n \geq 0$, we define the join of $T_1$ and $T_2$, $T_1 \ltimes T_2$, to be the tree formed by adding a new edges to leftmost side of the root of $T_1$ and attaching $T_2$ to this new edge. Note that $T_1 \ltimes T_2 \in \setofplanetrees_{n+m+1}$. Notice that for a tree $T \in \setofplanetrees_{> 0}$, there is unique $T_1, T_2 \in \setofplanetrees_{\geq 0}$ such that $T = T_1 \ltimes T_2$.

We define a \emph{property} of a plane tree to be a function $\mathcal{P} : \setofplanetrees_{\geq 0} \rightarrow \R$. We define the 2 types of properties we will consider from this point forward: \emph{additive properties} and \emph{subtree additive properties}. A property $\mathcal{P}$ is \emph{additive} if, for $T \in \setofplanetrees_{\geq 1}$ such that $T = T_1 \ltimes T_2$ for some $T_1, T_2$, 
$$\mathcal{P}(T) = \mathcal{P}(T_1) + \mathcal{P}(T_2) + f(T_2),$$ 
where $f(T): \setofplanetrees_{\geq 0} \rightarrow \R_{\geq 0}$. For a tree $T$, let $v_1, \cdots, v_d$ be the child vertices of the root vertex of $T$. By repeated use of the above definition, we see that 
\begin{equation} \label{eq additive prop def}
\mathcal{P}(T) = \mathcal{P}(T^*) + \sum_{i=1}^d f(T_{v_i}) + \sum_{i=1}^d\mathcal{P}(T_{v_i}),
\end{equation} 
where $d$ is the degree of the root of $T$ and $T^*$ is the tree on 1 vertex.

This is similar to the notion of an additive functional as described by Janson \cite{janson2016asymptotic}. An additive functional is a function $F: \setofplanetrees_{\geq 0} \rightarrow \R$ such that 
$$F(T) = f(T) + \sum_{i=1}^dF(T_{v_i}),$$ 
where $f:\setofplanetrees_{\geq 0} \rightarrow \R$ is known as the toll function. Due to $(\ref{eq additive prop def})$, we note that an additive property (as we have defined it) is an additive functional with toll function $f$ where there exist a function $f^*: \setofplanetrees_{\geq 0} \rightarrow \R_{\geq 0}$ such that $f(T) = c + \sum_{i=1}^df^*(T_{v_i})$ where $c = \mathcal{P}(T^*)$. We will borrow the terminology of additive functionals and call the tuple $(f, c)$ the toll of the additive property. It should be clear that a subtree additive property is uniquely determined by its toll. Thus, we will denote the additive property with a given toll by $\mathcal{P}^{(f,c)}$. We will call $f$, in the toll, the toll function of the subtree additive property. We will call an additive property non-negative integer valued if the co-domain of the toll function is a subset of $\Z_{\geq 0}$ (and $c \in \Z_{\geq 0}$).

It can also be shown inductively that, for $c, c_1, c_2 \inR$, $f_1, f_2: \setofplanetrees_{\geq 0} \rightarrow \R$ and $T \in \setofplanetrees_{n}$, 
\begin{equation} \label{additve prop linearity}
\mathcal{P}^{(c_1 \cdot f_1 + c_2 \cdot f_2, c)}(T) =  c_1 \cdot \mathcal{P}^{(f_1, 0)}(T) + c_2 \cdot \mathcal{P}^{(f_2, 0)}(T) + c \cdot \mathcal{P}^{(0, 1)}(T).
\end{equation}

A property $\mathcal{P}$ is \emph{subtree additive} if, for $T \in \setofplanetrees_{\geq 1}$, 
$$\mathcal{P}(T) = \sum_{v \in \overline{\mathcal{V}}(T)} f(T_v, T).$$ 
where $f: \setofplanetrees_{\geq 0} \times \setofplanetrees_{\geq 0} \rightarrow \R$, and $\mathcal{P}(T^*) =0$. We will call such a property simple if 
$$\mathcal{P}(T) = \sum_{v \in \overline{\mathcal{V}}(T)} \mathcal{P}'(T_v)$$ 
where $\mathcal{P}'$ is a non-negative integer valued additive property. In this case, we call $\mathcal{P}$ the subtree additive property induced by $\mathcal{P}'$.

\bigskip

For a plane tree $T$, the \emph{energy} of the tree is given by 
\begin{equation}
E(T) = \alpha d_0(T) + \beta d_1(T) + \gamma r(T)
\label{energydef}
\end{equation} where $\alpha, \beta, \gamma \inR$ are parameters of the energy function.

For fixed $n \inN$ and parameters $\alpha, \beta, \gamma \inR$ and property $\mathcal{P}$, we define the random variable $\mathcal{P}_{(\alpha, \beta, \gamma)}(\setofplanetrees_n)$ to be $\mathcal{P}(T)$ for a plane tree $T \in \setofplanetrees_n$ selected at random with probability 
$\frac{e^{-E(T)}}{\normConst_{(n,\alpha,\beta,\gamma)}}$
where $\normConst_{(n,\alpha,\beta,\gamma)}$ is a normalizing constant given by 
$\normConst_{(n,\alpha,\beta,\gamma)} = \sum_{T \in \setofplanetrees_n}e^{-E(T)}.$ For convenience, we will denote $\mathcal{P}_{(0,0,0)}(\setofplanetrees_n)$ simply as $\mathcal{P}(\setofplanetrees_n)$.

Let $\mathcal{P}$ and $\overline{\mathcal{P}}$ be properties and let $\alpha_1,\alpha_2, \beta_1, \beta_2, \gamma_1, \gamma_2 \inR$ be parameters. We use $$\mathcal{P}_{(\alpha_1, \beta_1, \gamma_1)}(\setofplanetrees_n)  \overset{d}{\longleftrightarrow}\overline{\mathcal{P}}_{(\alpha_2, \beta_2, \gamma_2)}(\setofplanetrees_n)$$ to imply that there exist a random variable $W$ such that as $n \rightarrow \infty$, $$\mathcal{P}_{(\alpha_1, \beta_1, \gamma_1)}(\setofplanetrees_n)  \overset{d}{\to} W \quad \quad \quad \quad \text{and} \quad \quad\quad \quad \overline{\mathcal{P}}_{(\alpha_2, \beta_2, \gamma_2)}(\setofplanetrees_n) \overset{d}{\to} W,$$ where $\overset{d}{\to}$ denotes convergence in distribution. In this case, we will say that $\mathcal{P}_{(\alpha_1, \beta_1, \gamma_1)}(\setofplanetrees_n)$ and $\mathcal{P}_{(\alpha_2, \beta_2, \gamma_2)}(\setofplanetrees_n)$ are equivalent in distribution.

We now state the well-known Carleman's condition \cite{carleman1926fonctions} which tells us that, under certain conditions (which the random variables we will consider satisfy), to show equivalence in distribution, it is sufficient to show equality is asymptotic moments.

\begin{theorem}[Carleman's condition] \label{Carleman's condition}
    Let $X$ be a random real-valued variable and let $m_k = \E[|X|^k] < \infty$ for all $k \geq 0$. If $$\sum_{k=1}^\infty |m_{2k}|^{-\frac{1}{2k}} = \infty,$$ then there exists a unique distribution with moments $m_k$. 
\end{theorem} 

\subsection{Examples}\label{sec:examples}
\begin{example} \label{Example number of leaves and edges} It can be shown inductively that the number of edges in a tree is given by the additive property with toll function $(f, 0)$ where $f(T) = 1$ for all $T \in\setofplanetrees_{\geq 0}$. We will denote this by $\mathcal{P}^{e}$. Similarly, the number of vertices in a tree is given by the additive property with toll function $(0, 1)$ where $0$ is the zero function. We will denote this by $\mathcal{P}^{v}$. 

\end{example}

\begin{example} 
For $f(T) = 1$ when $T = T^*$ and $f(T) = 0$ otherwise, we observe that $\mathcal{P}^{(f,0)}(T)$ represents the number of leaves in $T$. We will denote this by $\mathcal{P}^{d_0}$. Similarly, for $f(T) = 1$ when $T$ has root degree 1, and $f(T) = 0$ otherwise, $\mathcal{P}^{(f,0)}(T)$ represents the number of internal nodes in $T$. We will denote this by $\mathcal{P}^{d_1}$. Note that for $T \in \setofplanetrees_n$, $\mathcal{P}^{d_0}(T) \leq n$ and $\mathcal{P}^{d_1}(T) <  n$.
\end{example}

\begin{example}
We define the \emph{path length} of $T$, $\mathcal{P}^{PL}(T)$, to be the sum of the edge distances from each vertex of $T$ to the root. In the biological context, this quantity is also known as the \emph{total contact distance}. We can observe that $\mathcal{P}^{PL}(T)$ is a simple subtree additive property since 
$$\mathcal{P}^{PL}(T)= \sum_{v \in \overline{\mathcal{V}}(T)} \mathcal{P}^v(T_v) = \sum_{v \in \overline{\mathcal{V}}(T)} \mathcal{P}^e(T_v) + n$$ where $n = \left|\overline{\mathcal{V}}\right|$ is the number of edges in $T$. This holds because the number of paths from vertices to the root that utilize an edge $e$ is the number of vertices in the subtree directly below $e$.
\end{example}

It can be shown from the work of Tak{\'a}cs \cite{takacs1991bernoulli} or more directly from the work of Janson \cite{janson2003wiener} that the following result about the distribution of the path length holds.
\begin{theorem} 
As $n \rightarrow \infty$,
	\label{contactdistance - thm}
	$$\frac{\mathcal{P}^{PL}(\setofplanetrees_n)}{\sqrt{2n^3}}  \overset{d}{\to} \int_0^1e(t) \: dt$$ where $e(t)$ is a normalized Brownian excursion on $[0, 1]$. Thus, it is Airy Distributed. Furthermore,
	$$\lim_{n \rightarrow \infty}\E\left[\left(\frac{\mathcal{P}^{PL}(\setofplanetrees_n)}{\sqrt{2n^3}}\right)^k\right] \: \sim \: \frac{6k}{\sqrt{2}}\left(\frac{k}{12e}\right)^{\frac{k}{2}}$$ as $k \rightarrow \infty$.
\end{theorem}

\begin{example}
We define the \emph{Wiener index} of $T$, $\mathcal{P}^{WI}(T)$, to be the sum of the edge distances between any 2 vertices. In the biological context, this quantity is also known as the \emph{total ladder distance}. We can observe that $\mathcal{P}^{WI}(T)$ is a subtree additive property since 
$$\mathcal{P}^{WI}(T)= \sum_{v \in \overline{\mathcal{V}}(T)} \mathcal{P}^v(T_v)(\mathcal{P}^v(T) - \mathcal{P}^v(T_v)).$$ This holds because the number of paths between vertices that utilize an edge $e$ is the number of unordered pairs of vertices, one from the subtree below $e$ and the other not from that subtree.
\end{example}

From the work of Janson \cite{janson2003wiener}, the following result about the distribution of the Wiener index holds.
\begin{theorem} As $n \rightarrow \infty$,
	\label{ladderdistance - thm}
	$$\frac{\mathcal{P}^{WI}(\setofplanetrees_n)}{\sqrt{2n^5}}  \overset{d}{\to} \int \int_{0 < s < t < 1} (e(s) + e(t) - 2\min_{s \leq u \leq t}e(u)) \: ds \: dt$$ where $e(t)$ is a normalized Brownian excursion on $[0, 1]$. 
\end{theorem}

\begin{example}
We define the \emph{total leaf to root distance} of $T$, $\mathcal{P}^{LR}(T)$, and the \emph{total internal node to root distance} of $T$, $\mathcal{P}^{IR}(T)$, to be the sum of the edge distances from every leaf vertex to the root and the sum of the edge distances from every internal node to the root, respectively. We can observe that $\mathcal{P}^{LR}(T)$ and $\mathcal{P}^{IR}(T)$ can be described in terms of a simple subtree additive properties as follows.

$$\mathcal{P}^{LR}(T)= \mathcal{P}^{d_0}(T) + \sum_{v \in \overline{\mathcal{V}}(T)} \mathcal{P}^{d_0}(T_v) = \sum_{v \in \overline{\mathcal{V}}(T)} \mathcal{P}^{d_0}(T_v) + O(n)$$ 
and 
$$\mathcal{P}^{IR}(T)= \mathcal{P}^{d_1}(T) + \sum_{v \in \overline{\mathcal{V}}(T)} \mathcal{P}^{d_1}(T_v) = \sum_{v \in \overline{\mathcal{V}}(T)} \mathcal{P}^{d_1}(T_v) + O(n).$$ We see this as follows. The total leaf to root distance is the sum over all edges of the number of path that use that edge which is the number of leaves in the subtree below the edge, which is $T_v$, where $v$ is the vertex below the edge. We, however, notice that by our definition of a leaf, if $v$ is a leaf in $T$, it will not be counted as a leaf in $T_v$. Thus, overall, we under count by the number of leaves in $T$. A similar argument holds for the total internal node to root distance.

In Section \ref{Distributions of Simple Subtree Additive Properties}, we will show that $\frac{\mathcal{P}^{LR}(\setofplanetrees_n)}{\sqrt{n^3}}$ and $\frac{\mathcal{P}^{IR}(\setofplanetrees_n)}{\sqrt{n^3}}$ both converge weakly to an Airy random variable.
\end{example}

\subsection{Generating Functions and Analytic Combinatorics}
For a sequence $(a(n))$ for $n \inZ_{\geq 0}$, we define the \emph{generating function} corresponding to the sequence to be the formal power series, $$F(x) = \sum_{n \geq 0} a(n) x^n.$$ Similarly, for a $k$-dimensional sequence, $(a(n_1, \cdots, n_k))$ for $n_1, \cdots, n_k \inZ_{\geq 0}$, we define the \emph{multivariate generating function} corresponding to the sequence to be $$F(x_1, \cdots, x_k) = \sum_{n_1 \geq 0} \cdots \sum_{n_k \geq 0} a(n_1, \cdots, n_k) x_1^{n_1} \cdots x_k^{n_k}.$$

Conversely, we will use $\left[x_1^{n_1}\cdots x_k^{n_k}\right]F(x_1, \cdots, x_k)$ to denote the coefficient of $x_1^{n_1}\cdots x_k^{n_k}$ in the generating function $F(x_1, \cdots, x_k)$, namely $a(n_1, \cdots, n_k)$.

Fix $(a(n))$ for $n \inZ_{\geq 0}$ and let $F(z)$ be the associated generating function. We can treat $F(z)$ as a function over the complex plane. We say $F(z)$ is \emph{analytic} at a point $z_0 \inC$ if  there exist a region around $z_0$ such that $F(z)$ is differentiable. We say $F(z)$ is analytic on a domain if it is analytic at all points in the domain.

A point, $z_0 \inC$, is a \emph{singularity} of $F(z)$ if $F(z)$ is not analytic at $z_0$. Furthermore, that singularity is \emph{isolated} if there exist $\epsilon > 0$ such that $F(z)$ is analytic on the domain $\{z \inC : 0<|z-z_0| < \epsilon \}$. Let $F(z)$ be such that all its singularities of are isolated. We define a \emph{dominant singularity} of $F(z)$ to be an isolated singularity with minimal distance from the origin. 

From the work of Flajolet and Sedgewick \cite{flajolet2009analytic}, we now state a slightly simplified form of a so called Transfer Theorem that will allow us to deduce asymptotic information about $(a(n))$ using $F(z)$. For some $R>1$ and $0<\phi<\frac{\pi}{2}$, we define a $\Delta$-domain at $1$ to be the domain $$\Delta(\phi, R) = \{z\inC:|z|<R, z\neq 1, |\text{arg}(z-1)|>\phi\}.$$

\begin{theorem}[Transfer Theorem] \label{transfer lemma}
Let $(a(n))$ be a sequence with associated generating functions $F(z)$. Let $F(z)$ be a function analytic at $0$ with a unique dominant singularity at $1$ and let $F(z)$ be analytic in a $\Delta$-domain at $1$, $\Delta_0$. Assume there exist $\sigma, \tau$ such that $\sigma$ is a finite linear combination of terms of the form $(1-z)^{-\alpha}$ for $\alpha \inC$ and $\tau$ is a term of the form $(1-z)^{-\beta}$ for $\beta \inC$ such that, for $z \in \Delta_0$, 
$$F(z) = \sigma\left(z\right) + O\left(\tau\left(z\right)\right) \quad\quad \text{as} \quad\quad z \rightarrow 1.$$
Then, the following asymptotic estimation holds.
$$a(n) = [z^n]F(z) = [z^n]\sigma(z) + O(n^{\beta - 1})\,.$$
\end{theorem}

A basic application of the above result that we will be useful in Section \ref{Resuls Section} is as follows.
\begin{corollary} \label{analytic combo corollary 1} Let $(a(n))$ be a sequence with associated generating functions $F(z)$. Let $\frac{1}{\zeta} \inC$ be the unique dominant singularity of $F(z)$ and assume $F(z) = (1 - \zeta z)^{-\alpha}g(z)$ where $\alpha \inC$ and $g(z)$ is a complex-valued function that is analytic in the region $R = \left\{z \inC: |z| \leq \left|\frac{1}{\zeta}\right| \right\}$. Then, the following asymptotic estimation holds.
$$a(n) = \frac{\zeta^n g\left(\frac{1}{\zeta}\right)}{\Gamma(\alpha)} n^{\alpha - 1} + O\left(\zeta^n n^{\alpha - 2}\right)\,.$$
\end{corollary}
\begin{proof}
We expand $g(z)$ about $z = \frac{1}{\zeta}$ using Taylor's Theorem to get $g(z) = g\left(\frac{1}{\zeta}\right) + O\left(1 - \zeta z\right).$ Thus $$F\left(\frac{z}{\zeta}\right) = g\left(\frac{1}{\zeta}\right)(1 -  z)^{-\alpha} + O\left((1 - z)^{1-\alpha}\right) \quad\quad \text{as} \quad\quad z \rightarrow 1.$$ We then apply Theorem \ref{transfer lemma} and notice that $[z^n] F\left(\frac{z}{\zeta}\right) = \frac{1}{\zeta^n}[z^n]F(z)$ to achieve the desired result.

\end{proof}

%% file: Results_Sec_1.tex
\subsection{Generating functions for counting trees by leaves, internal nodes and root degree}
\label{Gen function count section}
Let $\mathcal{G}_n(d_0, d_1, r)$ be the set of plane trees on $n$ edges with $d_0$ leaves, $d_1$ internal nodes and root degree $r$. Let 

$$G(x,a,b) = \sum_{n = 0}^\infty\sum_{d_0 = 0}^\infty\sum_{d_1 = 0}^\infty\sum_{r = 0}^\infty |\mathcal{G}_n(d_0, d_1, r)|x^na^{d_0}b^{d_1},$$  
$$G_r(x,a,b) = \sum_{n = 0}^\infty\sum_{d_0 = 0}^\infty\sum_{d_1 = 0}^\infty |\mathcal{G}_n(d_0, d_1, r)|x^na^{d_0}b^{d_1}$$ and $$G(x,a,b,c) = \sum_{n = 0}^\infty\sum_{d_0 = 0}^\infty\sum_{d_1 = 0}^\infty \sum_{r = 0}^\infty |\mathcal{G}_n(d_0, d_1, r)|x^na^{d_0}b^{d_1}c^r.$$ As we will use it often from this point on, let $G^*(x,a,b) = G_1(x,a,b)$.

\begin{theorem} \label {G recurrence}
	The following recurrences hold. 
	\begin{equation} \label{primary recurrence G}
	G(x,a,b) = 1 + xG(x,a,b)^2 + (a-1)xG(x,a,b) + (b-1)xG(x,a,b)G^*(x,a,b),
	\end{equation}
	\begin{equation}\label{primary recurrence G with root}
	G(x,a,b,c) = \frac{1}{1-cG^*(x,a,b)},
	\end{equation}
	and
	\begin{equation}\label{primary recurrence G*}
	G^*(x,a,b) = 1- \frac{1}{G(x,a,b)}.
	\end{equation}
\end{theorem}
\begin{proof}
    To achieve (\ref{primary recurrence G}), we apply the decomposition of a plane tree into 2 subtrees given by $T = T_1 \ltimes T_2$. This adds a new edge. The only tree not accounted for by this process is $T^*$. Note that this process of adjoining trees never creates any new leaves or internal nodes in $T_1$. Notice that when $T_2 = T^*$, $T_2$ has an extra leaf which is not accounted for (since it would be at the root of $T_2$). Thus we re-weight these trees. (Hence $(a-1)xG(x,a,b)$.) Also notice that when $T_2$ has root degree 1, $T_2$ has an extra internal node which is not accounted for (since it would be at the root of $T_2$). Thus we re-weight these trees. (Hence $(b-1)xG(x,a,b)G^*(x,a,b)$.)
    
    To achieve (\ref{primary recurrence G with root}), we notice that a tree with root degree $r$ is equivalent to a sequence of $r$ trees with root degree 1 where we identify all the root vertices. This identification does not add or remove any edges, leaves or internal nodes. Thus the generating function for trees with root degree $r$, where we weight root degree, is $c^rG^*(x,a,b)^r$. Summing over all possible values for $r$, we get the desired expression.
    
    To achieve (\ref{primary recurrence G*}), we set $c$ to 1 in (\ref{primary recurrence G with root}) to ignore root degree, and rearrange the expression.
    
\end{proof}

\begin{corollary}\label{eval G func}
The following generating functions hold.
	\begin{equation}\label{G equation}
	G(x,a,b) = \frac{1 + (2-a-b)x - \sqrt{(1+(2-a-b)x)^2 - 4x(1 - (b-1)x)}}{2x} \end{equation} and 
	\begin{equation}\label{G^* equation}
	G^*(x,a,b) = \frac{1 + (a-b)x - \sqrt{(1+(2-a-b)x)^2 - 4x(1 - (b-1)x)}}{2(1-(b-1)x)}. \end{equation}
\end{corollary}
\begin{proof}
The result follows immediately by solving (\ref{primary recurrence G}) and (\ref{primary recurrence G*}) simultaneously.

\end{proof}

\begin{lemma}\label{dom sig lemma} For all $a, b \inR_{> 0}$, the dominant singularity of $G(x,a,b)$ and $G^*(x,a,b)$ occurs at $\rho = \rho(a, b) = a + b + 2\sqrt{a}$.
\end{lemma}
\begin{proof}
Fix $a, b > 0$. Let $\Psi = (1+(2-a-b)x)^2 - 4x(1 - (b-1)x)$ and let $\overline{\rho} = \overline{\rho}(a,b) = a + b - 2\sqrt{a}$. The roots of $\Psi$ are  $\frac{1}{\rho}$ and $\frac{1}{\overline{\rho}}$. Clearly, $0 < \frac{1}{\rho} < \frac{1}{\overline{\rho}}$. Notice that the only singularity caused by the numerator of (\ref{G^* equation}) and (\ref{G equation}) occurs when $\Psi = 0$. Thus the most significant of such singularities is $\frac{1}{\rho}$.

Notice that if there exists another singularity of $G(x,a,b)$, it must occur at $x = 0$ (caused by the denominator). We notice that when $x = 0$, the numerator of $(\ref{G equation})$ goes to 0. Thus, taking a Laurent expansion of $G(z,a,b)$ at $z = 0$, we get $$G(z,a,b) = \frac{1}{2z}\sum_{n \geq 1} d_nz^n = \sum_{n \geq 0} \frac{d_{n + 1}}{2}z^{n},$$ where $d_i$ are constants. Thus $G(z,a,b)$ is analytic at $z = 0$. Thus, the dominant singularity of $G(x,a,b)$ is $\frac{1}{\rho}$.

For $b = 1$, the denominator of $(\ref{G^* equation})$ is a constant, thus cannot cause another singularity. Assume $b \neq 1$. Notice that if there exists another singularity of $G^*(x,a,b)$, it must occur at $x = \frac{1}{b-1}$ (caused by the denominator). We notice that when $x = \frac{1}{b-1}$, the numerator of $(\ref{G^* equation})$ goes to $d'_0 = \frac{a - 1}{b-1} - \left| \frac{a-1}{b-1}\right|$. Notice that when the sign of $a-1$ and $b-1$ are the same (or $a = 1$), $d'_0 = 0$. Thus, taking a Laurent expansion of $G^*(z,a,b)$ at $z = \frac{1}{b-1}$, we get $$G^*(z,a,b) = \frac{1}{2(1 - (b-1)z)}\sum_{n \geq 0} d'_n\left(z - \frac{1}{b-1}\right)^n = \sum_{n \geq -1} \frac{d'_{n+1}}{2(1-b)}\left(z - \frac{1}{b-1}\right)^{n},$$ where $d'_i$ are constants. Notice that when $d'_0 = 0$, $G^*(z,a,b)$ is analytic at $z = \frac{1}{b-1}$. Thus, the dominant singularity of $G^*(x,a,b)$ is $\frac{1}{\rho}$. When $d'_0 \neq 0$, $G^*(z,a,b)$ has a singularity at $z = \frac{1}{b-1}$. For this to be the case, we must have that $a - 1$ and $b-1$ have the different signs. Notice that for $z = \frac{1}{b-1}$ to be the dominant singularity, $\rho^2 < (b-1)^2$, which implies that $a + 2b + 2\sqrt{a} < 1$. Since $a - 1$ and $b-1$ have different signs, one of $a$ and $b$ is at least 1, thus the inequality cannot hold. We thus conclude that the dominant singularity of $G^*(x,a,b)$ in this case is also $\frac{1}{\rho}$.

\end{proof}

\begin{corollary} \label{asymp normalizing const lemma}
Fix $\alpha, \beta \inR$. Let $\rho = e^{-\alpha} + e^{-\beta} + 2e^{-\frac{\alpha}{2}}$. The following estimate holds.
$$\normConst_{(n,\alpha,\beta,0)} = \frac{\sqrt{ e^{-\frac{\alpha}{2}} \rho }}{2\sqrt{\pi}} \cdot \rho^n \cdot n^{-\frac{3}{2}} + O\left(n^{-\frac{5}{2}}\right) $$
\end{corollary}
\begin{proof}
Let $\overline{\rho} = e^{-\alpha} + e^{-\beta} - 2e^{-\frac{\alpha}{2}}$. By definition, it should be clear that $\normConst_{(n,\alpha,\beta,0)} = [x^n]G(x,e^{-\alpha}, e^{-\beta})$. Thus, from Corollary \ref{eval G func}, for $n \geq 2$, $$\normConst_{(n-1,\alpha,\beta,0)} = [x^n]\left(-\frac{\sqrt{1 - \overline{\rho} x}}{2} \cdot \sqrt{1 - \rho x}\right).$$ By Lemma \ref{dom sig lemma}, for all $\alpha, \beta \inR$, $\frac{1}{\rho} < \frac{1}{\overline{\rho}}$, thus $\sqrt{1 - \overline{\rho} x}$ is analytic on the disk $R = \left\{z \inC: |z| \leq \frac{1}{\rho}\right\}$. We thus apply Corollary \ref{analytic combo corollary 1}, to see that 
\begin{eqnarray}
\normConst_{(n-1,\alpha,\beta,0)} &=& -\frac{\rho^n}{2\Gamma\left(-\frac{1}{2}\right)} \cdot \left( \sqrt{\frac{ 4e^{-\frac{\alpha}{2}}}{\rho}} \right)n^{-\frac{3}{2}} + O\left(n^{-\frac{5}{2}}\right) \nonumber\\
&=& \frac{\rho^n}{2\sqrt{\pi}} \cdot \sqrt{ \frac{ e^{-\frac{\alpha}{2}}}{\rho}}  \cdot n^{-\frac{3}{2}} + O\left(n^{-\frac{5}{2}}\right).
\end{eqnarray} 

We now set $n$ to $n + 1$ to get the desired result.

\end{proof}

\bigskip 

We will now extract the generating function $G_n(a,b)$ defined by $$G_n(a,b) = \sum_{d_0 = 0}^\infty\sum_{d_1 = 0}^\infty\sum_{r = 0}^\infty |\mathcal{G}_n(d_0, d_1, r)|a^{d_0}b^{d_1}$$ using the following technical lemma. We will defer the proof of the lemma. 

For a continuously differentiable function $F(x_1, \cdots, x_k)$ and $V$, a finite multiset  with elements $v_1, \cdots, v_m \in \{x_1, \cdots, x_k\}$, define $$\frac{\partial F(x_1, \cdots, x_k)}{\partial V} = \frac{\partial^m F(x_1, \cdots, x_k)}{\partial v_1 \cdots \partial v_m}.$$ For a set $V$, let $\text{Part}(V)$ be the set of unordered partitions on $V$, $(V_i)$. For a multiset, $V$, we define $\text{Part}(V)$ by distinguishing all elements of $V$, taking the partitions of the induced set, then removing the distinction from each of the repeated elements. Note that $\text{Part}(V)$ is itself a multiset. For example, $\text{Part}(\{x_1, x_1\}) = \{\{\{x_1\}, \{x_1\}\}, \{\{x_1, x_1\}\}\}$.

\begin{lemma} \label{deriv lemma}
	Let $F$ be a continuously differentiable function in $x$, $\Delta$ be a continuously differentiable function in $x_1, \cdots, x_k$ and $V$ be a non-empty finite set or multiset of the elements $x_1, \cdots, x_k$ with elements $v_1, \cdots, v_m$.
	
	\begin{equation}
	\frac{\partial F(\Delta)}{\partial V} = \sum_{(V_i) \in \text{Part}(V)} \frac{\partial \Delta}{\partial V_1} \cdots \frac{\partial \Delta}{\partial V_p} \cdot \left.\frac{\partial^p F(x)}{\partial x^p}\right|_{x = \Delta}
	\end{equation}
\end{lemma}

We now achieve the following expression for $G_n(a,b)$. 
We also compute the coefficients of the above expression explicitly via a combinatorial argument in Corollary \ref{count trees by leaves and internal}.

\begin{corollary} For $n \geq 2$,
	$$G_n(a,b) = \frac{1}{2^{n+1}} \sum_{0 \leq k \leq \frac{n+1}{2}} C_{n-k}\binom{n-k+1}{k} \cdot  \left(a+b\right)^{n-2k+1} \cdot \left(4a - (a+b)^2\right)^k.$$
\end{corollary}

\begin{proof}
We consider $xG(x,a,b)$. Notice that for $n \geq 2$, $$\frac{\partial^n xG(x,a,b)}{\partial x^n} = -\frac{1}{2} \cdot \frac{\partial \sqrt{\Delta}}{\partial V}$$ where $V$ is the multiset containing $n$ copies of $x$ and $\Delta = (1+(2-a-b)x)^2 - 4x(1 - (b-1)x)$. Thus we apply Lemma \ref{deriv lemma}. We notice that if $|V_i| > 2$, $\frac{\partial \Delta}{\partial V_i} = 0$. The number of elements in $\text{Part}(V)$ in which each part has size is at most 2 and there are $l$ parts of size $1$ and $k$ parts of size 2 is $\frac{n!}{2^k \cdot k! 	\cdot l!}.$ Thus, for $n \geq 2$,

\begin{eqnarray}\left.\frac{\partial^n xG(x,a,b)}{\partial x^n}\right|_{x =0} &=& \left.\frac{1}{2} \sum_{2k+l = n} \frac{n!}{2^k \cdot k! \cdot l!} \cdot  \left(\frac{\partial \Delta}{\partial x}\right)^l \cdot \left(\frac{\partial^2 \Delta}{\partial x^2}\right)^k \cdot \frac{\left(-\frac{1}{2}\right)^{k+l}\cdot (2(k+l)-3)!!}{\Delta^\frac{2(k+l) - 1}{2}}\right|_{x =0} \nonumber\\
&=& \frac{n!}{2^n} \sum_{0 \leq k \leq \frac{n}{2}} C_{n-k-1}\binom{n-k}{k} \cdot  \left(a+b\right)^{n-2k} \cdot \left(4a - (a+b)^2\right)^k.
\end{eqnarray}

Finally, notice that \begin{eqnarray}
\left.\frac{\partial^n xG(x,a,b)}{\partial x^n}\right|_{x =0} &=& n!\sum_{d_0 = 0}^\infty\sum_{d_1 = 0}^\infty\sum_{r = 0}^\infty |\mathcal{G}_{n-1}(d_0, d_1, r)|a^{d_0}b^{d_1}.
\end{eqnarray}

\end{proof}

%% file: Results_Sec_2.tex
\subsection{The Distributions of Simple Subtree Additive Properties}

\label{Distributions of Simple Subtree Additive Properties}

In this section, we will consider various additive properties. We will assume all the tolls in this section are of the form $(f, 0)$ since, from Example \ref{Example number of leaves and edges} and (\ref{additve prop linearity}), we see that for any $c \inZ_{\geq 0}$ and $T \in \setofplanetrees_{n}$, 
\begin{equation}
\mathcal{P}^{(f, c)}(T) = \mathcal{P}^{(f, 0)}(T) + c \cdot  \left(\mathcal{P}^{e}(T) + 1\right) = \mathcal{P}^{(f + c, 0)}(T) + c,
\end{equation} where $(f+c)(T) = f(T) + c$.

Let $\mathcal{P}$ be a non-negative integer valued additive property of plane trees. Let $\mathcal{P}^*$ be the subtree additive property induced by $\mathcal{P}$. Such a subtree additive property is simple as we have defined. The main result of this section is that if the toll function of $\mathcal{P}$ is bounded, the limiting distribution of $\mathcal{P^*}$ is determined by the limiting distribution of $\mathcal{P}$. Our primary result is stated as follows. 

\begin{theorem} \label{mainSimpleAdditveTheorem1}
	Let $\mathcal{P}_1$ and $\mathcal{P}_2$ be non-negative integer valued additive properties of plane trees with toll functions $f_1$ and $f_2$, respectively. Let the subtree additive properties induced by $\mathcal{P}_1$ and $\mathcal{P}_2$ be $\mathcal{P}^*_1$ and $\mathcal{P}^*_2$, respectively. Further assume that there exists $\zeta \inN$ such that for all $T \in \setofplanetrees_{\geq 0}$, $f_1(T) \leq  \zeta$ and $f_2(T) \leq  \zeta$. If, for all $m,n \inZ$, \begin{equation} \label{main condition on theorem sec 2}
	    \sum_{T \in \setofplanetrees_n} \mathcal{P}_1(T)^m =  \mu^m \cdot \sum_{T \in \setofplanetrees_n} \mathcal{P}_2(T)^m + O\left(n^{\frac{2m-4}{2}}4^n\right)
	\end{equation} where $\mu \inR$ is a constant, then as $n \rightarrow  \infty$, $$\frac{\mathcal{P}^*_1(\setofplanetrees_n)}{\sqrt{n^3}} \overset{d}{\longleftrightarrow} \mu \cdot \frac{\mathcal{P}^*_2(\setofplanetrees_n)}{\sqrt{n^3}}.$$
\end{theorem}

\bigskip

Let $\mathcal{P}$ be a non-negative integer valued property of plane trees that is additive with toll function $f$. Let $\mathcal{F}(n, m)$ be the set of trees, $T$, on $n$ edges such that $\mathcal{P}(T) = m$. Let $$F(x,p) = \sum_{n \geq 0} \sum_{m \geq 0} |\mathcal{F}(n, m)| x^np^m.$$ We may also refer to $F(x,p)$ by $F_{\mathcal{P}}(x,p)$ where the property we are referring to is unclear.
We now let $\mathcal{H}^v(n, m)$ be the set of trees, $T$, on $n$ edges such that $\mathcal{P}(T) = m$ and $f(T) = v$. Let 
$$H^v(x,p) = \sum_{n \geq 0} \sum_{m \geq 0} |\mathcal{H}^v(n, m)| x^np^m.$$ 

\begin{lemma} \label{F(x,p) recurrence}
For any fixed non-negative integer valued additive property of plane trees, the following recurrence holds.

\begin{equation}
F(x,p) = 1 + xF(x,p)\sum_{v\geq 0} p^v H^v(x,p)
\end{equation} 
\end{lemma}
\begin{proof}
We apply the decomposition of a plane tree into 2 subtrees given by $T = T_1 \ltimes T_2$. This adds a new edge. Recall that $\mathcal{P}(T) = \mathcal{P}(T_1) + \mathcal{P}(T_2) + f(T_2)$. Thus, the tree $T$ gains an extra $p^{f(T_2)}$ in the weighting. When $T_2 \in \mathcal{H}^v(n, m)$, $T$ gains an extra $p^{v}$ in the weighting. The only tree not accounted for by this process is $T^*$.

\end{proof}

\begin{lemma} \label{bound on F(x,p) lemma}
    Let $\mathcal{P}$ be a non-zero non-negative integer valued additive property with toll function $f$ such that $f(T) \leq \zeta$ for all $T \in \setofplanetrees_{\geq 0}$ where $f$ achieve $\zeta$. For all $n, m \geq 0$, $$\sum_{T \in \setofplanetrees_n} \mathcal{P}(T)^m = \Theta\left(n^{\frac{2m-3}{2}}4^n\right).$$ 
\end{lemma}
\begin{proof} 
Fix $m \geq 0$. Assume $\zeta > 0$ (otherwise $\mathcal{P}(T) = 0$ for all $T \in \setofplanetrees_{\geq 0}$). We first note that
\begin{eqnarray}
\frac{\partial F_{\mathcal{P}}(x,1)}{\partial p^m} &=& \sum_{n \geq 0}x^n \sum_{T \in \setofplanetrees_n} \mathcal{P}(T)(\mathcal{P}(T) - 1) \cdots (\mathcal{P}(T) - m + 1) \nonumber \\
&=& \sum_{n \geq 0}x^n \sum_{T \in \setofplanetrees_n} \mathcal{P}(T)^m +  O\left(\mathcal{P}(T)^{m-1}\right) \label{eq 10}
\end{eqnarray}

Consider $\mathcal{P}_1$, the additive property with toll function $f_1(T) = \zeta$ for all $T \in \setofplanetrees_{\geq 0}$. Let $T' \in \setofplanetrees_N$ be a tree such that $f(T') = \zeta$. Consider $\mathcal{P}_2$, the additive property with toll function $f_2(T') = \zeta$ and $f_2(T) = 0$ for all other $T \in \setofplanetrees_{\geq 0}$.
It should be clear that for any tree $T \in \setofplanetrees_{\geq 0}$, 
\begin{equation} \label{eq 11}
0 \leq \mathcal{P}_2(T) \leq \mathcal{P}(T) \leq \mathcal{P}_1(T).
\end{equation} 

We now see that $F_{\mathcal{P}_1}(x,p) = 1 + xp^\zeta F_{\mathcal{P}_1}(x,p)^2$ and $F_{\mathcal{P}_2}(x,p) = 1 + xF_{\mathcal{P}_2}(x,p)^2 + (p^\zeta - 1)x^N$, thus $$F_{\mathcal{P}_1}(x,p) = \frac{1 - \sqrt{1 - 4xp^\zeta}}{2xp^\zeta} \quad\quad \text{and} \quad\quad F_{\mathcal{P}_2}(x,p) = \frac{1 - \sqrt{1 - 4x(1 +(p^\zeta - 1)x^N) }}{2x}.$$ 
Thus, for $m \geq 1$, $$\frac{\partial F_{\mathcal{P}_1}(x,1)}{\partial p^m} = c_1 \cdot (1 - 4x)^{\frac{1-2m}{2}} + O\left((1 - 4x)^{\frac{3 - 2m}{2}}\right)$$
and
$$\frac{\partial F_{\mathcal{P}_2}(x,1)}{\partial p^m} = c_2 \cdot (1 - 4x)^{\frac{1 - 2m}{2}} + O\left((1 - 4x)^{\frac{3 - 2m}{2}}\right),$$ where $c_1, c_2$ are constants (that depend on $m$). We now apply Corollary \ref{analytic combo corollary 1} to see that $$[x^n]\frac{\partial F_{\mathcal{P}_1}(x,1)}{\partial p^m} = c'_1 \cdot n^{\frac{2m - 3}{2}}  4^n (1 + o(1))\quad\quad \text{and} \quad\quad [x^n]\frac{\partial F_{\mathcal{P}_2}(x,1)}{\partial p^m} = c'_2 \cdot n^{\frac{2m - 3}{2}}  4^n (1 + o(1))$$ where $c'_1, c'_2$ are constants. From (\ref{eq 10}) and (\ref{eq 11}),  we thus see that  for all $n \geq 0$, $$c'_2 \cdot n^{\frac{2m - 3}{2}}  4^n (1 + o(1)) = \sum_{T \in \setofplanetrees_n} \mathcal{P}_2(T)^m \leq \sum_{T \in \setofplanetrees_n} \mathcal{P}(T)^m \leq \sum_{T \in \setofplanetrees_n} \mathcal{P}_1(T)^m = c'_1 \cdot n^{\frac{2m - 3}{2}}  4^n (1 + o(1)).$$

\end{proof}

\begin{lemma}
    Let $\mathcal{P}$ be a non-zero non-negative integer valued additive property with toll function $f$ such that $f(T) \leq \zeta$ for all $T \in \setofplanetrees_{\geq 0}$ where $f$ achieve $\zeta$. As $n \rightarrow \infty$, the limiting distribution of $\frac{\mathcal{P}(\setofplanetrees_n)}{n}$ is uniquely determined by its moments.
\end{lemma}

\begin{proof}
    Consider $\mathcal{P}_1$, the additive property with toll function $f$ such that $f(T) = \zeta$ for all $T \in\setofplanetrees_{\geq 0}$. We also see that for all $T \in\setofplanetrees_{\geq 0}$, $\mathcal{P}(T) \leq \mathcal{P}_1(T)$. From Example \ref{Example number of leaves and edges} and (\ref{additve prop linearity}), for $T \in\setofplanetrees_{n}$, $\mathcal{P}_1(T) = \zeta n$. Thus $\E\left[\left(\frac{\mathcal{P}(\setofplanetrees_n)}{n}\right)^k\right] \leq \E\left[\left(\frac{\mathcal{P}_1(\setofplanetrees_n)}{n}\right)^k\right] = \zeta^k$. The result follows immediately by the Carleman's condition (Theorem \ref{Carleman's condition}).
    
\end{proof}

\bigskip

The above result tells us that the condition in (\ref{main condition on theorem sec 2}) implies that $\frac{\mathcal{P}_1(\setofplanetrees_n)}{n} \overset{d}{\longleftrightarrow} \mu \cdot \frac{\mathcal{P}_2(\setofplanetrees_n)}{n}$. Thus Theorem \ref{mainSimpleAdditveTheorem1} is equivalent to the following corollary.

\begin{corollary} \label{mainSimpleAdditveTheorem2}
	For non-negative integer valued additive properties, $\mathcal{P}_1$ and $\mathcal{P}_2$, with toll functions $f_1$ and $f_2$, respectively, such that there exists $\zeta \inN$ such that for any $T \in \setofplanetrees_{\geq 0}$, $f_1(T) \leq  \zeta$ and $f_2(T) \leq  \zeta$, 
	$$\frac{\mathcal{P}_1(\setofplanetrees_n)}{n} \overset{d}{\longleftrightarrow} \mu \cdot\frac{\mathcal{P}_2(\setofplanetrees_n)}{n} \Rightarrow \frac{\mathcal{P}^*_1(\setofplanetrees_n)}{\sqrt{n^3}} \overset{d}{\longleftrightarrow} \mu \cdot\frac{\mathcal{P}^*_2(\setofplanetrees_n)}{\sqrt{n^3}}$$ where $\mu \inR$ is constant and $\mathcal{P}^*_1$ and $\mathcal{P}^*_2$ are the induced subtree additive properties.\\
\end{corollary}

\bigskip

Let $\mathcal{P^*}$ be the subtree additive property induced by $\mathcal{P}$. Let $\setofplanetrees_ {n, m}$ be the set of plane trees, $T$, on $n$ edges such that $\mathcal{P}(T) = m$. We define 
\begin{eqnarray} 
M_{k,n, m} &=& \sum_{T \in \setofplanetrees_{n, m}} \mathcal{P^*}(T)^k = \sum_{T \in \setofplanetrees_{n, m}}\left(\sum_{v \in \overline{\mathcal{V}}(T)} \mathcal{P}(T_v)\right)^k \nonumber \\
&=& \sum_{T \in \setofplanetrees_{n, m}}\sum_{(v_1, \cdots, v_k) \in \overline{\mathcal{V}}(T)^k} \mathcal{P}(T_{v_1}) \cdots \mathcal{P}(T_{v_k}) \label{M_{k,n, m} sum}
\end{eqnarray} 
and let 
\begin{equation} 
M_{k}(x, p) = \sum_{n, m \geq 0} M_{k,n, m}x^np^m
\end{equation} 
and $M_k(x) = M_{k}(x, 1)$. Note that $[x^n]M_k(x) = \sum_{T \in \setofplanetrees_n }\mathcal{P^*}(T)^k = \E\left[\mathcal{P^*}(\setofplanetrees_n)^k\right] \cdot C_n$, where $C_n$ is the $n$th Catalan number. \\

Fix $t_1, \cdots, t_k \inZ_{\geq 0}$ and $n,m \geq 0$. We consider tuples of $(v_1, \cdots, v_k) \in \overline{\mathcal{V}}(T)$ for some $T \in \setofplanetrees_{n,m}$ where $\mathcal{P}(T_{v_i}) = t_i$. Notice that the contribution to $M_{k,n, m}$ of any such tuples is $\prod_{i=1}^kt_i$. Also notice that the number of such tuples is $$\sum_{T \in \setofplanetrees_{n, m}}\prod_{i=1}^k W(T, t_i)$$
where $W(T, t) = \left|\left\{v \in \overline{\mathcal{V}}(T): \mathcal{P}(T_v) = t \right\}\right|$. Let $\mathcal{F}_k(n, m, \vec{t}, \vec{s})$ be the set of trees, $T$, on $n$ edges such that $\mathcal{P}(T) = m$ and $W(T, t_i) = s_i$ where $\vec{s} = (s_1, \cdots, s_k) \in \Z_{\geq 0}^k$ and $\vec{t} = (t_1, \cdots, t_k) \in \Z_{\geq 0}^k$. We now let 
$$F_k(\vec{t}|x, \vec{y}|m) = \sum_{n \geq 0} \sum_{s_1 \geq 0} \cdots \sum_{s_k \geq 0}|\mathcal{F}_k(n, m, \vec{t}, \vec{s})| x^n\prod_{i=1}^k y_i^{s_i}$$  and 
$$F_k(\vec{t}|x, p, \vec{y}) = \sum_{m \geq 0} F_k(\vec{t}|x, \vec{y}|m) p^m.$$ 

Note that $$\frac{\partial F_k(\vec{t}|x, p, \vec{1})}{\partial y_1 \cdots \partial y_k} = \sum_{n \geq 0} \sum_{m \geq 0} \sum_{s_1 \geq 0} \cdots \sum_{s_k \geq 0}|\mathcal{F}_k(n, m, \vec{t}, \vec{s})| x^np^m\prod_{i=1}^k {s_i}$$ where $\vec{1}$ the vector of appropriate size for the context consisting of all 1s. 

For $V = \{v_1, \cdots, v_l\} \subset \N$ and $F$, a function, we define $$\frac{\partial F}{\partial y(V)} = \frac{\partial F}{\partial y_{v_1} \cdots \partial y_{v_l}}.$$ We define $\frac{\partial F}{\partial z(V)}$ similarly. We also denote the set of integers from 1 to $k$ by $[k]$. We now show the following lemma that will be integral to the rest of our analysis. 

\begin{lemma} For $k \inN$ and $V = \{v_1, \cdots, v_l\} \subset [k]$ such that $|V| = l \leq k$, the following holds.
\begin{equation}
M_l(x, p) = \left.\frac{\partial}{\partial z(V)}\left(\sum_{t_{v_1} \geq 0} \cdots \sum_{t_{v_l} \geq 0}\frac{\partial F_k(\vec{t}|x, p, \vec{1})}{\partial y(V)} \prod_{i=1}^lz^{t_{v_i}}_{v_i}\right)\right|_{\vec{z} =\vec{1}}.
\end{equation}
\label{M_l(x, p) lemma} 
\end{lemma}
\begin{proof} Let $W = [k]-V = \{w_1, \cdots, w_{k-l}\}$. We first see that \begin{eqnarray}
\frac{\partial F_k(\vec{t}|x, p, \vec{1})}{\partial y(V)} &=& \sum_{n \geq 0} \sum_{m \geq 0} \sum_{s_1 \geq 0} \cdots \sum_{s_k \geq 0}|\mathcal{F}_k(n, m, \vec{t}, \vec{s})| x^np^m\prod_{i=1}^l {s_{v_i}} \nonumber\\
&=& \sum_{n \geq 0} \sum_{m \geq 0} \sum_{s_{v_1} \geq 0} \cdots \sum_{s_{v_l}\geq 0} x^np^m\prod_{i=1}^l {s_{v_i}} \sum_{s_{w_1} \geq 0} \cdots \sum_{s_{w_{k-l}} \geq 0} |\mathcal{F}_k(n, m, \vec{t}, \vec{s})| \label{equat 1}\\
&=& \sum_{n \geq 0} \sum_{m \geq 0} \sum_{s_{v_1} \geq 0} \cdots \sum_{s_{v_l}\geq 0}|\mathcal{F}_l(n, m, \vec{t}, \vec{s})| x^np^m  \prod_{i=1}^l {s_{v_i}} \label{equat 2}\\
&=& \frac{\partial F_l(\vec{t}|x, p, \vec{1})}{\partial y(V)} \label{equat 3}
\end{eqnarray}

We go from (\ref{equat 1}) to (\ref{equat 2}) by noticing that when we sum $|\mathcal{F}_k(n, m, \vec{t}, \vec{s})|$ over $s_{w_{1}} \cdots s_{w_{k-l}}$, we lose our dependency on $t_{w_1}, \cdots, t_{w_{k-l}} $. We go from (\ref{equat 2}) to (\ref{equat 3}) by relabeling the indices from $v_1, \cdots, v_l$ to $1,\cdots, l$. Thus to prove the lemma, we need only consider the case when $V = [k]$. We now see that 
\begin{multline}
\left.\frac{\partial}{\partial z([k])}\left(\sum_{t_{v_1} \geq 0} \cdots \sum_{t_{v_l} \geq] 0}\frac{F_k(\vec{t}|x, p, \vec{1})}{\partial y([k])} \prod_{i=1}^lz^{t_{v_i}}_{v_i}\right)\right|_{\vec{z} =\vec{1}}\\
= \sum_{n \geq 0} \sum_{m \geq 0} \sum_{t_1 \geq 0}\sum_{s_1 \geq 0} \cdots \sum_{t_k \geq 0}\sum_{s_k \geq 0}|\mathcal{F}_k(n, m, \vec{t}, \vec{s})| x^np^m\prod_{i=1}^k s_it_i.
\end{multline}

For fixed $\vec{t}, \vec{s}$, we notice that trees in $\mathcal{F}_k(n, m, \vec{t}, \vec{s})$ are precisely the trees from which we can get $(v_1, \cdots, v_k) \in \overline{\mathcal{V}}(T)$ for some $T \in \setofplanetrees_{n,m}$ where $\mathcal{P}(T_{v_i}) = t_i$. The contribution of each tuple to $M_{k,n, m}$ is  $\prod_{i=1}^k t_i$. The number of tuples $(v_1, \cdots, v_k)$ that can be achieved from each tree $T \in \mathcal{F}_k(n, m, \vec{t}, \vec{s})$ is $\prod_{i=1}^k s_i$. Thus the total (weighted) contribution from tuples of the above form is $$|\mathcal{F}_k(n, m, \vec{t}, \vec{s})|x^np^m\prod_{i=1}^k s_it_i.$$ Hence,  summing over $\vec{t}, \vec{s} \inZ^k_{\geq 0}$ and $n,m \inZ_{\geq 0}$, we get $M_k(x,p)$, proving the result.

\end{proof}
\bigskip

Let $\mathcal{H}_k^v(n, m, \vec{t}, \vec{s}) = \mathcal{H}^v(n, m) \cap \mathcal{F}_k(n, m, \vec{t}, \vec{s})$, 
\begin{equation}
H^v_k(\vec{t}|x, \vec{y}|m) = \sum_{n \geq 0} \sum_{s_1 \geq 0} \cdots \sum_{s_k \geq 0}|\mathcal{H}_k^v(n, m, \vec{t}, \vec{s})| x^n\prod_{i=1}^k y_i^{s_i},
\end{equation} 
\begin{equation}
H^v_k(\vec{t}|x, p, \vec{y}) =  \sum_{m \geq 0} H^v_k(\vec{t}|x, \vec{y}|m) p^m
\end{equation} and 
\begin{equation}
J^v_k(x, p) = \left.\frac{\partial}{\partial z([k])}\left(\sum_{t_{1} \geq 0} \cdots \sum_{t_{k} \geq 0}\frac{ H^v_k(\vec{t}|x, p, \vec{1})}{\partial y([k])} \prod_{i=1}^kz^{t_{i}}_{i}\right)\right|_{\vec{z} =\vec{1}}.
\end{equation}

Using a similar argument to Lemma \ref{M_l(x, p) lemma}, we get that for $V = \{v_1, \cdots, v_l\} \subset [k]$ such that $|V| = l \leq k$,
\begin{equation} \label{J^v_l def}
J^v_l(x, p) = \left.\frac{\partial}{\partial z(V)}\left(\sum_{t_{v_1} \geq 0} \cdots \sum_{t_{v_l} \geq 0}\frac{H^v_k(\vec{t}|x, p, \vec{1})}{\partial y(V)} \prod_{i=1}^lz^{t_{v_i}}_{v_i}\right)\right|_{\vec{z} =\vec{1}}.
\end{equation}

Note that \begin{equation} \label{equat 5}
\sum_{v \geq 0} H^v_k(\vec{t}|x, p, \vec{y}) =  F_k(\vec{t}|x, p, \vec{y}) \quad\quad\quad\quad \text{and} \quad\quad\quad\quad \sum_{v \geq 0} J^v_k(x, p) =  M_k(x, p).
\end{equation} 

We define a partition of a set to be a set of disjoint subsets of the original set whose union is the original set. We call these subsets parts. We denote a partition of $S$ into $\lambda$ parts by $(S_i)_\lambda$, where $S_1, \cdots, S_\lambda$ are the parts of the partition. We say a partition of $S$, $(S_i)_\lambda$, refines another partition of $S$, $(S'_i)_\mu$ if for any $S_i$, there is a $S'_j$ such that $S_i \subset  S'_j$. We denote this by $(S_i)_\lambda \subset (S'_i)_\mu$.

Let $\vec{t} \in \Z_{\geq 0}^k$. Notice that $\vec{t}$ induces a partition of $S = [k]$ as follows. Let $(S_i)^{\vec{t}}_\lambda$ be such that the numbers $i, j$ are in the same part if and only if $t_i = t_j$. For a fixed $(S_i)^{\vec{t}}_\lambda$, let $t_i^*=t_j$ such that $j \in S_i$.

\begin{lemma} The following recurrence holds.
\begin{eqnarray} \label{Full F_k recurrence}
F_k(\vec{t}|x, p, \vec{y}) &=& 1 + x\sum_{v\geq 0} p^v F_k(\vec{t}|x, p, \vec{y})H_k^v(\vec{t}|x,p, \vec{y}) \nonumber \\
 && \quad\quad\quad\quad + xF_k(\vec{t}|x, p, \vec{y})\sum_{v\geq 0} \sum_{i = 1}^\lambda \left(\prod_{j \in S_i} y_j - 1\right) \cdot p^{t_i^*} \cdot p^v \cdot H_k^v(\vec{t}|x, \vec{y}|t_i^*), \label{full expansion sum 1}
\end{eqnarray} where the $S_i$ are the parts in $(S_i)^{\vec{t}}_\lambda$.
    
\end{lemma}
\begin{proof}
We apply the decomposition of a plane tree into 2 subtrees given by $T = T_1 \ltimes T_2$. By similar argument to Lemma \ref{F(x,p) recurrence}, we properly weight with respect to $n$ and $m$. To properly weight with respect to $\vec{y}$, we notice that the number of non-root vertices $v$ with $\mathcal{P}(T_v) = t$, for some $t$, in $T$ is the sum of the number of such vertices in $T_1$ and $T_2$. Additionally, when $\mathcal{P}(T_2) = t$, we get an extra such vertex (the root of $T_2$). When $\mathcal{P}(T_2) = t^*_i$, we get an extra vertex, $v$, where $\mathcal{P}(T_v) = t_{j} = t^*_i$, where $j \in S_i$. Notice that these are the trees counted by $\sum_{v \geq 0}H_k^v(\vec{t}|x, \vec{y}|t_i^*)$. These trees should get an extra $\prod_{j \in S_i} y_j$ in the weighting. We however notice that the weight with respect to $p$ of such trees is $p^{t_i^*}$. Trees counted by $H_k^v(\vec{t}|x, \vec{y}|t_i^*)$ also get an extra $p^v$ (to properly weight the entire tree with respect to $p$).

\end{proof}

Differentiating both sides of (\ref{Full F_k recurrence}), we see that

\begin{multline} \label{full expansion sum 2}
\frac{\partial F_k(\vec{t}|x, p, \vec{1})}{\partial y([k])}  = x\sum_{V \subset [k]} \frac{\partial F_k(\vec{t}|x, p, \vec{1})}{\partial y([k]-V)} \sum_{v \geq 0} p^v \cdot \frac{\partial H_k^v(\vec{t}|x,p, \vec{1})}{\partial y(V)}\\
+ x\sum_{V \subset [k]} \frac{\partial F_k(\vec{t}|x, p, \vec{1})}{\partial y([k]-V)}  \cdot \left. \frac{\partial}{\partial y(V)}\left(\sum_{v \geq 0}\sum_{i = 1}^\lambda \left(\prod_{j \in S_i} y_j - 1\right) \cdot p^{t_i^*} \cdot p^v \cdot H_k^v(\vec{t}|x, \vec{y}|t_i^*)\right)\right|_{\vec{y} = \vec{1}}.
\end{multline}

For $V = \{v_1, \cdots, v_l\} \subset [k]$, we let $\vec{t}(V) = (t_{v_1}, \cdots, t_{v_{l}})$. Using Lemma \ref{M_l(x, p) lemma}, we simplify 
$$\Phi_k^{(1)} = \frac{\partial}{\partial z([k])}\sum_{\vec{t} \inZ_{\geq 0}^k}\left(\sum_{V \subset [k]} \frac{\partial F_k(\vec{t}|x, p, \vec{1})}{\partial y([k]-V)} \sum_{v \geq 0} p^v \cdot \frac{\partial H_k^v(\vec{t}|x,p, \vec{1})}{\partial y(V)}\right) \prod_{i \in [k]}z_i^{t_i}$$ as follows.
\begin{eqnarray}
\Phi_k^{(1)} &=&  \sum_{V \subset [k]}\frac{\partial}{\partial z(V)}\left(\sum_{\vec{t}(V) \inZ_{\geq 0}^k} \frac{\partial F_k(\vec{t}|x, p, \vec{1})}{\partial y(V)}\prod_{i \in V}z_i^{t_i}\right) \times \nonumber\\
&& \quad\quad\quad \sum_{v \geq 0} p^v \cdot \frac{\partial}{\partial z([k]-V)}\left(\sum_{\vec{t}([k]-V) \inZ_{\geq 0}^k}\frac{\partial H_k^v(\vec{t}|x,p, \vec{1})}{\partial y([k]-V)} \prod_{i \in [k]-V}z_i^{t_i}\right) \\
&=&  \sum_{V \subset [k]}M_{|V|}(x,p) \cdot \sum_{v \geq 0} p^v \cdot J_{k - |V|}^v(x,p)
\end{eqnarray} where we set $z_i$ to 1. To simplify 
\begin{multline*}
    \Phi_k^{(2)} = \frac{\partial}{\partial z([k])}\sum_{\vec{t} \inZ_{\geq 0}^k}\left(\sum_{V \subset [k]}  \frac{\partial F_k(\vec{t}|x, p, \vec{y})}{\partial y([k]-V)} \right. \times 
    \\ \left.\frac{\partial}{\partial y(V)}\left(\sum_{v \geq 0} \sum_{i = 1}^\lambda \left(\prod_{j \in S_i} y_j - 1\right) \cdot p^{t_i^*} \cdot p^v \cdot H_k^v(\vec{t}|x, \vec{y}|t_i^*)\right) \right) \prod_{i \in V}z_i^{t_i},
\end{multline*} we will use Lemma \ref{M_l(x, p) lemma} and the following lemma (whose proof is deferred to the Appendix). 

\begin{lemma}
	\label{lemma expansion}
	Let $\emptyset \neq W = \{w_1, \cdots, w_m\}$ for some $m \inN$. For any $x_{w_1}, \cdots, x_{w_m}$, \begin{equation}
	\label{lemma expansion 1}
	\prod_{i=1}^m x_{w_{i}}  - 1 = \sum_{V = \{v_1, \cdots, v_l\} \subset W} (x_{v_1} - 1) \cdots (x_{v_l} - 1)
	\end{equation}
\end{lemma}

\bigskip

For $V \subset [k]$ and $\vec{t} \in \Z^k_{\geq 0}$, we let $$\Phi(\vec{t}|V) = \frac{\partial}{\partial y(V)}\left(\sum_{v \geq 0}\sum_{i = 1}^\lambda \left(\prod_{j \in S_i} y_j - 1\right) \cdot p^{t_i^*} \cdot p^v \cdot H_k^v(\vec{t}|x, \vec{y}|t_i^*)\right).$$ 
We apply Lemma \ref{lemma expansion} to get
\begin{equation} \label{eq 2}
\Phi(\vec{t}|V) = \sum_{v \geq 0} p^v \cdot \frac{\partial}{\partial y(V)}\left(\sum_{i = 1}^\lambda  \sum_{U = \{u_1, \cdots, u_q\} \subset S_i} (y_{u_1} - 1) \cdots (y_{u_q} - 1) \cdot p^{t_i^*}  \cdot H_k^v(\vec{t}|x, \vec{y}|t_i^*)\right).
\end{equation}

Our goal is to take the sum
$$\sum_{\vec{t}(V) \in \Z_{\geq 0}^{|V|}}\Phi(\vec{t}|V) \prod_{i \in V}z_i^{t_i}.$$
For some $U = \{u_1, \cdots, u_q\}  \subset [k]$, consider $(S'_i)$, the partition of $[k]$ where the elements $u_1, \cdots, u_q$ are in the same part and all other elements are in their own parts. Notice that 
\begin{equation} \label{equat 4}
(y_{u_1} - 1) \cdots (y_{u_q} - 1) \cdot p^{t_i^*}  \cdot H_k^v(\vec{t}|x, \vec{y}|t_i^*)
\end{equation} 
is a term in the bracket of (\ref{eq 2}) if and only if the partition induced by $\vec{t}$ is such that $(S'_i)$ is a refinement of $(S_i)_{\lambda}^{\vec{t}}$, ie. $(S_i) \subset (S_i)_{\lambda}^{\vec{t}}$. 

Assume $V$ is such that $U \subset V$ (otherwise the term we are considering will vanish when we set $y_i$ to 1). Let $V-U = \{j_1, \cdots, j_l\}$. We now consider the sum
\begin{eqnarray}
\Phi(V, U) &=&  \sum_{v \geq 0} p^v \cdot  \frac{\partial }{\partial y({V})} \sum_{t_{j_1}  \geq 0} \cdots \sum_{t_{j_l}  \geq 0}(y_{u_1} - 1) \cdots (y_{u_q} - 1) \times\\ 
&& \quad \quad \quad \quad \sum_{t^*_{i}  \geq 0} H_k^v(\vec{t}|x, \vec{y}|t_i^*) \cdot p^{t^*_{i}} (z_{u_1} \cdots z_{u_q})^{t^*_{i}} \prod_{i = 1}^lz^{t_{j_i}}_{j_i} \nonumber\\
&=&  \sum_{v \geq 0} p^v \cdot \sum_{t_{j_1}  \geq 0} \cdots \sum_{t_{j_l}  \geq 0} \frac{\partial H^v_{k}(\vec{t}|x, p z_{u_1} \cdots z_{u_q},\vec{y})}{\partial y({V-U})}  \prod_{i = 1}^lz^{t_{j_i}}_{j_i}.
\end{eqnarray}
We see that as we sum up $\Phi(\vec{t}|V) \prod_{i \in V}z^{t_{i}}_{i}$ over $\vec{t}$, we actually take the sum above for all $U \subset V$. Thus \begin{eqnarray}
\left.\frac{\partial}{\partial z(V)}\sum_{\vec{t}(V) \inZ_{\geq 0}^{|V|}} \Phi(\vec{t}|V) \prod_{i \in V}z^{t_{i}}_{i}\right|_{\vec{z} = \vec{1}} &=& \left.\sum_{\emptyset \neq U \subset V} \frac{\partial \Phi(V, U)}{\partial z(V)}\right|_{\vec{z} = \vec{1}}\\
&=& \left.\sum_{\emptyset \neq U \subset V} \sum_{v \geq 0} p^v \cdot \frac{\partial J^v_{|V|-|U|}(x, p z_{u_1} \cdots z_{u_q})}{\partial z(U)}\right|_{\vec{z} = \vec{1}}\\
&=& \sum_{\emptyset \neq U \subset V} \sum_{v \geq 0} p^v \cdot \mathcal{D}_p^{|U|} J^v_{|V|-|U|}(x, p)
\end{eqnarray} where for $F$, a function in some variables including $p$, we define the operator $\mathcal{D}_p^m$ recursively as 
\begin{equation} \label{D_p definition}
    \mathcal{D}_p^mF = p \cdot \frac{\partial \mathcal{D}_p^{m-1} F}{\partial p} \quad\quad\quad\quad \text{ and } \quad\quad\quad\quad \mathcal{D}_p^0 F = F.
\end{equation}

Thus, we simplify $\Phi_k^{(2)} $ as follows.
\begin{eqnarray}
\Phi_k^{(2)} &=&  \sum_{V \subset [k]}\frac{\partial}{\partial z([k]-V)}\left(\sum_{\vec{t}([k]-V) \inZ_{\geq 0}^{k-|V|}} \frac{\partial F_k(\vec{t}|x, p, \vec{1})}{\partial y([k]-V)}\prod_{i \in [k]-V}z_i^{t_i}\right) \times \nonumber\\
&& \quad\quad\quad\quad  \frac{\partial}{\partial z(V)}\sum_{\vec{t}(V) \inZ_{\geq 0}^{|V|}} \Phi_V(\vec{t}) \prod_{i \in V}z^{t_{i}}_{i} \\
&=&  \sum_{V \subset [k]}M_{k-|V|}(x,p) \cdot \sum_{\emptyset \neq U \subset V} \sum_{v \geq 0} p^v \cdot \mathcal{D}_p^{|U|} J^v_{|V|-|U|}(x, p)
\end{eqnarray} where we set $z_i$ to 1. 

We will denote $\left.\mathcal{D}_p^{m}\left(F(x,p)\right)\right|_{p=1}$ simply as $\mathcal{D}_p^{|U|}F(x)$.

\begin{theorem} For $k \geq 1$, the following recurrence holds.
\begin{equation}
M_k(x,p) = \sum_{U \subset V \subset [k]}M_{k-|V|}(x,p) \cdot \sum_{v \geq 0} p^v \cdot \mathcal{D}_p^{|U|} J^v_{|V|-|U|}(x, p).
\end{equation}
Furthermore, for $k \geq 1$ and $m \geq 0$, let $$S(k, m) = \{(a, b, c, d) \inZ^4: 0 \leq a \leq b \leq m, 0 \leq c \leq d \leq k\}$$ and $$S'(k, m) = S(k, m) - \{(0,0,0,0), (0,m,0,k)\}.$$ The following recurrence also holds.
\begin{equation}
\mathcal{D}^m_p M_k(x)  = \frac{x}{\sqrt{1-4x}} \cdot \sum_{(a, b, c, d) \in S'(k, m)} \binom{k}{d}\binom{d}{c}\sum_{v \geq 0} v^a \cdot \mathcal{D}^{m-b}_pM_{k-d}(x) \cdot  \mathcal{D}_p^{c + b - a} J^v_{d-c}(x). \label{equation 10} 
\end{equation}
\end{theorem}
\begin{proof}

From $(\ref{full expansion sum 2})$, we see that, for $k \geq 1$,
\begin{eqnarray}
M_k(x,p) &=&  x\Phi_k^{(1)} + x\Phi_k^{(2)} \nonumber\\
&=& \sum_{V \subset [k]}M_{k-|V|}(x,p) \cdot \sum_{\emptyset \neq U \subset V} \sum_{v \geq 0} p^v \cdot \mathcal{D}_p^{|U|} J^v_{|V|-|U|}(x, p) \nonumber\\
&& \quad\quad\quad\quad + \sum_{V \subset [k]} \sum_{v \geq 0} p^v \cdot M_{k - |V|}(x,p)   \cdot J_{|V|}^v(x,p) \nonumber\\
&=& \sum_{U \subset V \subset [k]}M_{k-|V|}(x,p) \cdot \sum_{v \geq 0} p^v \cdot \mathcal{D}_p^{|U|} J^v_{|V|-|U|}(x, p). \nonumber
\end{eqnarray}

Differentiating the above expression, we see that
\begin{eqnarray}
\mathcal{D}^m_p M_k(x,p) &=& x\sum_{U \subset V \subset [k]}\sum_{b=0}^m \sum_{a=0}^b \sum_{v \geq 0} \mathcal{D}^a_p(p^v) \cdot \mathcal{D}^{m-b}_pM_{k-|V|}(x,p) \times \nonumber\\
&& \quad\quad\quad\quad \mathcal{D}_p^{|U| + b - a} J^v_{|V|-|U|}(x, p) \nonumber\\
\mathcal{D}^m_p M_k(x) &=& x\sum_{(a, b, c, d) \in S(k, m)} \binom{k}{d}\binom{d}{c}\sum_{v \geq 0} v^a \cdot \mathcal{D}^{m-b}_pM_{k-d}(x) \cdot  \mathcal{D}_p^{c + b - a} J^v_{d-c}(x) \nonumber\\
\mathcal{D}^m_p M_k(x)(1 - 2xM_0(x)) &=& x\sum_{(a, b, c, d) \in S'(k, m)} \binom{k}{d}\binom{d}{c}\sum_{v \geq 0} v^a \cdot \mathcal{D}^{m-b}_pM_{k-d}(x) \cdot  \mathcal{D}_p^{c + b - a} J^v_{d-c}(x). \nonumber
\end{eqnarray}

Notice that $M_0(x)$ is the generating function counting plane trees by number of edges. Thus $1 - 2xM_0(x) = \sqrt{1-4x}$. Thus, we arrive at the desired result.

\end{proof}

\bigskip

For simple subtree additive properties $\mathcal{P}^*$ where the property in question is not clear, we will denote the generating function for $\sum_{T \in \setofplanetrees_n}\mathcal{P}^*(T)^k$ by $M_{k}(\mathcal{P}^*, x)$. Let $[x^n]\mathcal{D}_p^mM_k(\mathcal{P}^*, x) = M^{(m)}_{k,n}(\mathcal{P}^*)$ (or simply $M^{(m)}_{k,n}$ if there is no ambiguity). Note that $M^{(m)}_{0,n}(\mathcal{P}^*) = \sum_{T \in \setofplanetrees_n} \mathcal{P}(T)^m$ where $\mathcal{P}$ is the additive property from which $\mathcal{P}^*$ is derived.

We now prove the following lemmas from which the main theorem of this section will follows. To do so, we utilize the following lemma (whose proof is deferred to the Appendix).

\begin{lemma}\label{tech lemma 1}
Let $a_1, a_2 \inR$ and $n \inN$ be large. For $\min\{a_1, a_2\} > -1 $,
$$\sum_{\substack{n_1 + n_2 = n}\\n_1, n_2 \geq 1} n_1^{a_1} \cdot n_2^{a_2} = \Theta\left(n^{a_1 + a_2 + 1}\right)$$
and, for $\min\{a_1, a_2\} < -1 $,
$$\sum_{\substack{n_1 + n_2 = n}\\n_1, n_2 \geq 1} n_1^{a_1} \cdot n_2^{a_2} = \Theta\left(n^{\max\{a_1, a_2\}}\right).$$
\end{lemma}

\noindent 

\begin{lemma} \label{bound M lemma}
	Let $\mathcal{P}$ be a non-negative integer valued additive property with toll function $f$ and let their induced subtree additive property be $\mathcal{P}^*$. Further assume that there exists $\zeta \inN$ such that for all $T \in \setofplanetrees_{\geq 0}$, $f(T) \leq \zeta \inN$. For all $m, k$, the following holds.
	$$M^{(m)}_{k,n} =  \Theta\left(n^{\frac{2m + 3k - 3}{2}} 4^n\right).$$
	
\end{lemma}

\begin{proof}
We show this result by induction on $k$ and $m$. The base case of $k = 0$ (and any $m$) holds by Lemma \ref{bound on F(x,p) lemma}. We now consider (\ref{equation 10}). Let 
\begin{eqnarray}
\Psi(k,m,a,b,c,d) &=& \frac{x}{\sqrt{1-4x}}\binom{k}{d}\binom{d}{c}\sum_{v \geq 0} v^a \cdot \mathcal{D}^{m-b}_pM_{k-d}(x) \cdot  \mathcal{D}_p^{c + b - a} J^v_{d-c}(x).
\end{eqnarray} We now take the coefficients on both sides of the equation. Let $[x^n]\mathcal{D}_p^mJ^v_k(x) = J^{(m)}_{k,n}(v)$.

\begin{eqnarray}
[x^n]\Psi(k,m,a,b,c,d) &=& \binom{k}{d}\binom{d}{c}\sum_{v \geq 0} v^a  \sum_{\substack{n_1 + n_2 + n_3 = n - 1\\n_1, n_2, n_3 \geq 1}} M^{(m-b)}_{k-d,n_1} \cdot J^{(c + b - a)}_{d-c,n_2}(v) \cdot \Theta\left(n_3^{-\frac{1}{2}}4^{n_3}\right) \nonumber \label{equation 23}
\end{eqnarray}
 where $(a,b,c,d) \in S'(k,m)$. 
 
 Fix $k \geq 1$, $m \geq 0$. Assuming the the theorem holds for all smaller $k$ and any $m$ as well as for equal $k$ and smaller $m$. Notice that the RHS of $(\ref{equation 23})$ depends precisely on terms for which the theorem holds. We now see using Lemma \ref{tech lemma 1} that 
 \begin{eqnarray}
[x^n]\Psi(k,m,a,b,c,d) &\leq& \binom{k}{d}\binom{d}{c} \zeta^a  \sum_{\substack{n_1 + n_2 + n_3 = n - 1\\n_1, n_2, n_3 \geq 1}} M^{(m-b)}_{k-d,n_1} \cdot M^{(c + b - a)}_{d-c,n_2}(v) \cdot O\left(n_3^{-\frac{1}{2}}4^{n_3}\right) \nonumber\\
&\leq&  O\left(4^{n} \sum_{\substack{n_1 + n_2 + n_3 = n - 1\\n_1, n_2, n_3 \geq 1}}  n_1^{a_1} \cdot n_2^{a_2} \cdot n_3^{-\frac{1}{2}} \right)
\end{eqnarray} where $a_1 = \frac{2(m-b) + 3(k-d) - 3}{2}$ and $a_2 = \frac{2(c+b-a) + 3(d-c) - 3}{2}$. Notice that for $(a,b,c,d) \in S'(k,m)$, $a_1, a_2 \geq -\frac{3}{2}$. We now consider the 2 possible cases:

Case 1: $(a,b,c,d) \in S'(k,m)$ such that $d = k, m = b$ or $d = c, c + b = a$. We see that $\min\{a_1, a_2\}  = -\frac{3}{2}$ and $\max\{a_1, a_2\}  = \frac{2m + 3k -c -2a - 3}{2}$. Thus we apply Lemma \ref{tech lemma 1} twice, noting that $\min\{a_1, a_2\}= -\frac{3}{2} < -1$, to get
\begin{eqnarray}
[x^n]\Psi(k,m,a,b,c,d) &\leq& O\left(n^{\frac{2m + 3k -c -2a - 2}{2}} 4^{n}\right).
\end{eqnarray}
Note that we get the most significant upper bound when we minimize $c + 2a$. The minimal $c + 2a$ for which there is $(a,b,c,d) \in S'(k,m)$ for some $k,m$ in this case is when  $c = 1, a = 0$.

Case 2: All other cases. We see that $\min\{a_1, a_2\}  > -1$. Thus we apply Lemma \ref{tech lemma 1} twice to get 
\begin{eqnarray}
[x^n]\Psi(k,m,a,b,c,d) &\leq& O\left(n^{\frac{2m + 3k - c - 2a  - 3}{2}} 4^{n}\right).
\end{eqnarray}
Note that we get the most significant upper bound when we minimize $c + 2a$. The minimal $c + 2a$ for which there is $(a,b,c,d) \in S'(k,m)$ in this case is when  $c = 0, a = 0$. We note that when $k = 1, m = 0$ there is no $(a,b,c,d) \in S'(k,m)$ in this case.

Thus, for all $(a,b,c,d) \in S'(k,m)$, $[x^n]\Psi(k,m,a,b,c,d) \leq  O\left(4^{n} n^{\frac{2m + 3k  - 3}{2}}\right)$. From (\ref{equation 10}), we hence get
\begin{eqnarray}
\mathcal{D}^m_p M_k( x) &=& \sum_{(a, b, c, d) \in S'(k, m)}\Psi(k,m,a,b,c,d)\\
M^{(m)}_{k,n} &\leq & O\left(4^{n} n^{\frac{2m + 3k  - 3}{2}}\right).
\end{eqnarray}

Towards the lower bound, we see when $a = 0$,
\begin{eqnarray}
[x^n]\Psi(k,m,a,b,c,d) &\geq& \binom{k}{d}\binom{d}{c}  \sum_{\substack{n_1 + n_2 + n_3 = n - 1\\n_1, n_2, n_3 \geq 1}} M^{(m-b)}_{k-d,n_1}\cdot M^{(c + b - a)}_{d-c,n_2} \cdot \Theta\left(n_3^{-\frac{1}{2}}4^{n_3}\right) \nonumber\\
&\geq& \Omega\left(4^{n} \sum_{\substack{n_1 + n_2 + n_3 = n - 1\\n_1, n_2, n_3 \geq 1}}  n_1^{a_1} \cdot n_2^{a_2} \cdot n_3^{-\frac{1}{2}} \right)
\end{eqnarray} where $a_1 = \frac{2(m-b) + 3(k-d) - 3}{2}$ and $a_2 = \frac{2(c+b-a) + 3(d-c) - 3}{2}$. We break this into cases exactly as before. We notice that there is $(a,b,c,d) \in S'(k,m)$ where $a = 0$ and $[x^n]\Psi(k,m,a,b,c,d) \geq \Omega\left(4^{n} n^{\frac{2m + 3k  - 3}{2}}\right)$.
Thus, from (\ref{equation 10}), we get
\begin{eqnarray}
M^{(m)}_{k,n}(\mathcal{P}^*) &\geq & \Omega\left(4^{n} n^{\frac{2m + 3k  - 3}{2}}\right).
\end{eqnarray} The result follows by induction. 

\end{proof}

\bigskip
\begin{observation} \label{bound on psi observation}
In the proof of Lemma \ref{bound M lemma}, for any $k,m \geq 0$ and $(a,b,c,d) \in S'(k,m)$, when $a \geq 1$, 
$$[x^n]\Psi(k,m,a,b,c,d) \leq O\left(n^{\frac{2m + 3k -4}{2}} 4^{n}\right).$$
\end{observation}

\bigskip
\begin{lemma} \label{c^m+k lemma}
	Let $\mathcal{P}_1$ and $\mathcal{P}_2$ be non-negative integer valued additive properties with toll functions $f_1$ and $f_2$, respectively. Let the induced subtree additive properties of $\mathcal{P}_1$ and $\mathcal{P}_2$ be $\mathcal{P}^*_1$ and $\mathcal{P}^*_2$, respectively. Further assume that there exists $\zeta \inN$ such that for all $T \in \setofplanetrees_{\geq 0}$, $f_1(T) \leq \zeta \inN$ and $f_2(T) \leq \zeta \inN$.  Assume for all $m, n \geq 0$, $$\sum_{T \in \setofplanetrees_n} \mathcal{P}_1(T)^m =  \mu^m \cdot \sum_{T \in \setofplanetrees_n}  \mathcal{P}_2(T)^m + O\left(n^{\frac{2m-4}{2}}4^n\right)$$ where $\mu$ is a constant. 
	It then holds that for all $n, m, k \geq 0$,
	$$M^{(m)}_{k,n}(\mathcal{P}_1^*) = \mu^{k+m} \cdot M^{(m)}_{k,n}(\mathcal{P}_2^*) + O\left(n^{\frac{2m + 3k-4}{2}}4^n\right).$$
\end{lemma}
	
\begin{proof}
	We prove the result by induction on $k$ and $m$. The base case of $k = 0$ (and any $m$) is true by assumption. Fix $k \geq 1$, $m \geq 0$. Assuming the the theorem holds for all smaller $k$ and any $m$ as well as for equal $k$ and smaller $m$. We now apply Lemma \ref{bound M lemma} and Observation \ref{bound on psi observation} regarding $\Psi(k,m,a,b,c,d)$ to get 
	
	\begin{eqnarray}
	M^{(m)}_{k,n}(\mathcal{P}_1^*) &=& \sum_{(a, b, c, d) \in S'(k, m)} \binom{k}{d}\binom{d}{c}\sum_{v \geq 0} v^a  \sum_{n_1 + n_2 + n_3 = n - 1} M^{(m-b)}_{k-d,n_1}(\mathcal{P}_1^*) \cdot J^{(c + b - a)}_{d-c,n_2}(\mathcal{P}_1^*, v) \cdot \binom{2n_3}{n_3}\nonumber\\
	&=& \mu^{m+k}\sum_{\substack{(a, b, c, d) \in S'(k, m)\\ a =0}} \binom{k}{d}\binom{d}{c} \sum_{n_1 + n_2 + n_3 = n - 1} M^{(m-b)}_{k-d,n_1}(\mathcal{P}_2^*) \cdot M^{(c + b - a)}_{d-c,n_2}(\mathcal{P}_2^*) \cdot \binom{2n_3}{n_3}\nonumber \\
	&&\quad\quad\quad + \sum_{\substack{(a, b, c, d) \in S'(k, m)\\ a \geq 1}} \Psi(k,m,a,b,c,d) + O\left(n^{\frac{2m + 3k -4}{2}} 4^{n}\right)\\
	&=& \mu^{m+k} \cdot M^{(m)}_{k,n}(\mathcal{P}_2^*) + O\left(n^{\frac{2m + 3k -4}{2}} 4^{n}\right)
	\end{eqnarray}
	
	Thus the result holds for all $k, m \inN$ by induction.\\
	
\end{proof}

\begin{lemma} \label{ssap are uniquely detrmined by moments}
    Let $\mathcal{P}$ be an non-negative integer valued additive properties of plane trees with toll function $f$. Let the subtree additive property induced by $\mathcal{P}$ be $\mathcal{P}^*$. Further assume that there exist $\zeta \inN$ such that for any $T \in \setofplanetrees_{\geq 0}$, $f(T) \leq  \zeta$. The limiting distribution of $$\frac{\mathcal{P}^*(\setofplanetrees_n)}{\sqrt{n^3}}$$ is unique determined by its moments. Specifically, it satisfies the Carleman's condition.
\end{lemma}

\begin{proof} Let $\mathcal{P}_1$ be the additive property with toll function $f$ where $f(T) = \zeta$ for all $T \in \setofplanetrees_{\geq 0}$. Let $\mathcal{P}^*_1$ be the subtree additive property derived from $\mathcal{P}_1$. We see that $\mathcal{P}(T) \leq \mathcal{P}_1(T)$, and hence, $\mathcal{P}^*(T) \leq \mathcal{P}_1^*(T)$ for all $T \in \setofplanetrees_{\geq 0}$. From Example  \ref{Example number of leaves and edges}, $\mathcal{P}_1(T) = \zeta \cdot \mathcal{P}^v(T)$, thus $\mathcal{P}_1^*(T) = \zeta \cdot \mathcal{P}^{PL}(T)$. Applying Theorem \ref{contactdistance - thm}, we see that 
$$\lim_{n \rightarrow \infty}\E\left[\left(\frac{\mathcal{P}^*(\setofplanetrees_{n})}{\sqrt{2n^3}} \right)^k \right] \leq \lim_{n \rightarrow \infty}\E\left[\left(\frac{\zeta \cdot \mathcal{P}^{PL}(\setofplanetrees_{n})}{\sqrt{2n^3}} \right)^k \right] \: \sim \: \frac{6k}{\sqrt{2}}\left(\frac{k}{12e}\right)^{\frac{k}{2}} \cdot \zeta^k$$ as $k \rightarrow \infty$. Thus applying the Carleman Condition (Theorem \ref{Carleman's condition}), we get the desired result.

\end{proof}

\begin{proof}[Proof of Theorem \ref{mainSimpleAdditveTheorem1}]
    We recall that for all $n,k \geq 0$ and $\mathcal{P}^* \in \{\mathcal{P}_1^*, \mathcal{P}_2^*\}$, $M^{(0)}_{k,n}(\mathcal{P}^*) = \sum_{T \in \setofplanetrees_n} \mathcal{P}^*(T)^k$, hence $\E[\mathcal{P}^*(\setofplanetrees_n)^k] = \frac{M^{(0)}_{k,n}(\mathcal{P}^*)}{C_n}$. We apply the assumption of the Theorem \ref{mainSimpleAdditveTheorem1}, Lemma \ref{bound M lemma} and Lemma \ref{c^m+k lemma} to achieve 
    \begin{eqnarray}
    \E\left[\left(\frac{\mathcal{P}_1^*(\setofplanetrees_n)}{\sqrt{n^3}}\right)^k\right] &=& \mu^k \cdot \E\left[\left(\frac{\mathcal{P}_2^*(\setofplanetrees_n)}{\sqrt{n^3}}\right)^k\right] + O\left(\frac{1}{\sqrt{n}}\right).
    \end{eqnarray}
    
    We now apply Lemma \ref{ssap are uniquely detrmined by moments} to see that the limiting distribution of the properties are unique determined by their moments. Thus, we arrive at the desired result.
    
\end{proof}

\bigskip

\begin{corollary} \label{corollary P^LRD and P^IRD dist}
The total leaf to root distance of a random plane tree on $n$ edges, $\mathcal{P}^{LR}(\setofplanetrees_n)$, and the total internal node to root distance of a random plane tree on $n$ edges, $\mathcal{P}^{IR}(\setofplanetrees_n)$ is asymptotically Airy distributed, up to a scaling factor. Specifically, as $n \rightarrow \infty$, $$2\cdot \frac{\mathcal{P}^{LR}(\setofplanetrees_n)}{\sqrt{2n^3}} \quad \overset{d}{\longleftrightarrow} \quad 4 \cdot \frac{\mathcal{P}^{IR}(\setofplanetrees_n)}{\sqrt{2n^3}} \quad \overset{d}{\longleftrightarrow} \quad \frac{\mathcal{P}^{PL}(\setofplanetrees_n)}{\sqrt{2n^3}}.$$

\end{corollary}
\begin{proof}
Let $F(x, p)$ be the generating function for plane trees weighted by number of edges and vertices. Let $F_{d_0}(x, p)$ be the generating function for plane trees weighted by number of edges and leaves. Let $F_{d_1}(x, p)$ be the generating function for plane trees weighted by number of edges and internal nodes. From Corollary \ref{eval G func},
$$F_{d_0}(x, p) = \frac{1 + (1-p)x - \sqrt{(1+(1-p)x)^2 - 4x}}{2x}$$ 
and
$$F_{d_1}(x, p) = \frac{1 + (1-p)x - \sqrt{(1+(1-p)x)^2 - 4x(1 - (p-1)x)}}{2x}.$$
Thus, we see that $$[x^n]\frac{\partial F_{d_0}(x, 1)}{\partial p^m} = \frac{1}{2^m} \cdot  [x^n]\frac{\partial F(x, 1)}{\partial p^m}  + O\left((1- 4x)^{\frac{2-2m}{2}}\right)$$ 
and 
$$[x^n]\frac{\partial F_{d_1}(x, 1)}{\partial p^m} = \frac{1}{4^m} \cdot  [x^n]\frac{\partial F(x, 1)}{\partial p^m} + O\left((1- 4x)^{\frac{2-2m}{2}}\right).$$ 
Thus $$\sum_{T \in \setofplanetrees} \mathcal{P}^{d_0}(T)^m = \frac{1}{2^m} \cdot \sum_{T \in \setofplanetrees} \mathcal{P}^{v}(T)^m + O\left(n^{\frac{2m-4}{2}}4^n\right)$$ and $$\sum_{T \in \setofplanetrees} \mathcal{P}^{d_1}(T)^m = \frac{1}{4^m} \cdot \sum_{T \in \setofplanetrees} \mathcal{P}^{v}(T)^m + O\left(n^{\frac{2m-4}{2}}4^n\right).$$
We now apply Theorem \ref{mainSimpleAdditveTheorem1} to get that, as $n \rightarrow \infty$,
$$2 \cdot \frac{\mathcal{P}^{LR}(\setofplanetrees_n)}{\sqrt{2n^3}} \quad \overset{d}{\longleftrightarrow} \quad 4 \cdot \frac{\mathcal{P}^{IR}(\setofplanetrees_n)}{\sqrt{2n^3}} \quad \overset{d}{\longleftrightarrow} \quad \frac{\mathcal{P}^{PL}(\setofplanetrees_n)}{\sqrt{2n^3}}.$$  
From Theorem \ref{contactdistance - thm}, we know that $\frac{\mathcal{P}^{PL}(\setofplanetrees_n)}{\sqrt{2n^3}}$ converges weakly to an Airy random variable, thus the same holds for $2 \cdot \frac{\mathcal{P}^{LR}(\setofplanetrees_n)}{\sqrt{2n^3}}$ and $4 \cdot \frac{\mathcal{P}^{IR}(\setofplanetrees_n)}{\sqrt{2n^3}}$.

\end{proof}

%% file: Results_Sec_3.tex
\subsection{The Distribution of Classes of Subtree Additive Properties under NNTM}
\label{Dist of Subtree Additive Properties NNTM}

Let $f: \setofplanetrees_{\geq 0} \times \setofplanetrees_{\geq 0} \rightarrow \R$ be a polynomial in the number of edges, leaves and  internal nodes of its input trees. Hence, $f(T', T) \in \R[t, \lvl, \lvi][n, \gvl, \gvi]$ where $t, \lvl, \lvi$ represents the number of edges, leaves and internal nodes of $T'$ and $n, \gvl, \gvi$ represents the number of edges, leaves and internal nodes of $T$.

We now consider $\mathcal{P}^f$, the subtree additive property given by 
\begin{equation} \label{f tree def}
    \mathcal{P}^{f}(T)=  \sum_{v \in \overline{\mathcal{V}}(T)} f(n(T_v), d_0(T_v), d_1(T_v), n(T), d_0(T), d_1(T)) = \sum_{v \in \overline{\mathcal{V}}(T)} f(T_v, T) 
\end{equation} where, for $v \in \overline{\mathcal{V}}(T)$, the root vertex in $T_v$ is allowed to count as a leaf or internal node. 

Let $f$ be  a polynomial over $\R$ in the variables $x_1, \cdots, x_n$. We can write $f$ in the form of $\sum_{j=1}^L w_j \cdot v_j$ where $w_j \inR$ and $v_j$ is a monic monomial in the variables such that the monomial are all distinct. We will call this the reduced form of $f$. It should be clear that $f$ can be written uniquely in this form. For any such $v_j$ and variables $x_1, \cdots, x_l$, let $\Delta(v_j|x_1, \cdots, x_l)$ be the degree of the monomial $v_j$ if we consider all variables other than $x_1, \cdots, x_l$ to be constant. For a vector of variables $\vec{x} = (x_1, \cdots, x_l)$, let $\Delta(v_j|\vec{x}) = \Delta(v_j|x_1, \cdots, x_l)$. We also define $\Delta(f|x_1, \cdots, x_l) = \max_{1 \leq j \leq L}\Delta(v_j|x_1, \cdots, x_l)$. \\

Fix parameters $\alpha,\beta, \gamma \inR$. Let
$$M_{k,n}(f, \alpha,\beta, \gamma) = \sum_{T \in \setofplanetrees_ n} \mathcal{P}^f(T)^ke^{-E(T)} = \E\left[\mathcal{P}^f_{(\alpha, \beta, \gamma)}(\setofplanetrees_n)^k\right] \cdot \normConst_{(n,\alpha,\beta,\gamma)}$$ and $$M_{k}(f, \alpha,\beta, \gamma)(x) = \sum_{n \geq 0}M_{k,n}(f, \alpha,\beta, \gamma)x^n.$$\\

We provide the following theorem that specifies the asymptotic form of $M_{k,n}(f, \alpha,\beta, 0)$.\\

\begin{theorem} 
	\label{mainTheorem01}
	Fix $k \inN$, $\alpha,\beta \inR$ and $f  \inR[t, \lvl, \lvi][n, \gvl, \gvi]$. Let $\delta = \Delta(f| t, n,  \lvl, \gvl, \lvi, \gvi)$. Let $V_k(f) = \frac{(2\delta+ 1)k -1}{2}$ if for every monomial of $f$, $u$, such that $\Delta(u| t, n,  \lvl, \gvl, \lvi, \gvi) = \delta$, $\Delta(u| t, \lvl, \lvi) > 0$  and $V_k(f) = \frac{ 2(\delta+ 1)k - 1}{2}$ otherwise. We have that
	
	$$M_{k}(f, \alpha,\beta, 0)(x) =  \frac{W(\alpha, \beta, f)}{\left(1-\rho x\right)^{V_k(f)}} + O\left(\frac{1}{\left(1-\rho x\right)^{V_k(f) - \frac{1}{2}}}\right),$$ where $\rho = e^{-\alpha} + e^{-\beta} + 2e^{-\frac{\alpha}{2}}$ and $W(\alpha, \beta, f)$ is a constant that depends on $f, \alpha, \beta$.
	
	Furthermore, assume every monomial of $f$, $u$, such that $\Delta(u| t, n,  \lvl, \gvl, \lvi, \gvi) = \delta$ is such that  $\Delta(u| t, n) = \delta_n$, $\Delta(u| \lvl, \gvl) = \delta_{d_0}$ and $\Delta(u| \lvi, \gvi) = \delta_{d_1}$ where $\delta_n$, $\delta_{d_0}$ and $\delta_{d_1}$ are constants that depend only on $f$. We have that	
	$$W(\alpha, \beta, f) = Q_k(f, \delta_{d_0}, \delta_{d_1}, \alpha,\beta) \cdot P_{k}(f)$$ 
	where $Q_k(f, \delta_{d_0}, \delta_{d_1}, \alpha,\beta)$ depends on $\delta_{d_0}, \delta_{d_1}$ and $V_k(f)$, and $P_{k}(f)$ is a constant that depends on $f$ and is independent of $\alpha$ and $\beta$. 
	
\end{theorem}

\bigskip

\input{Proof_of_Moment_Generalization_Theorem_1.tex}

Fix $h \inN$. Let $\setofplanetrees^h_ n$ for $h \inN$ be the set of plane trees on $n$ edges with root degree at most $h$. We define $$M_{k,n}(\left. f, \alpha,\beta, \gamma\right| h) = \sum_{T \in \setofplanetrees^h_n} \mathcal{P}^f(T)^ke^{-E(T)}$$ and
$$M_{k}(\left. f, \alpha,\beta, \gamma\right| h)(x) = \sum_{n \geq 0} M_{k,n}(\left. f, \alpha,\beta, \gamma\right| h) x^n.$$

\begin{theorem} 
	\label{mainTheorem02}
	Fix $k \inZ_{\geq 0}$ and $\alpha,\beta, \gamma \inR$. Let $f  \inR[t, \lvl, \lvi][n, \gvl, \gvi]$. Let $\delta = \Delta(f| t, n,  \lvl, \gvl, \lvi, \gvi)$. Let $V_k(f) = \frac{(2\delta+ 1)k -1}{2}$ if for every monomial of $f$, $u$, such that $\Delta(u| t, n,  \lvl, \gvl, \lvi, \gvi) = \delta$, $\Delta(u| t, \lvl, \lvi) > 0$  and $V_k(f) = \frac{ 2(\delta+ 1)k - 1}{2}$ otherwise. We have
	$$M_{k}(\left. f, \alpha,\beta, \gamma\right| h)(x) =  \frac{J_h(\alpha,\beta, \gamma) \cdot W(\alpha, \beta, f)}{\left(1-\rho x\right)^{V_k(f)}} + O\left(\frac{1}{\left(1-\rho x\right)^{V_k(f) - \frac{1}{2}}}\right)$$ where $\rho = e^{-\alpha} + e^{-\beta} + 2e^{-\frac{\alpha}{2}}$, $J_h(\alpha,\beta, \gamma)$ is a constant independent of $f$ and $W(\alpha, \beta, f)$ is the same as in Theorem \ref{mainTheorem01}.

\end{theorem}

\bigskip

\input{Proof_of_Moment_Generalization_Theorem_2.tex}

\bigskip

Let $\mathcal{P}^f_{(\alpha, \beta, \gamma)}(\setofplanetrees^h_n)$ be defined similarly to $\mathcal{P}^f_{(\alpha, \beta, \gamma)}(\setofplanetrees_n)$ conditioned on the tree chosen having root degree at most $h$. From the above theorem, we can conclude the following.

\begin{theorem} \label{maintheorem3}
Let $\alpha,\beta, \gamma \inR$ and $f  \inR[t, \lvl, \lvi][n, \gvl, \gvi]$. Assume every monomial of $f$, $u$, such that $\Delta(u| t, n,  \lvl, \gvl, \lvi, \gvi) = \delta$ is such that  $\Delta(u| t, n) = \delta_n$, $\Delta(u| \lvl, \gvl) = \delta_{d_0}$ and $\Delta(u| \lvi, \gvi) = \delta_{d_1}$ where $\delta_n$, $\delta_{d_0}$ and $\delta_{d_1}$ are constants that depend only on $f$. Let $V'(f) = \frac{2\delta+ 1}{2}$ if for every monomial of $f$, $u$, such that $\Delta(u| t, n,  \lvl, \gvl, \lvi, \gvi) = \delta$, $\Delta(u| t, \lvl, \lvi) > 0$  and $V'(f) = \delta+ 1$ otherwise. We have

$$\frac{\mathcal{P}^f_{(\alpha, \beta, \gamma)}(\setofplanetrees^h_ n)}{n^{V'(f)}} \overset{d}{\longleftrightarrow}  \frac{\mathcal{P}^f_{(\alpha, \beta, 0)}(\setofplanetrees_ n)}{n^{V'(f)}}  \overset{d}{\longleftrightarrow} \frac{Q(f, \alpha, \beta) \cdot \mathcal{P}^f_{(0,0,0)}(\setofplanetrees_ n)}{n^{V'(f)}} $$ 
where, for $V'(f) = \frac{2\delta+ 1}{2}$, $$Q(f, \alpha, \beta) = 2^{\delta_{d_0} + 2\delta_{d_1}} \cdot \frac{ (e^{-\frac{\alpha}{2}}+1)^{\delta_{d_0} - 1} \cdot e^{-\frac{\alpha}{4}(2 \delta_{d_0}-1)} \cdot e^{-\beta \delta_{d_1}}}{\sqrt{\rho^{2\delta_{d_0} + 2\delta_{d_1} - 1}}}$$
and, for $V'(f) = \delta+ 1$, $$Q(f, \alpha, \beta) = 2^{\delta_{d_0} + 2\delta_{d_1}} \cdot \frac{(e^{-\frac{\alpha}{2}}+1)^{\delta_{d_0}} \cdot e^{-\frac{\alpha \delta_{d_0}}{2}} \cdot e^{-\beta \delta_{d_1}}}{\rho^{\delta_{d_0} + \delta_{d_1}}}.$$
	
\end{theorem}

\bigskip

\begin{proof}

We notice that $M_{k,n}(1,\alpha, \beta, \gamma) = n^k\normConst_{(n, \alpha, \beta, \gamma)}$. For a random variable $X$, let $\text{Mom}_k(X)$ the $k$th moment of $X$. Fixed $n \inZ_{\geq 0}$.  By definition, $$\text{Mom}_k(\mathcal{P}^f_{(\alpha, \beta, 0)}(\setofplanetrees_ n)) = \frac{M_{k,n}(f,\alpha, \beta, 0)}{\normConst_{(n, \alpha, \beta, 0)}} =  \frac{n^k \cdot M_{k,n}(f,\alpha, \beta, 0)}{M_{k,n}(1,\alpha, \beta, 0)}$$

We now apply Theorem \ref{mainTheorem01} and use the standard expression for the asymptotic coefficients of the Taylor series expansion of $(1-x)^k$ for $k \inC$. We use the Transfer Theorem from \cite{flajolet2009analytic}, to get $$M_{k,n}(f, \alpha, \beta, 0) = \frac{Q_k(f, \delta_{d_0}, \delta_{d_1}, \alpha,\beta) \cdot P_{k}(f) \cdot n^{V_k(f)-1} \cdot (e^{-\alpha} + e^{-\beta}+ 2e^{-\frac{\alpha}{2}})^n}{\Gamma(V_k(f))}$$ where we only consider the most significant terms.

Notice that $V_k(1) = \frac{2k-1}{2}$ and every maximal monomial of $1$ is independent of $t, \lvl, \lvi$. Thus 
$$\text{Mom}_k(\mathcal{P}^f_{(\alpha, \beta, 0)}(\setofplanetrees_ n)) =  \frac{Q_k(f, \delta_{d_0}, \delta_{d_1}, \alpha,\beta) \cdot P_{k}(f) \cdot \Gamma(V_k(1))}{Q_k(1, 0,0, \alpha,\beta) \cdot P_{k}(1) \cdot \Gamma(V_k(f)) } \cdot n^{V_k(f)+ \frac{1}{2}} + O\left(n^{V_k(f)}\right)$$

Notice that $V'(f)k = V_k(f)+ \frac{1}{2}$. Thus 

$$\lim_{n\rightarrow \infty}\text{Mom}_k\left(\frac{\mathcal{P}^f_{(\alpha, \beta, 0)}(\setofplanetrees_ n)}{n^{V'(f)}}\right) =  \frac{Q_k(f, \delta_{d_0}, \delta_{d_1}, \alpha,\beta) \cdot P_{k}(f) \cdot \Gamma(V_k(1))}{Q_k(1, 0,0, \alpha,\beta) \cdot P_{k}(1) \cdot \Gamma(V_k(f)) }.$$ 

We now see that 
$$\text{Mom}_k(\mathcal{P}^f_{(\alpha, \beta, \gamma)}(\setofplanetrees_n^h)) = \frac{n^k \cdot M_{k,n}(f,\alpha, \beta, \gamma|h)}{M_{k,n}(1,\alpha, \beta, \gamma|h)}$$ 

Using Theorem \ref{mainTheorem02} and nearly identical computation to above, we see that 

$$\lim_{n\rightarrow \infty}\text{Mom}_k\left(\frac{\mathcal{P}^f_{(\alpha, \beta, \gamma)}(\setofplanetrees^h_ n)}{n^{V'(f)}}\right) =  \frac{Q_k(f, \delta_{d_0}, \delta_{d_1}, \alpha,\beta) \cdot P_{k}(f) \cdot \Gamma(V_k(1))}{Q_k(1, 0,0, \alpha,\beta) \cdot P_{k}(1) \cdot \Gamma(V_k(f)) }.$$ 

Finally, we set that \begin{eqnarray*} 
\lim_{n\rightarrow \infty}\text{Mom}_k\left(\frac{\mathcal{P}^f_{(\alpha, \beta, 0)}(\setofplanetrees_ n)}{n^{V'(f)}}\right) &=& \frac{Q_k(f, \delta_{d_0}, \delta_{d_1}, \alpha,\beta) \cdot P_{k}(f) \cdot \Gamma(V_k(1))}{Q_k(1, 0,0, \alpha,\beta) \cdot P_{k}(1) \cdot \Gamma(V_k(f))}\\
&=& \frac{Q_k(f, \delta_{d_0}, \delta_{d_1}, \alpha,\beta) \cdot Q_k(1, 0,0,0,0)}{Q_k(f, \delta_{d_0}, \delta_{d_1}, 0,0) \cdot Q_k(1, 0,0, \alpha,\beta)} \cdot \lim_{n\rightarrow \infty}\text{Mom}_k\left(\frac{\mathcal{P}^f_{(0, 0, 0)}(\setofplanetrees_ n)}{n^{V'(f)}}\right)\\
&=&  \lim_{n\rightarrow \infty}\text{Mom}_k\left(Q(f, \alpha, \beta) \cdot \frac{\mathcal{P}^f_{(0, 0, 0)}(\setofplanetrees_ n)}{n^{V'(f)}}\right)
\end{eqnarray*} where $Q(f, \alpha, \beta) = \sqrt[k]{\frac{Q_k(1, 0,0,0,0)}{Q_k(f, \delta_{d_0}, \delta_{d_1}, 0,0)} \cdot \frac{Q_k(f, \delta_{d_0}, \delta_{d_1}, \alpha,\beta)}{Q_k(1, 0,0, \alpha,\beta)}}$ is a constant independent of $k$.

\end{proof}

\bigskip

\begin{corollary} \label{full dist corr CD}
Let $\alpha,\beta, \gamma \inR$. For any $h \inN$, we have  

$$\frac{\mathcal{P}^{PL}_{(\alpha, \beta, \gamma)}(\setofplanetrees^h_ n)}{\sqrt{2n^3}} \quad \overset{d}{\longleftrightarrow}  \quad  \frac{\mathcal{P}^{PL}_{(\alpha, \beta, 0)}(\setofplanetrees_ n) }{\sqrt{2n^3}} \quad \overset{d}{\to} \quad \frac{\sqrt{\rho}}{ (e^{-\frac{\alpha}{2}}+1) \cdot e^{-\frac{\alpha}{4}}} \int_0^1e(t) \: dt$$ where $e(t)$ is a normalized Brownian excursion on $[0, 1]$.
\end{corollary}
\begin{proof}
	We take $f = t+1$ then apply Theorem $\ref{maintheorem3}$ and Theorem \ref{contactdistance - thm}.
	
\end{proof}

\bigskip

\begin{corollary}
	Let $\alpha,\beta, \gamma \inR$. For any $h \inN$, we have  
	
	\begin{multline*}
	    \frac{\mathcal{P}^{WI}_{(\alpha, \beta, \gamma)}(\setofplanetrees^h_ n)}{\sqrt{2n^5}} \quad \overset{d}{\longleftrightarrow}  \quad  \frac{\mathcal{P}^{WI}_{(\alpha, \beta, 0)}(\setofplanetrees_ n) }{\sqrt{2n^5}}\\
	    \overset{d}{\to} \quad \frac{\sqrt{\rho}}{ (e^{-\frac{\alpha}{2}}+1) \cdot e^{-\frac{\alpha}{4}}} \int \int_{0 < s < t < 1} (e(s) + e(t) - 2\min_{s \leq u \leq t}e(u)) \: ds \: dt
	\end{multline*}
	where $e(t)$ is a normalized Brownian excursion on $[0, 1]$.
\end{corollary}
\begin{proof}
	We take $f = (t+1)(n - t)$ then apply Theorem $\ref{maintheorem3}$ and Theorem \ref{ladderdistance - thm}.
	
\end{proof}

\bigskip

\begin{corollary} \label{full dist corr LRD/IRD}
Let $\alpha,\beta, \gamma \inR$. For any $h \inN$, we have  

$$\frac{\mathcal{P}^{LR}_{(\alpha, \beta, \gamma)}(\setofplanetrees^h_ n)}{\sqrt{2n^3}} \quad \overset{d}{\longleftrightarrow}  \quad   \frac{\mathcal{P}^{LR}_{(\alpha, \beta, 0)}(\setofplanetrees_ n) }{\sqrt{2n^3}} \quad \overset{d}{\to}  \quad \frac{e^{-\frac{\alpha}{4}}}{\sqrt{\rho}} \int_0^1e(t) \: dt$$ 
and 
$$\frac{\mathcal{P}^{IR}_{(\alpha, \beta, \gamma)}(\setofplanetrees^h_ n)}{\sqrt{2n^3}} \quad \overset{d}{\longleftrightarrow}  \quad  \frac{\mathcal{P}^{IR}_{(\alpha, \beta, 0)}(\setofplanetrees_ n) }{\sqrt{2n^3}} \quad \overset{d}{\to}  \quad \frac{ e^{-\beta}}{\sqrt{\rho} \cdot e^{-\frac{\alpha}{4}} \cdot (e^{-\frac{\alpha}{2}}+1)} \int_0^1e(t) \: dt$$
where $e(t)$ is a normalized Brownian excursion on $[0, 1]$.
\end{corollary}
\begin{proof}
	For $\mathcal{P}^{LR}$, we take $f = \lvl$.  For $\mathcal{P}^{IR}$, we take $f = \lvi$. We then apply Theorem $\ref{maintheorem3}$ and Corollary \ref{corollary P^LRD and P^IRD dist}.
	
\end{proof}

\bigskip

%% file: Proof_of_Moment_Generalization_Theorem_1.tex
We now describe a construction that will allow us to prove the above result. Fix $k \inN$. Let $\setofcomp^p$ be the set of compositions of $k$ into $p$ parts, where a composition of an integer $k$ into $p$ parts is a $p$-tuple $(c_1, \cdots, c_p)$ such that $c_i \inN$ and $\sum_{i=1}^p c_i = k$. Let $$\setoftuples = \bigcup_{T \in \setofplanetrees_{\geq 1}} \{T\} \times \overline{\mathcal{V}}(T)^k \quad\quad\quad \text{and} \quad\quad\quad \setoftreecomppairs = \bigcup_{1 \leq i \leq k} \setofplanetrees_i \times \setofcomp^i.$$  

Let $(T, \vec{v}) \in \setoftuples$. We will define the \emph{skeleton tree} of this tuple to be the tree formed by contracting the parent edge of all vertices except those in $\vec{v}$. We will denote the skeleton tree of the tuple as $\skele(T, \vec{v})$. We define the composition of the tuple $(T, \vec{v})$ by $\comp(T, \vec{v}) = (c^*_1, \cdots, c^*_\lambda)$, where $c^*_i$ is the number of times the $i$th non-root vertex in the pre-ordering of $\skele(T, \vec{v})$ appears in $\vec{v}$. Notice that $\comp(T, \vec{v}) \in \setofcomp^\lambda$ and $\skele(T, \vec{v}) \in \setofplanetrees_\lambda$ where $\lambda \leq k$ is the number distinct vertices in $\vec{v}$.

We define $\phi:\setoftuples \rightarrow \setoftreecomppairs$ as $\phi(T, \vec{v}) = (\skele(T, \vec{v}), \comp(T, \vec{v}))$. (See the top of Figure \ref{fig:phiexamplefunction}.) We will denote the pre-image of $(T_s, C) \in \setoftreecomppairs$ under $\phi$ as $\phi^{-1}(T_s, C) \subset \setoftuples$. We will now show a construction for the elements of $\phi^{-1}(T_s, C)$ that will allow us to count the elements of $\phi^{-1}(T_s, C)$.

\input{Figures/Figure_1}

Let $(T_s, C) \in \setoftreecomppairs$ such that $T_s \in \setofplanetrees_\lambda$. We label each vertex in $T_s$. We assign 0 to the root vertex and assign $i$ to the $i$th non-root vertex in the pre-ordering of $T_s$. From this point, we will refer to the vertices of $T_s$ by their assigned number. Let $m_i$ be the number of children of the $i$th vertex of $T_s$.

Each element of $\phi^{-1}(T_s, C)$ may be constructed as follows. To each vertex in $T_s$, we shall assign 2 trees. We shall call these trees the local tree and the fused tree assigned to the vertex. We now proceed recursively. Consider a vertex $v$, labelled $i$, in $T_s$ with children to all of which a local and fused tree has been assigned. We select a tree with $t_i$ edges, $\lvl_i$ leaves, $\lvi_i$ internal nodes and root degree $\lvr_i$. This tree will be the local tree assigned to the vertex, denoted $T^L_i$. 

We select $m_i$ leaves in $T^L_i$. We then identify these leaves with the root vertices of the fused trees assigned to the children of $v$. We denote these root vertices $v_1, \cdots, v_{m_i}$, where $v_i$ is the root vertex assigned to the fused tree of the $i$th child of $v$ encountered in a pre-order traversal. We identify each root vertex to a unique leaf in such a way that $v_i$ is encountered before $v_{i+1}$ is a pre-order traversal of the new tree formed. This new tree will be the fused tree assigned to $v$, denoted $T^F_i$. We continue this process until all the vertices in $T_s$ have a fused and local tree assigned to them. (See the bottom of Figure \ref{fig:phiexamplefunction}.)

The elements of $\phi^{-1}(T_s, C)$ this construction will correspond to, $(T, \vec{v})$, will have $T = T_0^F$. The vertices in $\vec{v}$ will be root vertices of each $T^L_i$ for $i \neq 0$. The number of times each vertex appears in $\vec{v}$ is determined by $C = (c^*_1, \cdots, c^*_\lambda)$, where the root of $T^L_i$ appears $c^*_i$ times. We are now free to choose any arrangement of these vertices into a tuple. There are $\frac{k!}{c^*_1!\cdots c^*_\lambda!}$ such choices, which we will denote as $\binom{k}{C}$. Notice that every tuple formed by this process of building up a tree and arranging special vertices from the tree is unique.\\

\begin{observation}\label{observation about sumtree props 1}
	Let $S_i \subset \{0, \cdots, \lambda\}$ be the set indices $j$, such that the vertex $i$ of $T_s$ is an ancestor of vertex $j$. The number of edges in $T^F_i$ is $\hat{t}_i= \sum_{j \in S_i} t_j$. The number of vertices in $T^F_i$ which are leaves in $T^F_0$ is $\hat{\lvl}_i =  \sum_{j \in S_i} (\lvl_j - m_j)$ if we declare that $T^L_j$ has 1 leaf if $T^L_j = T^*$ and vertex $j$ is a leaf vertex of $T_s$. The number of vertices in $T^F_i$ which are internal nodes in $T^F_0$ is $\hat{\lvi}_i = \sum_{j \in S_i} \lvi_j$ if we declare that for all $j \neq 0$, $T^L_j$ has an extra leaf if it has root degree 1.
	
	The adjustments are to account for the fact that the root of a proper sub-tree of a larger tree may be an internal node or a leaf relative to the larger tree but not relative to itself.\\ 
\end{observation}

Recall the definition of $\mathcal{G}_n(d_0, d_1, r)$ in Section \ref{Gen function count section}. Let $\mathcal{G}'_n(d_0, d_1, r)$ be defined similarly to $\mathcal{G}_n(d_0, d_1, r)$ but where we declare that root vertices with down degree 1 are internal nodes. Let $\mathcal{G}''_n(d_0, d_1, r)$ be defined similarly to $\mathcal{G}'_n(d_0, d_1, r)$ but where we also declare that the tree on 0 edges has 1 leaf. Let $L(T_s) \subset \{1, \cdots, \lambda\}$ be the set of indices greater than 0 corresponding to leaf vertices in $T_s$. Let $\overline{L}(T_s) \subset \{1, \cdots, \lambda\}$ be the set of indices greater than 0 corresponding to non-leaf vertices in $T_s$.

Let $\vec{t} = (t_0, \cdots, t_{\lambda})$, $\vec{\lvl} = (\lvl_0, \cdots, \lvl_{\lambda})$, $\vec{\lvi} = (\lvi_0, \cdots, \lvi_{\lambda})$ and $\vec{\lvr} = (\lvr_0, \cdots, \lvr_{\lambda})$. Let $\mathcal{H}(T_s, C, \vec{\lvr},\vec{t},\vec{\lvl}, \vec{\lvi})$ be the set of all elements in $\phi^{-1}(T_s, C)$ with fixed $t_i, \lvl_i, \lvi_i, \lvr_i$ given by $\vec{t}, \vec{\lvl}, \vec{\lvi}, \vec{\lvr}$. From the above construction, after making the adjustments specified in Observation \ref{observation about sumtree props 1}, we see that

\begin{eqnarray}
|\mathcal{H}(T_s, C, \vec{\lvr},\vec{t},\vec{\lvl}, \vec{\lvi})| &=& \binom{k}{C}\cdot \left[\binom{\lvl_0}{m_0} \cdot |\mathcal{G}_{t_0}(\lvl_0, \lvi_0, \lvr_0)|\right] \cdot \left[\prod_{i\in \overline{L}(T_s)}  \binom{\lvl_i}{m_i} \cdot |\mathcal{G}'_{t_i}(\lvl_i, \lvi_i, \lvr_i)|\right] \times \nonumber \\
&& \quad\quad\quad \left[\prod_{i \in L(T_s)}  \binom{\lvl_i}{m_i} \cdot |\mathcal{G}''_{t_i}(\lvl_i, \lvi_i, \lvr_i)|\right]
\end{eqnarray} where $\binom{\lvl_i}{m_i}$ counts the number of ways to pick the leaves which will be identified with the root vertices of the fused trees assigned to the children of vertex $i$. Notice that when $\lvl_i < m_i$, $\binom{\lvl_i}{m_i}$ forces the entire term to 0 thus we do not require any extra restriction on the properties of the trees chosen.

We now let $x_i, a_i, b_i$ be variables that count the number of edges in $T_i^L$, the number of vertices in $T_i^L$ which are leaves in $T_0^F$ and the number of vertices in $T_i^L$ which are internal nodes in $T_0^F$, respectively. Also, let $c$ weigh the root degree of $T_0^F$. Let $\vec{x} = (x_0, \cdots, x_\lambda)$, $\vec{a} = (a_0, \cdots, a_\lambda)$ and $\vec{b} = (b_0, \cdots, b_\lambda)$. Let 
\begin{equation}\label{R def 1}
R'_{i}(T_s, t_i, \lvl_i, \lvi_i, \lvr_i) = x_i^{t_i} a_i^{\lvl_i - m_i} b_i^{\lvi_i} c^{\lvr_i}\cdot\binom{\lvl_i}{m_i} \cdot |\mathcal{G}_{t_i}(\lvl_i, \lvi_i, \lvr_i)|
\end{equation}
for $i = 0$, 
\begin{equation}\label{R def 2}
R'_{i}(T_s, t_i, \lvl_i, \lvi_i, \lvr_i) = x_i^{t_i} a_i^{\lvl_i - m_i} b_i^{\lvi_i} \cdot\binom{\lvl_i}{m_i} \cdot |\mathcal{G}'_{t_i}(\lvl_i, \lvi_i, \lvr_i)|
\end{equation} 
for $i \in \overline{L}(T_s)$, and 
\begin{equation}\label{R def 3}
R'_{i}(T_s, t_i, \lvl_i, \lvi_i, \lvr_i) = x_i^{t_i} a_i^{\lvl_i - m_i} b_i^{\lvi_i} \cdot\binom{\lvl_i}{m_i} \cdot |\mathcal{G}''_{t_i}(\lvl_i, \lvi_i, \lvr_i)|
\end{equation} 
for $i \in L(T_s)$. We now let 
\begin{equation}\label{R_i def}
R_i(T_s, x_i, a_i, b_i, c) =  \sum_{t_i, \lvl_i, \lvi_i, \lvr_i \geq 0} R'_{i}(T_s, t_i, \lvl_i, \lvi_i, \lvr_i).
\end{equation}

For convenience, we make the following definitions.
\begin{eqnarray}
\Psi(x, a, b) &=& [1 + (2 - a -b)x]^2  - 4x[1 - (b-1)x]\\
X(x, a, b) &=& - \frac{\partial \Psi(x, a, b)}{\partial x}\\
A(x, a, b) &=& - \frac{\partial \Psi(x, a, b)}{\partial a}\\
B(x, a, b) &=& -  \frac{\partial \Psi(x, a, b)}{\partial b}
\end{eqnarray}

 For conciseness, we let $\Psi_i = \Psi(x_i, a_i, b_i)$, $X_i = X(x_i, a_i, b_i)$, $ A_i = A(x_i, a_i, b_i)$ and $B_i =  B(x_i, a_i, b_i)$. We will simply write $\Psi, X, A, B$ to refer to $\Psi(x, a, b)$, $X(x, a, b)$, $A(x, a, b)$ and $B(x, a, b)$, respectively.
 
 Recall Theorems \ref{G recurrence} and \ref{eval G func} and the definitions in Section \ref{Gen function count section}. From (\ref{R_i def}), we see that 

\begin{equation} \label{R_i i = 0}
R_i(T_s, x_i, a_i, b_i, c) = \frac{1}{m_i!} \cdot \frac{\partial G(x_0,a_0,b_0, c)}{\partial a_0^{m_0}}
\end{equation} for $i = 0$, 

\begin{equation}
R_i(T_s, x_i, a_i, b_i, c) = \frac{1}{m_i!} \cdot \frac{\partial \left(G(x_i,a_i,b_i) + (b_i - 1)G^*(x_i,a_i,b_i)\right)}{\partial a_i^{m_i}} 
\end{equation} for $i \in \overline{L}(T_s)$, and

\begin{equation}
R_i(T_s, x_i, a_i, b_i, c) = \frac{1}{m_i!} \cdot \frac{\partial \left(G(x_i,a_i,b_i) + (b_i - 1)G^*(x_i,a_i,b_i) + a_i - 1\right)}{\partial a_i^{m_i}} 
\end{equation} for $i \in \overline{L}(T_s)$, where the difference between the cases is to account for whether or not trees with root degrees 0 and 1 as having an extra leaf and internal node, respectively. Note that 
$$G(x_i,a_i,b_i) + (b_i - 1)G^*(x_i,a_i,b_i) + a_i - 1 = \frac{G^*(x_i,a_i,b_i)}{x_i}. $$
Recall the definition of the operator $\mathcal{D}_p^m$ in (\ref{D_p definition}).

\begin{lemma} \label{differentation lemma 1}
    Fix $(T_s, C) \in \setoftreecomppairs$ for $k \geq 1$. Let $u_x, u_a, u_b, \inZ_{\geq 0}$ and $u = u_x + u_a + u_b$. Let $\mathbbm{1}_n = -1$ for $n = 0$ and $\mathbbm{1}_n = 1$ otherwise. Let $\mathbbm{1}'_n = 1$ for $n = 0$ and $\mathbbm{1}'_n = 0$ otherwise. We have that, for $0 \leq i \leq \lambda$,
    \begin{multline}
    \mathcal{D}^{u_x}_{x_i}\mathcal{D}^{u_a}_{a_i}\mathcal{D}^{u_b}_{b_i}\left[a_i^{m_i}R_i(T_s, x_i, a_i, b_i, 1)\right] = \mathbbm{1}'_{u+m_i} \cdot Z_i' + \mathbbm{1}_{u+m_i} \cdot  Z_i(u + m_i) \cdot (x_iX_i)^{u_x} \cdot (a_iA_i)^{u_a + m_i} \times\\ 
    (b_iB_i)^{u_b} \cdot \Psi_i^{\frac{1- 2(u+m_i)  }{2}} + O\left(p_i \cdot \Psi_i^{\frac{2- 2(u+m_i)}{2}} \right),
    \end{multline}
    where, $Z_0(n) = \frac{(2n - 3)!!}{2^{n} \cdot m_0!} \cdot \frac{1}{2x_0}$, for $i > 0$,  $Z_i(n) = \frac{(2n - 3)!!}{2^n \cdot m_i!} \cdot \frac{1}{2x_i(1-(b_i-1)x_i)}$, for $i \in \overline{L}$, $Z'_i = \frac{1 + (a_i - b_i)x_i}{2x_i(1-(b_i-1)x_i)}$  and $p_i$ is a sum of quotients of polynomial whose denominators are the products of powers of $x_i$ and $1-(b_i-1)x_i$.
    
\end{lemma}
\begin{proof} Follows by differentiating $G(x,a,b)$, $\frac{G^*(x,a,b)}{x}$ and $\frac{G^*(x,a,b)}{x} + 1 - a$. We note that for $i \not\in \overline{L}$, we omit stating $Z'_i$ since since $u + m_i > 0$.

\end{proof}

\bigskip

Fix $\alpha, \beta \inR$ and $f \inR[t, \lvl, \lvi][n, \gvl, \gvi]$. For $(T_s, C) \in \setoftreecomppairs$ such that $T_s \in \setofplanetrees_\lambda$ and $C = (c_1^*, \cdots, c_\lambda^*)$, let 
\begin{equation}\label{def f_k}
f_k(T_s, C) = \prod_{i=1}^{\lambda} f(\hat{t}_i, \hat{\lvl}_i, \hat{\lvi}_i, \hat{t}_0, \hat{\lvl}_0, \hat{\lvi}_0)^{c_i^*} 
\end{equation} 
such that 
\begin{equation}
\label{H sum}
f_k(T_s, C) = \sum_{j=1}^L w_j \cdot v_j \quad\quad \text{where} \quad\quad v_j =  \prod_{i=0}^\lambda  t_i^{v_{n}(i,j)} \cdot \lvl_i^{v_{d_0}(i,j)} \cdot \lvi_i^{v_{d_1}(i,j)}
\end{equation} where $w_j \inR$ and $v_{n}(i,j)$, $v_{d_0}(i,j)$, $v_{d_1}(i,j) \inZ_{\geq 0}$ for all $0 \leq i \leq \lambda$ and $1 \leq j \leq L$. Furthermore, let $\sigma_{i,j} = v_{n}(i,j) + v_{d_0}(i,j) + v_{d_1}(i,j) + m_i$ and $\{\sigma_{i,j}\}_{> 0} = \{0 \leq i \leq \lambda : \sigma_{i,j} > 0\}$. Finally, let \begin{equation} \label {R_i^j def 1}
R_{i,j}(T_s, x_i, a_i, b_i) = \frac{1}{a^{m_i}} \cdot \mathcal{D}^{v_{n}(i,j)}_{x_i}\mathcal{D}^{v_{d_0}(i,j)}_{a_i}\mathcal{D}^{v_{d_1}(i,j)}_{b_i}\left[a^{m_i}R_i(T_s, x_i, a_i, b_i, 1)\right]
\end{equation} 
for $0 \leq i \leq \lambda$.

\begin{lemma} \label{M generating function expression lemma}
For $f_k(T_s, C) = \sum_{j=1}^L w_j \cdot v_j$ as in (\ref{H sum}), the following equation holds.
\begin{multline}
M_{k}(f, \alpha,\beta, 0)(x) = \sum_{(T_s, C) \in \setoftreecomppairs} \binom{k}{C}\sum_{j=1}^L w_j \cdot  A^{\lambda}\cdot   (e^{-\alpha}A)^{ \sum_{i=0}^{\lambda}v_{d_0}({i,j})}\cdot (e^{-\beta}B)^{\sum_{i=0}^{\lambda} v_{d_1}({i,j})} \times \\
(xX)^{\sum_{i=0}^{\lambda} v_{n}({i,j})} \cdot \Psi^{\frac{|\{\sigma_{i,j}\}_{> 0}| - 2\sum_{i=0}^{\lambda}\sigma_{i,j}}{2}} \cdot (Z')^{1 + \lambda - |\{\sigma_{i,j}\}_{> 0}|} \cdot \left(\prod_{i \in \{\sigma_{i,j}\}_{> 0}} Z_i(\sigma_{i,j})\right)\\
 + O\left(p(x) \cdot \Psi^{\frac{1 + |\{\sigma_{i,j}\}_{> 0}| - 2\sum_{i=0}^{\lambda}\sigma_{i,j}}{2}}\right),
\end{multline} where $Z' = \frac{1 + (e^{-\alpha} - e^{-\beta})x}{2x(1-(e^{-\beta}-1)x)}$, for $i  = 0$, $Z_i(n) = \frac{(2n - 3)!!}{2^{n} \cdot m_i!} \cdot \frac{1}{2x}$, for $i > 0$,  $Z_i(n) = \frac{(2n - 3)!!}{2^n \cdot m_i!} \cdot \frac{1}{2x(1-(e^{-\beta}-1)x)}$ and $p(x)$ is a sum of quotients of polynomial whose denominators are the products of powers of $x$ and $1-(e^{-\beta}-1)x$.

\end{lemma}

\begin{proof}
We observe that
\begin{eqnarray}
M_{k,n}(f, \alpha,\beta, 0) &=& \sum_{T \in \setofplanetrees_ n} \left(\sum_{v \in \overline{\mathcal{V}}(T)} f(T_v, T)\right)^k e^{-E(T)}\nonumber\\
&=& \sum_{T\in \setofplanetrees_n} \left(\sum_{\vec{v} \in \overline{\mathcal{V}}(T)^k} \prod_{i=1}^k f(T_{v_i}, T)\right)e^{-E(T)} \nonumber\\
\label{setupsum1}&=& \sum_{(T_s, C) \in \setoftreecomppairs} \left(\sum_{\substack{(T, \vec{v}) \in \phi^{-1}(T_s, C)\\ T \in \setofplanetrees_n}} e^{-E(T)}\prod_{i=1}^k f(T_{v_i}, T) \right), 
\label{eq M_k,n expansion 1}
\end{eqnarray} where $\vec{v} = (v_1, \cdots, v_k)$. 

We now notice that for any $(T, \vec{v}) \in \phi^{-1}(T_s, C)$ formed by the construction we previously described, with fixed $\vec{t}$, $\vec{\lvl}$, $\vec{\lvi}$ and $\vec{\lvr}$, $\prod_{i=1}^k f(T_{v_i}, T) = f_k(T_s, C)$. We have also shown that the number of such $(T, \vec{v})$ is $|\mathcal{H}(T_s, C,\vec{t},\vec{\lvl}, \vec{\lvi}, \vec{\lvr})|$. Hence, following from (\ref{eq M_k,n expansion 1}), we achieve
\begin{equation} \label{eq 21}
\sum_{\substack{(T, \overline{e}) \in \phi^{-1}(T_s, C)\\ T \in \setofplanetrees_n}} e^{-E(T)}\prod_{i=1}^k f(T_{v_i}, T)  = \sum_{\sum_{i=0}^\lambda t_i = n}\sum_{\vec{\lvl}, \vec{\lvi},\vec{\lvr}  \inZ_{\geq 0}^{\lambda+1}}e^{-E(\hat{\lvl}_0, \hat{\lvi}_0)}|\mathcal{H}(T_s, C,\vec{t},\vec{\lvl}, \vec{\lvi}, \vec{\lvr})|f_k(T_s, C)
\end{equation} 
where $\hat{t}_i, \hat{\lvl}_i, \hat{\lvi}_i$ are defined as in Observation \ref{observation about sumtree props 1} and $E(\hat{\lvl}_0, \hat{\lvi}_0) = \alpha \hat{\lvl}_0 +\beta  \hat{\lvi}_0$. Recall the definitions in (\ref{H sum}) and (\ref{R_i^j def 1}). We now show that

\begin{eqnarray}
M_{k}(f, \alpha,\beta, 0)(x) &=& \sum_{(T_s, C) \in \setoftreecomppairs}\sum_{n \geq 0} x^n \sum_{\substack{(T, \vec{v}) \in \phi^{-1}(T_s, C)\\ T \in \setofplanetrees_n}} e^{-E(T)}\prod_{i=1}^k\mathcal{P}(T_{v_i}, T) \nonumber\\
&=& \sum_{(T_s, C) \in \setoftreecomppairs} \binom{k}{C}\sum_{j=1}^L w_j \prod_{i=0}^\lambda R_{i,j}(T_s, x, e^{-\alpha}, e^{-\beta}). \label{All together using R_i^j}
\end{eqnarray}

We notice that to get $R_{i,j}(T_s, x_i, a_i, b_i)$, we first multiply by $a^{m_i}$ to get the power of $a_i$ to be $p_i$.
The application of the differentiation operator $\mathcal{D}$ then creates the factor of $t_i^{v_{n}(i,j)} \cdot \lvl_i^{v_{d_0}(i,j)} \cdot \lvi_i^{v_{d_1}(i,j)}$ (without changing the power of $x_i, a_i, b_i$). We then divide by $a^{m_i}$ to get the power of $a_i$ to be $p_i - m_i$. Taking the product of all these terms, we now notice that if we set all the $x_i$ to $x$, $a_i$ to $a$ and $b_i$ to $b$, $x$, $a$ and $b$ now weight the number of edges, leaves and internal nodes in $T$ for $(T, \vec{v}) \in \phi^{-1}(T_s, C)$. Finally, we set $a$ to $e^{-\alpha}$ and $b$ to $e^{-\beta}$ to weight by $e^{-E(T)}$. We set $c$ to 1 as, for $\gamma = 0$, we are not weighting by the root degree.\\

By Lemma \ref{differentation lemma 1}, we get an expression for $R_{i,j}(T_s, x, e^{-\alpha}, e^{-\beta})$. We consider when $\sigma_{i,j} = 0$ for some $i$. We notice that the most significant term of $R_{i,j}$ with respect to $\Psi$ has order $\frac{1}{2}$. Thus we get a higher order term in $\prod_{i=0}^\lambda R_{i,j}(T_s, x, e^{-\alpha}, e^{-\beta})$ by taking the term independent of $\Psi$ for $i$ such that $\sigma_{i,j} = 0$. Furthermore, note that any $i$ such that  $\sigma_{i,j} = 0$ must correspond to a leaf vertex of $T_s$ since otherwise $\sigma_{i,j} \geq m_i \geq 1$. Thus $0 \in  \{\sigma_{i,j}\}_{> 0}$ since the root of $T_s$ cannot be a leaf (since $T_s \in \setofplanetrees_{\geq 1}$).

\begin{multline}\label{sum R_i^j sigma = 0}
\prod_{i=0}^\lambda R_{i,j}(T_s, x, e^{-\alpha}, e^{-\beta}) = A^{\sum_{i=0}^{\lambda} m_i} \cdot (e^{-\alpha}A)^{\sum_{i=0}^{\lambda}v_{d_0}({i,j})} \cdot (e^{-\beta}B)^{\sum_{i=0}^{\lambda} v_{d_1}({i,j})} \cdot (xX)^{\sum_{i=0}^{\lambda} v_{n}({i,j})} \times \\
\Psi^{\frac{1 + \lambda - m - 2\sum_{i=0}^{\lambda}\sigma_{i,j}}{2}} \cdot (Z')^m \cdot \left(\prod_{i \in \{\sigma_{i,j}\}_{> 0}} Z_i(\sigma_{i,j}) \right) 
+ O\left(p(x) \cdot \Psi^{\frac{2 + \lambda-m - 2\sum_{i=0}^{\lambda}\sigma_{i,j}}{2}}\right) 
\end{multline} where $p(x)$ is a sum of quotients of polynomials whose denominators are the products of powers of $x$ and $1-(e^{-\beta}-1)x$ and $m$ is the number of $i$ such that $\sigma_{i,j} = 0$ and is thus equal to $1  + \lambda - |\{\sigma_{i,j}\}_{> 0}|$.

To conclude the proof, we see that $\sum_{i=0}^\lambda m_i = \lambda$ since this is the sum of the number of children of all vertices of a tree on $\lambda$ edges.

\end{proof}

\bigskip


Notice that every monomial of $f^k$ is of the form $\prod_{i=1}^k u_i$ where $u_i$ is a monomial in $f$. To go from $f^k$ to $f_k$, in each of the monomials $\prod_{i=1}^k u_i$, we set $t, \lvl, \lvi$ to $\hat{t}_j, \hat{\lvl}_j, \hat{\lvi}_j$ and $n, \gvl, \gvi$ to $\hat{t}_0, \hat{\lvl}_0, \hat{\lvi}_0$, respectively, in $u_i$ from some $j$ for each $i$. After making this change, we can expand each of the $\hat{t}_j$, $\hat{\lvl}_j$, $\hat{\lvi}_j$ to express $\prod_{i=1}^k u_i$ as the sum of monomials in $t_i$, $\lvl_i$, $\lvi_i$. Notice that for every monomial formed in this way, $v_j$, $\Delta\left(v_j|\vec{t}\right) = \Delta\left.\left(\prod_{i=1}^k u_i\right|t, n\right)$, $\Delta\left(v_j|\vec{\lvi}\right) = \Delta\left.\left(\prod_{i=1}^k u_i\right|\lvi, \gvi\right)$ and $\Delta\left(v_j|\vec{\lvl}\right) \leq \Delta\left.\left(\prod_{i=1}^k u_i\right|\lvl, \gvi\right)$. We note that the inequality in the last relation occurs because $\hat{\lvl}_j = \sum_{i \in S_j}p_i - m_i$ may have a constant term. We however also note that the inequality is tight. \\

Let $\delta = \Delta(f|t, \lvl, \lvi, n, \gvl, \gvi)$. Towards simplifying the expression in Lemma \ref{M generating function expression lemma}, we establish the following lemmas.\\

\begin{lemma} \label{product of f observation}
	Fix $k \inN$. Consider $f^k = \sum_{j=1}^{L_k} w'_j \cdot v'_j$ in reduced form where $v'_j$ is a monomial. For any $1 \leq j \leq L_k$, the inequality
	
	$$\Delta(v'_j|t, \lvl, \lvi, n, \gvl, \gvi) = \Delta(v'_j| t, n) + \Delta(v'_j| \lvl, \gvl) + \Delta(v'_j| \lvi, \gvi) \leq \delta k$$ holds and is tight.
\end{lemma}
\begin{proof}
The fact that $\delta k$ is an upper bound for the degree of monomials in the reduced form of $f^k$ should be clear. Every monomial of degree $\delta k$, must be achieved as the product of monomials of degree $\delta$ in $f$. Let $f = \sum_{j=1}^L v_j$ in reduced form such that the $v_{j_1} \cdots v_{j_m}$ are all the monomials on $f$ of degree $\delta k$. Since $\R[t, \lvl, \lvi][n, \gvl, \gvi]$ is an integral domain, $(w_{j_1}\cdot v_{j_1}+ \cdots + w_{j_m}\cdot v_{j_m})^k \neq 0$ since $w_{j_1}\cdot v_{j_1}+ \cdots + w_{j_m}\cdot v_{j_m} \neq 0$. Thus, in reduced form, there is at least one monomial of $f^k$ of degree $\delta k$.

\end{proof}

Recall $f_k = \sum_{j=1}^L w_j \cdot v_j$. Since every monomial in $f_k$, $v_j$, is formed from a monomial $f^k$ in the way we previously described, $\Delta(v_j|\vec{t}, \vec{\lvl},\vec{\lvi}) \leq \delta k$ and this inequality is tight. Clearly $v_n(i,j) = \Delta(v_j|t_i)$, thus $\sum_{i=0}^\lambda v_n(i,j) = \Delta(v_j|\vec{t})$. Similarly, $\sum_{i=0}^\lambda v_{d_0}(i,j) = \Delta(v_j|\vec{\lvl})$ and $\sum_{i=0}^\lambda v_{d_1}(i,j) = \Delta(v_j|\vec{\lvi})$. Thus, we achieve 
\begin{equation} \label{sigma_i,j bound}
\sum_{i=0}^\lambda \sigma_{i,j} = \sum_{i=0}^\lambda  v_n(i,j) + v_{d_0}(i,j) + v_{d_1}(i,j) + m_i \leq \delta k + \lambda
\end{equation} and this inequality is tight.

We will call $v_j$, a mononial of $f_k$, such that $\sum_{i=0}^\lambda\sigma_{i,j} = \delta k + \lambda$ a maximal monomial of $f_k$. We also call $u$, a monomial in $f$, such that $\Delta(u|t, \lvl, \lvi, n, \gvl, \gvi) = \delta$ a maximal monomial of $f$. We now consider, for a maximal monomial $v_j$, when $\sigma_{i,j} = 0$ for some $0 \leq i \leq \lambda$. \\

\begin{lemma} \label{when sigma_i,j = 0 lemma}
There exist a maximal monomial $v_j$ where $\sigma_{i,j} = 0$ for some $0 \leq i \leq \lambda$ if and only if, in the reduced form, $f$ has a maximal monomial, $u$, such that $\Delta(u|t, \lvl, \lvi) = 0$. Furthermore, when $f$ has a maximal monomial, $u$, such that $\Delta(u|t, \lvl, \lvi) = 0$, there exist $j$ where $\sigma_{i,j} = 0$ where vertex $i$ is a leaf of $T_s$.
\end{lemma}

\begin{proof}
For each vertex $i$ corresponding to a non-leaf of $T_s$, $m_i \geq 1$, hence $\sigma_{i,j} \geq 1$. Let vertex $i$ be leaf vertex of $T_s$. 

Assume that every maximal monomial in the the reduced form of $f$ is such that $\Delta(u|t, \lvl, \lvi) \geq 1$. We have seen that $v_j$ is derived from a monomial of $f^k$, $\prod_{i=1}^k u_i$, where $u_i$ must be maximal monomial of $f$. To go from $\prod_{l=1}^k u_l$ to monomials in $f_k$, in at least one of $u_l$, $t, \lvl, \lvi$ are set to $\hat{t}_i = \sum_{l \in S_i} t_l = t_i$, $\hat{\lvl}_i = \sum_{l \in S_i} \lvl_l - m_l = \lvl_i$ and $\hat{\lvi}_i = \sum_{l \in S_i} \lvi_l = \lvi_i$, respectively. Assume (WLOG) $\Delta(u|t) \geq 1$. Thus $t_i$ divides every monomial of $f_k$ from $\prod_{i=1}^k u_i$, including $v_j$. Thus $\sigma_{i,j} \geq v_{n}(i,j) \geq 1$.

Assume, in the reduced form, $f$ has a maximal monomial, $u$, such that $\Delta(u|t, \lvl, \lvi) = 0$, thus $u(n,\gvl,\gvi) = n^{l_n}\gvl^{l_{d_0}}\gvi^{l_{d_1}}$ for some $l_n, l_{d_0}, l_{d_1}$ such that $l_n + l_{d_0} + l_{d_1} = \delta$. Consider $\prod_{l=1}^k u_l$, where $u_l = u$ for all $l$. 

Setting $n, \gvl, \gvi$ to $\hat{t}_0,\hat{\lvl}_0, \hat{\lvi}_0$, respectively, in $\prod_{l=1}^k u_l$ and expanding the term, we get 
$$\prod_{l=1}^k u_l = \hat{t}_0^{kl_n}\hat{\lvl}_0^{kl_{d_0}}\hat{\lvi}_0^{kl_{d_1}} = \left(\sum_{i=0}^\lambda t_i\right)^{kl_n}\left(\sum_{i=0}^\lambda \lvl_i\right)^{kl_{d_0}}\left(\sum_{i=0}^\lambda \lvi_i - m_i\right)^{kl_{d_1}} = t_0^{kl_{n}}\lvl_0^{kl_{d_0}}\lvi_0^{kl_{d_1}} + \cdots.$$ 
Thus $t_0^{kl_{n}}\lvl_0^{kl_{d_0}}\lvi_0^{kl_{d_1}}$ is maximal monomial of $f_k$ with $v_{n}(i,j) = v_{d_0}(i,j) = v_{d_1}(i,j) = 0$. Thus $\sigma_{i,j} = 0$. Notice that this holds whenever vertex $i$ is a leaf vertex of $T_s$.

\end{proof}

From Lemma \ref{M generating function expression lemma}, to achieve the most significant with respect to $\Psi$ in $M_k(f, \alpha, \beta, 0)(x)$, we must minimize $|\{\sigma_{i,j}\}_{> 0}|$. When $f$ does not have a maximal monomial, $u$, such that $\Delta(u|t, \lvl, \lvi) = 0$, $|\{\sigma_{i,j}\}_{> 0}| = \lambda + 1$. When $f$ has a maximal monomial, $u$, such that $\Delta(u|t, \lvl, \lvi) = 0$, $|\{\sigma_{i,j}\}_{> 0}|$ is at least the number of non-leaf vertices in $T_s$ (and this is tight). Thus, taking $T_s$ to be the tree where every non-root vertex is a leaf (the bush), we get $|\{\sigma_{i,j}\}_{> 0}| = 1$.

Finally, we achive the following lemma that tells us that in all cases, the term of $M_k(f, \alpha, \beta, 0)(x)$ from non-maximal monomial of $f_k$ are insignificant with respect to $\Psi$.\\

\begin{lemma} \label{most significan term is profuct of maximal monomials}
	The most significant terms in the sum for $M_k(f, \alpha, \beta, 0)(x)$ in Lemma \ref{M generating function expression lemma} must be achieve from maximal monomial of $f_k$ and $T_s \in T_k$.
\end{lemma}
\begin{proof} 
	When $f$ has a maximal monomial, $u$, such that $\Delta(u|t, \lvl, \lvi) = 0$, By Lemma \ref{when sigma_i,j = 0 lemma}, the maximum order with respect to $\Psi$ of a term in $M_k(f, \alpha, \beta, 0)(x)$ is $\frac{1 - 2(\delta k + k)}{2}$ (when $T_s$ is the bush on $k$ edges).
	
	Assume $f$ has no maximal monomial, $u$, such that $\Delta(u|t, \lvl, \lvi) = 0$. Notice that the order with respect to $\Psi$ of a term in $M_k(f, \alpha, \beta, 0)(x)$ from a maximal monomial of $f_k$ is $\frac{1 - 2\delta k - k}{2}$ (when $T_s \in T_k$). To possibly get a non-maximal monomial of $f_k$ which leads to a term in $M_k(f, \alpha, \beta, 0)(x)$ of higher order with respect to $\Psi$, we must minimize $|\{\sigma_{i,j}\}_{> 0}|$. 
	
	By similar argument to the latter part of the proof of Lemma \ref{when sigma_i,j = 0 lemma}, any monomial, $v_j$, of $f_k$ achieved from a monomial of $f^k$, $\prod_{i=1}^ku_i$, where each of the $u_i$ are maximal monomials of $f$, has $\sigma_{i, j} \geq 1$ for all $i$. Thus, we need only consider monomial, $v_j$, of $f_k$ where achieved from a monomial of $f^k$, $\prod_{i=1}^ku_i$, where at least one of the $u_i$ are maximal monomials of $f$.

	Consider any monomial, $v_j$, of $f_k$ achieved from a monomial of $f^k$, $\prod_{l=1}^ku_l$, where each of the $u_l$ are monomials of $f$ and at least one of which is non-maximal. Let $s \geq 1$ be the number of non-maximal $u_l$. By similar argument the former part of the proof of Lemma \ref{when sigma_i,j = 0 lemma}, the number of $i$ where $\sigma_{i, j} = 0$ is at most $s$. Clearly $\Delta(u_l|n,t, \lvl, \gvl, \lvi, \gvi) \leq \delta - 1$ for all non maximal $u_l$. Thus we see that $\sum_{i=0}^\lambda \sigma_{i,j} \leq \delta(k - s)  + (\delta-1)s + \lambda = \delta k  - s + \lambda$. Considering the order with respect to $\Psi$ of the term in $M_k(f, \alpha, \beta, 0)(x)$ from $v_j$, we see that
	\begin{eqnarray}
	\frac{|\{\sigma_{i,j}\}_{> 0}| - 2\sum_{i=0}^{\lambda}\sigma_{i,j}}{2} &\geq& \frac{1  + s - 2\delta k- \lambda}{2} > \frac{1 - 2\delta k - \lambda}{2}.
	\end{eqnarray} We notice that to minimize $\frac{1 - 2\delta k - \lambda}{2}$, we must maximize $\lambda \leq k$, which is achieved when $T_s \in \setofplanetrees_k$ proving the result.
	
\end{proof}

We now prove the main theorem using the above results.

\begin{proof}[Proof of Theorem \ref{mainTheorem01}]
From Lemma \ref{M generating function expression lemma}, we have an expression for $M_{k}(f, \alpha, \beta, 0)(x)$. We assume (WLOG) that the maximal monomials of $f_k$ are $v_1, \cdots , v_{L_{\text{max}}}$. We now break the proof into 2 cases.\\

\noindent Case 1: All the maximal monomial of $f$, $u$, are such that $\Delta(u|t, \lvl, \lvi) > 0$. By Lemma  \ref{when sigma_i,j = 0 lemma}, for every maximal monomial of $f_k$, $v_k$, $\sigma_{i,j} \geq 1$. By Lemma \ref{most significan term is profuct of maximal monomials}, to get the most significant term with respect to $\Psi$ we need only consider the maximal monomials of $f_k$ when $T_s \in \setofplanetrees_k$. For $(T_s, C) \in \setoftreecomppairs$, when $T_s \in \setofplanetrees_k$, $C = (1, 1, \cdots, 1)$. We now see that 
\begin{eqnarray}
M_{k}(f, \alpha,\beta, 0)(x) &=& \sum_{(T_s, C) \in \setoftreecomppairs}  \sum_{j=1}^L w_j \cdot \binom{k}{C} \prod_{i=0}^\lambda R_{i,j}(x, e^{-\alpha}, e^{-\beta})\\
&=& \sum_{\substack{T_s \in \setofplanetrees_k}}  W_{T_s}(f, x, a, b) \Psi^{\frac{1 - (2\delta + 1)k}{2}} + O\left(\Psi^{\frac{2 - (2\delta + 1)k}{2}}\right)\\
&=& W_f(x, \alpha, \beta) \Psi^{\frac{1- (2\delta + 1)k}{2}} + O\left(\Psi^{\frac{2- (2\delta + 1)k}{2}}\right)
\end{eqnarray} where, for $T_s \in \setofplanetrees_k$,
\begin{multline}
W_{T_s}(f, x, a, b) =  k!\sum_{j=1}^{L_{\text{max}}} w_j \cdot  A^{k}   \left(e^{-\alpha}A\right)^{ \sum_{i=0}^{\lambda}v_{d_0}({i,j})}\left(e^{-\beta}B\right)^{\sum_{i=0}^{\lambda} v_{d_1}({i,j})} \times\\
(xX)^{\sum_{i=0}^{\lambda} v_{n}({i,j})} \cdot  \left(\prod_{0 \leq i \leq k} Z_i(\sigma_{i,j}) \right)
\end{multline} and 
$W_{f}(x, a, b) = \sum_{T_s \in \setofplanetrees} W_{T_s}(f, x, a, b)$.

We now assume further that for any maximal monomial $u$ of $f$, $\Delta(u| t, n) = \delta_n$, $\Delta(u| \lvl, \gvl) = \delta_{d_0}$ and $\Delta(u| \lvi, \gvi) = \delta_{d_1}$ where $\delta_n$, $\delta_{d_0}$ and  $\delta_{d_1}$ are constants that depend only on $f$ and not $u$. Thus for any maximal monomial of $f_k$, $v_j$, we see that $\sum_{i=0}^\lambda v_n(i,j) = \Delta(v_j| \vec{t}) = \delta_nk$. Similarly, $\sum_{i=0}^\lambda v_{d_0}(i,j) = \delta_{d_0}k$ and $\sum_{i=0}^\lambda v_{d_1}(i,j) = \delta_{d_1}k$. Thus, 
\begin{multline}
W_{f}(x, a, b) =  \frac{A^{k}   \left(e^{-\alpha}A\right)^{\delta_{d_0}k}\left(e^{-\beta}B\right)^{\delta_{d_1}k}\left(xX\right)^{\delta_{n}k}}{x^{k + 1}(1 - (e^{-\beta}-1)x)^k}  \cdot \frac{k!}{2^{(2 + \delta)k + 1}} \times\\
\sum_{T_s \in \setofplanetrees}\sum_{j=1}^{L_{\text{max}}} w_j    \left(\prod_{{0 \leq i \leq k}} \frac{(2\sigma_{i,j} - 3)!!}{ m_i!} \right)
\end{multline}

\noindent Case 2: There exists a maximal monomial of $f$, $u$, are such that $\Delta(u|t, \lvl, \lvi) = 0$. Combining Lemma  \ref{when sigma_i,j = 0 lemma} and Lemma \ref{most significan term is profuct of maximal monomials}, to get the most significant term with respect to $\Psi$ we need only consider the maximal monomials of $f_k$ when $T_s$ is the bush on $k$ edges. Thus $C = (1, 1 , \cdots, 1)$. Assume $v_1, \cdots , v_{L_{\text{max}}}$ are the maximal monomials of $f_k$ where $\sigma_{i,j} = 0$ for all $i \neq 0$. We now see that
\begin{eqnarray}
M_{k}(f, \alpha,\beta, 0)(x) &=& \sum_{(T_s, C) \in \setoftreecomppairs}  \sum_{j=1}^L w_j \cdot \binom{k}{C} \prod_{i=0}^\lambda   R_{i,j}(x, e^{-\alpha}, e^{-\beta})\\
&=& W_f(x, \alpha, \beta) \Psi^{\frac{1- (2\delta + 2)k}{2}} + O\left(\Psi^{\frac{2- (2\delta + 2)k}{2}}\right)
\end{eqnarray} where
\begin{multline}
W_{f}(x, a, b) =  k!\sum_{j=1}^{L_{\text{max}}} w_j \cdot  A^{k}   \left(e^{-\alpha}A\right)^{ \sum_{i=0}^{\lambda}v_{d_0}({i,j})}\left(e^{-\beta}B\right)^{\sum_{i=0}^{\lambda} v_{d_1}({i,j})} \times\\
(xX)^{\sum_{i=0}^{\lambda} v_{n}({i,j})} \cdot  (Z')^k \cdot  Z_0.
\end{multline}

We now assume further that for any maximal monomial $u$ of $f$, $\Delta(u| t, n) = \delta_n$, $\Delta(u| \lvl, \gvl) = \delta_{d_0}$ and $\Delta(u| \lvi, \gvi) = \delta_{d_1}$ where $\delta_n$, $\delta_{d_0}$ and $\delta_{d_1}$ are constants that depend only on $f$ and not $u$. Thus for any maximal monomial of $f_k$, $v_j$, we see that $\sum_{i=0}^\lambda v_n(i,j) = \Delta(v_j| \vec{t}) = \lambda\delta_n$. Similarly, $\sum_{i=0}^\lambda v_{d_0}(i,j) = \lambda\delta_{d_0}$ and $\sum_{i=0}^\lambda v_{d_1}(i,j) = \lambda\delta_{d_1}$. Thus, 
\begin{multline}
W_{f}(x, a, b) = \frac{A^{k}   \left(e^{-\alpha}A\right)^{\delta_{d_0}k}\left(e^{-\beta}B\right)^{\delta_{d_1}k}\left(xX\right)^{\delta_{n}k}}{x}  \cdot \left(\frac{1 + (e^{-\alpha} - e^{-\beta})x}{x(1-(e^{-\beta}-1)x)}\right)^k \times\\  
\frac{(2(1+ \delta)k - 3)!!}{2^{(2+ \delta)k + 1}} \cdot\sum_{j=1}^{L_{\text{max}}} w_j.
\end{multline}

Let $\rho = e^{-\alpha} + e^{-\beta} + 2e^{-\frac{\alpha}{2}}$ and $\overline{\rho} = e^{-\alpha} + e^{-\beta} - 2e^{-\frac{\alpha}{2}}$. Notice that $M_{k}(f, \alpha,\beta, 0)(x)$ is the sum of products of derivatives of $G(x, e^{-\alpha}, e^{-\beta})$ and $G^*(x, e^{-\alpha}, e^{-\beta})$. Thus, by Lemma \ref{dom sig lemma}, the dominant singularity in $M_{k}(f, \alpha,\beta, 0)(z)$  occurs at $z = \frac{1}{\rho}$. Thus
\begin{eqnarray} \label{todo ref}
M_{k}(f, \alpha,\beta, 0)(z) &=& \frac{W_f(z, \alpha, \beta)}{(1 - \overline{\rho}z)^{V_k(f)}} \cdot (1 - \rho z)^{-V_k(f)} + O\left(p(x) \cdot (1 - \rho z)^{\frac{1}{2}-V_k(f)}\right)
\end{eqnarray} where $V_k(f) = \frac{(2\delta+ 1)k -1}{2}$ if for every monomial of $f$, $u$, such that $\Delta(u| t, n,  \lvl, \gvl, \lvi, \gvi) = \delta$, $\Delta(u| t, \lvl, \lvi) > 0$  and $V_k(f) = \frac{ 2(\delta+ 1)k - 1}{2}$ otherwise. Notice that $p(x)$ and $\frac{W_f(x, \alpha, \beta)}{(1 - \overline{\rho}x)^{V_k(f)}}$ are analytic in the disk $R = \left\{z \inC: |z| \leq \frac{1}{\rho}\right\}$. Thus by Taylor's Theorem, 
\begin{eqnarray}
M_{k}(f, \alpha,\beta, 0)(z) &=& \frac{W_f\left(\frac{1}{\rho}, \alpha, \beta \right)}{\left(1 - \frac{\overline{\rho}}{\rho}\right)^{V_k(f)}} \cdot (1 - \rho z)^{-V_k(f)} + O\left((1 - \rho z)^{\frac{1}{2}-V_k(f)}\right)
\end{eqnarray}

This completes the proof of the theorem.
\end{proof}

\bigskip

\begin{remark} \label{remark of Q_k} 
	Specifically, when $V_k(f) = \frac{(2\delta+ 1)k -1}{2}$,
    \begin{equation}
        Q_k(f, \delta_{d_0}, \delta_{d_1}, \alpha,\beta) = \sqrt{\rho e^{-\frac{\alpha}{2}}} \cdot \left(\frac{(e^{-\frac{\alpha}{2}}+1)^{\delta_{d_0} - 1} \cdot e^{-\frac{\alpha}{4}(2 \delta_{d_0}-1)} \cdot e^{-\beta \delta_{d_1}}}{\sqrt{\rho^{2\delta_{d_0} + 2\delta_{d_1} - 1}}}\right)^{k}
    \end{equation} 
	and, when $V_k(f) = \frac{2(\delta + 1)k -1}{2}$,
	\begin{equation}
	    Q_k(f, \delta_{d_0}, \delta_{d_1}, \alpha,\beta) = \sqrt{\rho e^{-\frac{\alpha}{2}}} \cdot \left(\frac{(e^{-\frac{\alpha}{2}}+1)^{\delta_{d_0}} \cdot e^{-\frac{\alpha \delta_{d_0}}{2}} \cdot e^{-\beta \delta_{d_1}}}{\rho^{\delta_{d_0} + \delta_{d_1}}}\right)^{k}.
	\end{equation}
\end{remark}

\bigskip

\begin{corollary} Fix $\alpha,\beta \inR$ and $f  \inR[t, \lvl, \lvi][n, \gvl, \gvi]$. Let $\delta = \Delta(f| t, n,  \lvl, \gvl, \lvi, \gvi)$. Let $V'(f) = \frac{2\delta+ 1}{2}$ if for every monomial of $f$, $u$, such that $\Delta(u| t, n,  \lvl, \gvl, \lvi, \gvi) = \delta$, $\Delta(u| t, \lvl, \lvi) > 0$  and $V'(f) = \frac{ 2(\delta+ 1)}{2}$ otherwise. There is a unique distribution with $k$th moment given by $$\lim_{n \rightarrow \infty}\E\left[\left(\frac{\mathcal{P}^f_{(\alpha, \beta, 0)}(\setofplanetrees_n)}{n^{V'(f)}}\right)^k\right].$$

\end{corollary}

\begin{proof}
We first note that $V_k(f) + \frac{1}{2} = V'(f)k$. Let $$f = \sum_{i = 1}^L w_i \cdot t^{a_{i,1}}\lvl^{a_{i,2}} \lvi^{a_{i,3}} n^{a_{i,4}} \gvl^{a_{i,5}} \gvi^{a_{i,6}}$$ where $w_i \inR$, $a_{i,j} \inZ_{\geq 0}$. We notice that for any tree, the number of leaves and the number of internal node is at most the number of edges in the tree. We thus see that for $t, \lvl, \lvi, n, \gvl, \gvi \inZ_{\geq 0}$ where $\lvl, \lvi \leq t$ and $\gvl, \gvi \leq n$,
\begin{eqnarray}
f &\leq& \sum_{i = 1}^L |w_i| \cdot t^{a_{i,1}}\lvl^{a_{i,2}} \lvi^{a_{i,3}} n^{a_{i,4}} \gvl^{a_{i,5}} \gvi^{a_{i,6}}\\
&\leq& \sum_{i = 1}^L |w_i| \cdot t^{\sum_{j = 1}^3 a_{i,j}} \cdot n^{\sum_{j = 4}^6 a_{i,j}}.
\end{eqnarray}

We now break the proof into 2 cases as we did for the proof of Theorem \ref{mainTheorem01}.\\

\noindent Case 1: All the maximal monomial of $f$, $u$, are such that $\Delta(u|t, \lvl, \lvi) > 0$.  For $t \leq n$,
\begin{equation}
f \leq \sum_{i = 1}^L |w_i| \cdot t \cdot n^{\sum_{j = 1}^6 a_{i,j} - 1} \leq \sum_{i = 1}^L |w_i| \cdot t \cdot n^{\delta - 1}.
\end{equation}
Let $f' = \omega \cdot t \cdot n^{\delta - 1}$ and $f'' = t$ where $\omega = \sum_{i = 1}^L |w_i|$. We thus see that $$\mathcal{P}^{f}(T) \leq \mathcal{P}^{f'}(T) = \omega \cdot n(T)^{\delta - 1}\sum_{v \in \overline{\mathcal{V}}}   n(T_v) = \omega \cdot n(T)^{\delta - 1} \cdot \mathcal{P}^{f''}(T) $$ for all $T \in \setofplanetrees_{\geq 0} $. Thus $$\E\left[\left(\frac{\mathcal{P}^f_{(\alpha, \beta, 0)}(\setofplanetrees_n)}{n^{V'(f)}}\right)^k\right] \leq \E\left[\left(\frac{\omega \cdot n^{\delta - 1} \cdot \mathcal{P}^{f''}_{(\alpha, \beta, 0)}(\setofplanetrees_n)}{n^{\frac{ 2\delta+ 1}{2}}}\right)^k\right] = \omega^k \cdot \E\left[\left(\frac{\mathcal{P}^{f''}_{(\alpha, \beta, 0)}(\setofplanetrees_n)}{\sqrt{n^3}}\right)^k\right].$$

Let $\rho = e^{-\alpha} + e^{-\beta} + 2e^{-\frac{\alpha}{2}}$ and $\overline{\rho} = e^{-\alpha} + e^{-\beta} - 2e^{-\frac{\alpha}{2}}$. By the defintion of $M_{k,n}(f'', \alpha, \beta, 0)$, $$\E\left[\left(\frac{\mathcal{P}^{f''}_{(\alpha, \beta, \gamma)}(\setofplanetrees_n)}{\sqrt{n^3}}\right)^k\right] = \frac{M_{k,n}(f'', \alpha, \beta, 0)}{\normConst_{(n,\alpha,\beta,0)} \cdot \sqrt{n^{3k}}}.$$ 

We now apply Theorem \ref{transfer lemma}, Theorem \ref{mainTheorem01}, and Corollary \ref{asymp normalizing const lemma}, to see that 
$$M_k = \lim_{n \rightarrow \infty}\E\left[\left(\frac{\mathcal{P}^{f''}_{(\alpha, \beta, 0)}(\setofplanetrees_n)}{\sqrt{n^3}}\right)^k\right] = \frac{W_f\left(\frac{1}{\rho}, \alpha, \beta \right)}{\left(1 - \frac{\overline{\rho}}{\rho}\right)^{\frac{3k - 1}{2}}}  \cdot \frac{2\sqrt{\pi}}{\sqrt{e^{-\frac{\alpha}{2}} \rho } \cdot \Gamma\left(\frac{3k - 1}{2}\right)}.$$

Applying Remark \ref{remark of Q_k}, we see that $$ \lim_{n \rightarrow \infty}\E\left[\left(\frac{\mathcal{P}^{f''}_{(\alpha, \beta, 0)}(\setofplanetrees_n)}{\sqrt{n^3}}\right)^k\right] = c_0 \cdot c_1^k \cdot \lim_{n \rightarrow \infty}\E\left[\left(\frac{\mathcal{P}^{f''}_{(0,0, 0)}(\setofplanetrees_n)}{\sqrt{n^3}}\right)^k\right]$$ where $c_0, c_1$ are constants. Note that $\mathcal{P}^{f''}$ is a simple subtree additive property with bounded toll function. Thus by \ref{ssap are uniquely detrmined by moments}, the moments of $\frac{\mathcal{P}^{f''}_{(0,0, 0)}(\setofplanetrees_n)}{\sqrt{n^3}}$ satisfy Carleman's condition (Theorem \ref{Carleman's condition}). This thus implies that the moments of $\frac{\mathcal{P}^f_{(\alpha, \beta, 0)}(\setofplanetrees_n)}{n^{V'(f)}}$ also satisfy Carleman's condition. Thus we get the desired result.\\

\noindent Case 2: There exists a maximal monomial of $f$, $u$, are such that $\Delta(u|t, \lvl, \lvi) = 0$. For $t \leq n$,
\begin{equation}
f \leq \sum_{i = 1}^L |w_i| \cdot n^{\sum_{j = 1}^6 a_{i,j}} \leq \sum_{i = 1}^L |w_i| \cdot n^{\delta}.
\end{equation}
Let $f' = \omega \cdot  n^{\delta}$ and $f'' = 1$ where $\omega = \sum_{i = 1}^L |w_i|$. We note that $\mathcal{P}^{f''}_{(\alpha, \beta, 0)}(T) = n$ for all $T \in \setofplanetrees_n$ (since this counted the number of subtrees in the tree). We thus see that $$\mathcal{P}^{f}(T) \leq \mathcal{P}^{f'}(T) = \omega \cdot n(T)^{\delta}\sum_{v \in \overline{\mathcal{V}}}   1 = \omega \cdot n(T)^{\delta} \cdot \mathcal{P}^{f''}(T) $$ for all $T \in \setofplanetrees_{\geq 0} $. Thus \begin{equation}\E\left[\left(\frac{\mathcal{P}^f_{(\alpha, \beta, 0)}(\setofplanetrees_n)}{n^{V'(f)}}\right)^k\right] \leq \E\left[\left(\frac{\omega \cdot n^{\delta} \cdot \mathcal{P}^{f''}_{(\alpha, \beta, 0)}(\setofplanetrees_n)}{n^{\frac{ 2(\delta+ 1)}{2}}}\right)^k\right] = \omega^k \cdot \E\left[\left(\frac{\mathcal{P}^{f''}_{(\alpha, \beta, 0)}(\setofplanetrees_n)}{n}\right)^k\right] = \omega^k.
\end{equation}

We now apply the Carleman's condition to get the desired result in this case.

\end{proof}


%% file: Figures/Figure_1.tex
\input{Figures/convexhull}
\begin{figure}
	\centering
    \begin{tikzpicture}	
	
	\begin{scope}[shift={(3,0)}]
	\begin{scope}[shift={(0,0)}]	
	
	\node (Te) at (0, -4)  [circle,fill=white,inner sep=0pt]{\footnotesize
		$\vec{v} = (v_2, v_3, v_4, v_1, v_2, v_5, v_2, v_3, v_5)$};
	
	\node (r) at (0, 0) [circle,fill,inner sep=1.5pt]{};
	\node (v11) at (-1.5,-1)  [circle,fill,inner sep=1.5pt, label=left:\tiny $v_1$]{};
	\node (v12) at (0,-1)  [circle,fill,inner sep=1.5pt]{};
	\node (v13) at (1.5,-1)  [circle,fill,inner sep=1.5pt]{};
	\node (v21) at (-2,-2)  [circle,fill,inner sep=1.5pt]{};
	\node (v22) at (-1,-2)  [circle,fill,inner sep=1.5pt]{};
	\node (v23) at (0.5,-2)  [circle,fill,inner sep=1.5pt]{};
	\node (v24) at (1.5,-2)  [circle,fill,inner sep=1.5pt, label=left:\tiny $v_4$]{};
	\node (v25) at (2.5,-2)  [circle,fill,inner sep=1.5pt, label=left:\tiny $v_5$]{};
	\node (v31) at (-2.5,-3)  [circle,fill,inner sep=1.5pt, label=left:\tiny $v_2$]{};
	\node (v32) at (-1.5,-3)  [circle,fill,inner sep=1.5pt, label=left:\tiny $v_3$]{};
	\node (v33) at (1.5,-3)  [circle,fill,inner sep=1.5pt]{};

	\draw (r) -- (v11)
	(r) -- (v12) 
	(r) -- (v13)
	(v11) -- (v21)
	(v11) -- (v22)
	(v13) -- (v23)
	(v13) -- (v24)
	(v13) -- (v25)
	(v21) -- (v31)
	(v21) -- (v32)
	(v24) -- (v33);
	
	\end{scope}

	\node (A) at (3.5,-2.1) [circle,fill=white,inner sep=0pt]{};
	\node (B) at (5.5,-2.1)  [circle,fill=white,inner sep=0pt]{};
	\draw[->] (A) -- (B) node [midway, above, fill=white,inner sep=0pt] {\small $\phi$};

	\begin{scope}[shift={(8,-1)}]
	
	\node (C) at (0, -3)  [circle,fill=white,inner sep=0.1pt]{\footnotesize $\comp(T, \vec{v}) = (1,3,2,1,2)$};	
	
	\node (r) at (0, 0) [circle,fill,inner sep=1.5pt, label=left:\tiny $v_0$]{};
	\node (v'11) at (-1,-1)  [circle,fill,inner sep=1.5pt, label=left:\tiny $v_1$]{};
	\node (v'12) at (0,-1)  [circle,fill,inner sep=1.5pt, label=left:\tiny $v_4$]{};
	\node (v'13) at (1,-1)  [circle,fill,inner sep=1.5pt, label=left:\tiny $v_5$]{};
	\node (v'21) at (-1.5,-2)  [circle,fill,inner sep=1.5pt, label=left:\tiny $v_2$]{};
	\node (v'22) at (-0.5,-2)  [circle,fill,inner sep=1.5pt, label=left:\tiny $v_3$]{};

	\draw (r) -- (v'11) 
	(r) -- (v'12)
	(r) -- (v'13)
	(v'11) -- (v'21)
	(v'11) -- (v'22);
	
	\end{scope} 
	
	\end{scope}
	
	\begin{scope}[shift={(0,-6)}]
	
	\node (a) at (0, 0)  [circle,fill=white,inner sep=0.1pt]{\footnotesize $T^L_0$};
	
	\begin{scope}[shift={(0,-.5)}]
	\node (r) at (0, 0) [circle,fill=blue,inner sep=1.5pt]{};
	\node (v11) at (-1,-1)  [circle,fill=red,inner sep=1.5pt]{};
	\node (v12) at (0,-1)  [circle,fill,inner sep=1.5pt]{};
	\node (v13) at (1,-1)  [circle,fill,inner sep=1.5pt]{};
	\node (v23) at (0,-2)  [circle,fill,inner sep=1.5pt]{};
	\node (v24) at (1,-2)  [circle,fill=red,inner sep=1.5pt]{};
	\node (v25) at (2,-2)  [circle,fill=red,inner sep=1.5pt]{};
	
	\draw (r) -- (v11)
	(r) -- (v12) 
	(r) -- (v13)
	(v13) -- (v23)
	(v13) -- (v24)
	(v13) -- (v25);	
	\end{scope}
	
	\node (b) at (4, 0)  [circle,fill=white,inner sep=0.1pt]{\footnotesize $T^L_1$};
	
	\begin{scope}[shift={(4,-.5)}]
	\node (v11) at (0,0)  [circle,fill=blue,inner sep=1.5pt]{};
	\node (v21) at (-0.5,-1)  [circle,fill,inner sep=1.5pt]{};
	\node (v22) at (0.5,-1)  [circle,fill,inner sep=1.5pt]{};
	\node (v31) at (-1,-2)  [circle,fill=red,inner sep=1.5pt]{};
	\node (v32) at (0,-2)  [circle,fill=red,inner sep=1.5pt]{};	
	
	\draw (v11) -- (v21)
	(v11) -- (v22)
	(v21) -- (v31)
	(v21) -- (v32);
	\end{scope}
	
	\node (c) at (5.5, 0)  [circle,fill=white,inner sep=0.1pt]{\footnotesize $T^L_4$};
	
	\begin{scope}[shift={(5.5,-.5)}]
	\node (v11) at (0,0)  [circle,fill=blue,inner sep=1.5pt]{};
	\node (v21) at (0,-1)  [circle,fill,inner sep=1.5pt]{};
	
	\draw (v11) -- (v21);
	\end{scope}
	
	\node (d) at (7, 0)  [circle,fill=white,inner sep=0.1pt]{\footnotesize $T^L_2, T^L_3, T^L_5$};
	\begin{scope}[shift={(7,-.5)}]
	\node (v11) at (0,0)  [circle,fill=blue,inner sep=1.5pt]{};
	\end{scope}

	\begin{scope}[shift={(12,0.5)}]
	
	\node (r) at (0, 0) [circle,fill=blue,inner sep=1.5pt]{};
	\node (v11) at (-1.5,-1)  [circle,fill=blue,inner sep=1.5pt]{};
	\node (v12) at (0,-1)  [circle,fill,inner sep=1.5pt]{};
	\node (v13) at (1.5,-1)  [circle,fill,inner sep=1.5pt]{};
	\node (v21) at (-2,-2)  [circle,fill,inner sep=1.5pt]{};
	\node (v22) at (-1,-2)  [circle,fill,inner sep=1.5pt]{};
	\node (v23) at (0.5,-2)  [circle,fill,inner sep=1.5pt]{};
	\node (v24) at (1.5,-2)  [circle,fill=blue,inner sep=1.5pt]{};
	\node (v25) at (2.5,-2)  [circle,fill=blue,inner sep=1.5pt]{};
	\node (v31) at (-2.5,-3)  [circle,fill=blue,inner sep=1.5pt]{};
	\node (v32) at (-1.5,-3)  [circle,fill=blue,inner sep=1.5pt]{};
	\node (v33) at (1.5,-3)  [circle,fill,inner sep=1.5pt]{};
	
	\draw[red] \convexpath{v33,v24}{0.26cm};
	\draw[red] \convexpath{v11,v22,v32,v31}{0.26cm};
	\draw[red] \convexpath{r,v13,v25,v33,v31,v11}{0.45cm};
	\draw[red] (-2.5,-3) circle (.18cm);
	\draw[red] (-1.5,-3) circle (.18cm);
	\draw[red] (2.5,-2) circle (.18cm);

	\draw (r) -- (v11)
	(r) -- (v12) 
	(r) -- (v13)
	(v11) -- (v21)
	(v11) -- (v22)
	(v13) -- (v23)
	(v13) -- (v24)
	(v13) -- (v25)
	(v21) -- (v31)
	(v21) -- (v32)
	(v24) -- (v33);
	\end{scope}
	
	\end{scope}
	
	\end{tikzpicture}
	\caption{Illustration of the map $\phi$ and the construction of elements in $\phi^{-1}(T_s, C)$. We consider $k = 9$. The top left tree is $T \in \setofplanetrees_{11}$ and the top right tree is $T_s = \skele(T, \vec{v})$. To achieve $(T, \vec{v})$ from the construction we describe, we select local trees $T^L_i$ as in the bottom left trees. The red vertices are the leaf vertices we must choose. The bottom rightmost tree is $T$ formed by the construction. The resulting fused trees $T^F_i$ are outlined in red. The blue vertices are the roots of the local trees.} 
	\label{fig:phiexamplefunction}
\end{figure}

%% file: Figures/convexhull.tex
\newcommand{\convexpath}[2]{
[   
    create hullnodes/.code={
        \global\edef\namelist{#1}
        \foreach [count=\counter] \nodename in \namelist {
            \global\edef\numberofnodes{\counter}
            \node at (\nodename) [draw=none,name=hullnode\counter] {};
        }
        \node at (hullnode\numberofnodes) [name=hullnode0,draw=none] {};
        \pgfmathtruncatemacro\lastnumber{\numberofnodes+1}
        \node at (hullnode1) [name=hullnode\lastnumber,draw=none] {};
    },
    create hullnodes
]
($(hullnode1)!#2!-90:(hullnode0)$)
\foreach [
    evaluate=\currentnode as \previousnode using \currentnode-1,
    evaluate=\currentnode as \nextnode using \currentnode+1
    ] \currentnode in {1,...,\numberofnodes} {
-- ($(hullnode\currentnode)!#2!-90:(hullnode\previousnode)$)
  let \p1 = ($(hullnode\currentnode)!#2!-90:(hullnode\previousnode) - (hullnode\currentnode)$),
    \n1 = {atan2(\y1,\x1)},
    \p2 = ($(hullnode\currentnode)!#2!90:(hullnode\nextnode) - (hullnode\currentnode)$),
    \n2 = {atan2(\y2,\x2)},
    \n{delta} = {-Mod(\n1-\n2,360)}
  in 
    {arc [start angle=\n1, delta angle=\n{delta}, radius=#2]}
}
-- cycle
}

%% file: Proof_of_Moment_Generalization_Theorem_2.tex
\begin{proof}
This proof proceeds very similarly to the proof of Theorem \ref{mainTheorem01}. Thus we will only describe how the proof of this theorem differs from the latter proof.

Restricting the construction in the latter theorem to tree of root degree at most $h$, from $(\ref{R_i def})$, when $i = 0$, $$R_{i,h}(T_s, x_i, a_i, b_i, c) = \frac{1}{m_i!} \cdot \frac{\partial G_h(x_0,a_0,b_0, c)}{\partial a_0^{m_0}}$$ where $G_h(x_0,a_0,b_0, c)$ is the generating function for tree with root degree at most $h$ where the tree are weighted by number of edges, leaves, internal nodes and root degree. Adapting Theorem \ref{G recurrence}, we see that 
\begin{eqnarray}
G_h(x_0,a_0,b_0, c) &=& \sum_{r=0}^h[cG^*(x_0,a_0,b_0)]^r \nonumber\\
&=& \sum_{r=0}^h c^r \left[\left(\frac{1 + (a_0-b_0)x_0}{2(1-(b_0-1)x_0)}\right)^r - r\left(\frac{1 + (a_0-b_0)x_0}{2(1-(b_0-1)x_0)}\right)^{r-1} \Psi^{\frac{1}{2}}\right] + O(\Psi).
\end{eqnarray} Recall that $m_0 \geq 1$ (since for $k \geq 0$, the root of a tree with k edges has degree at least 1). Thus for, $i = 0$,

\begin{eqnarray}
R_{i,h}(T_s, x_i, a_i, b_i, c) &=& R_{i}(T_s, x_i, a_i, b_i, c) \cdot W_h(x_i, a_i, b_i, c) + O\left(\Psi^{\frac{2 - 2m_i}{2}}\right)
\end{eqnarray} where $W_h(x_i, a_i, b_i, c) = \sum_{r=0}^h c^r  \cdot r \cdot 2x_i\left(\frac{1 + (a_i-b_i)x_i}{2(1-(b_i-1)x_i)}\right)^{r-1}$ and $R_{i}(T_s, x_i, a_i, b_i, c)$ is as in (\ref{R_i i = 0}). Thus, we adapt Theorem \ref{differentation lemma 1} to get that for $i = 0$,
    \begin{eqnarray}
    \mathcal{D}^{u_x}_{x_i}\mathcal{D}^{u_a}_{a_i}\mathcal{D}^{u_b}_{b_i}\left[a_i^{m_i}R_{i,h}(T_s, x_i, a_i, b_i, c_i)\right] &=& W_h(x_i, a_i, b_i, c) \cdot \mathcal{D}^{u_x}_{x_i}\mathcal{D}^{u_a}_{a_i}\mathcal{D}^{u_b}_{b_i}\left[a_i^{m_i}R_i(T_s, x_i, a_i, b_i, c_i)\right]\nonumber \\
    && \quad \quad \quad \quad + O\left(p_i(x_i) \cdot \Psi_i^{\frac{2- 2(u+m_i)}{2}} \right),
    \end{eqnarray} where $u_x, u_a, u_b, \inZ_{\geq 0}$ and $u = u_x + u_a + u_b$.
    
From this point, it should be clear that 
\begin{equation}
    M_{k}(\left. f, \alpha,\beta, \gamma \right| h )(x) = W_h(x, e^{-\alpha}, e^{-\beta}, e^{-\gamma}) \cdot M_{k}(f, \alpha,\beta, \gamma )(x) + O\left((1 - \rho x)^{\frac{1}{2}-V_k(f)} \right).
\end{equation}

We see that $W_h(z, e^{-\alpha}, e^{-\beta}, e^{-\gamma})$ is analytic in the disk $R = \left\{z \inC: |z| \leq \frac{1}{\rho}\right\}$ (since $W_h(x, e^{-\alpha}, e^{-\beta}, e^{-\gamma}) $ is a finite sum). Thus 
\begin{equation}
    M_{k}(\left. f, \alpha,\beta, \gamma \right| h )(x) = W_h\left(\frac{1}{\rho}, e^{-\alpha}, e^{-\beta}, e^{-\gamma}\right) \cdot M_{k}(f, \alpha,\beta, \gamma )(x) + O\left((1 - \rho x)^{\frac{1}{2}-V_k(f)} \right).
\end{equation} which is the desired result.

\end{proof}

%% file: Other_Results.tex
\section{Other Results} 
\subsection{Counting trees by leaves, internal nodes and root degree}
\label{Counting trees - enumerations}


In the study of the Nearest Neighbour Thermodynamic Model, the partition function, $\normConst_{(n,\alpha,\beta,\gamma)}$, is of general interest. Thus, we provide enumeration results to assist with the computation this quantity. We define $\treecount_n(m,k,r)$ to be the number of trees on $n$ vertices with $m$ internal nodes, $k$ leaves and root degree $r$. We will omit any of the above parameters in $\treecount_n(m,k,r)$ to denote the number of trees summed over the omitted parameters. 

It is well-known that $\treecount_n(k)$ is given by the Narayana numbers \cite{riordan1968combinatorial, narayana1955treillis}. Hence,

\begin{theorem} For $n \geq k$,
	\label{count trees by leaves} 
	$$\treecount_n(k) = \frac{1}{n} \binom{n}{k} \binom{n}{k-1}.$$ 
\end{theorem}

From the work of Dershowitz and Zaks \cite{dershowitz1980enumerations}, we also know that

\begin{theorem} For $n \geq r$,
	\label{count trees by root} 
	$$\treecount_n(r) = \frac{r}{n} \binom{2n-1-r}{n-1}.$$ 
\end{theorem}

From the work of Donaghey and Shapiro \cite{donaghey1977motzkin}, we know that the number of tree on $n$ edges with no internal nodes is given by the $(n-1)$th Motzkin number. Hence, we have the following result.
\begin{theorem} For $n > m$,
	\label{count trees by intenral} 
	$$\treecount_n(m) = \binom{n-1}{m}M_{n-m-1},$$ where $$M_n = \sum_{k=0}^{\left\lfloor \frac{n}{2} \right\rfloor} \binom{n}{2k}C_k$$ are the Motzkin numbers.
\end{theorem}

The main result of this section is the closed form expression for $\treecount_n(m,k,r)$ given in the following theorem.

\begin{theorem}
	\label{count trees by all params} For $n - m > k > r, $
	$$\treecount_n(m,k,r) = \frac{r}{n} \binom{n}{k+m} \binom{k+m}{k}\binom{k-r-1}{n-m-k-1},$$ and for $n - m =  k = r$,
	$$\treecount_n(m,k,r) = \binom{n-1}{m}.$$
\end{theorem}

Before proving this theorem, we first observe the following useful lemmas.

\begin{lemma}
	\label{psi translation lemma}
	The number of trees with root degree $r$, $m$ internal nodes and $k$ leaves is equal to the number of trees with leftmost path with $r$ edge, $m$ leaves with a left sibling and $n+1-k$ leaves in total.
	
\end{lemma}
\begin{proof}
	We briefly describe the following bijection from $\setofplanetrees_n$ to $\setofplanetrees_n$ given by Dershowitz and Zaks \cite{dershowitz1980enumerations}, which we will denote as $\psi$.  We describe $\psi$ recursively. For $T = T_1 \ltimes T_2 \in \setofplanetrees_{\geq 1}$, $\psi(T) = \psi(T_2) \ltimes \psi(T_1)$ and $\psi(T^*) = T^*$. (See top left of Figure \ref{fig:treebijectionimage} for illustration.) 
	
	Fix $T \in \setofplanetrees_n$. We denote a leaf edge to be the edge incident to a leaf. Notice that leaf edges in $T$ become edges without a right sibling in $\psi(T)$ (the rightmost edge under each vertex). Notice that the number of rightmost edges in a tree is the number of vertices with at least 1 child, hence the number of non-leaf vertices. 
	To describe the effect of $\psi$ on internal node, for an internal node, $v$, of $T$, we denote its parent edge and child edge by $p(v)$ and $c(v)$. Since we assume the root cannot be internal, a vertex is an internal node if and only if it has a parent edge and a child edge which has no siblings. Notice that the edge $c(v)$ and $p(v)$ in $T$ becomes a leaf in $\psi(T)$ and a left sibling of that leaf respectively. Furthermore, every leaf and left sibling pair induce an internal node. Finally notice that the root degree of $T$ is the length of the path from the root to the leaf reached by moving down leftmost edges. (See Figure \ref{fig:treebijectionimage} for illustration.)
	
\end{proof}

\input{Figures/Figure_2.tex}
\begin{lemma}
	\label{corrNumTree}
	The number of trees with root degree $r$, no internal nodes and $k$ leaves is equal to the number of balanced parenthesis sequences with an initial run of $r$ opening parenthesis with $n+1 -k$ occurrences of $'()'$ with the restriction that no $'()'$ is preceded by a closing parenthesis.
\end{lemma}
\begin{proof}
	We consider the standard bijection from plane trees to balances parenthesis sequences, denoted $\psi_2$. Fix $T \in \setofplanetrees_n$. The length of the leftmost path of $T$ is the length of the initial run of opening parenthesis of $\psi_2(T)$. A leaf in $T$ corresponds to the sub-sequence $'()'$ in $\psi_2(T)$.	A leaf of $T$ with a left sibling corresponds to the sub-sequence $')()'$ in $\psi_2(T)$. Thus for $T$ to have no leaves with left sibling, we must avoid the sub-sequence $')()'$. Applying Lemma \ref{psi translation lemma}, we achieve the result.
	
\end{proof}
 
 \begin{lemma}[Cycle Lemma \cite{dershowitz1982cycle}]
 	For any sequence $p_0 p_1\cdots p_{m+n-1}$ of
 	$m$ open parentheses and $n$ close parentheses, where $m > n$, there exist exactly $m - n$ cyclic permutations
 	$$p_i p_{i+1}\cdots p_{m+n-1}p_0\cdots p_{i-1}$$
 	such that the number of opening parenthesis is always greater than the number of closing parenthesis. We will say these sequences are dominating.
 \end{lemma}

We adapt the Cycle Lemma to aid the computation that will following.
\begin{corollary} \label{ratio cycle lemma}
	For any sequence $p = p_0 p_1\cdots p_{m+n-1}$ of
	$m$ open parentheses and $n$ close parentheses, where $m > n$, among all distinct cyclic permutation of $p$, the ratio between the number of dominating sequences and total number of such permutations is $\frac{m-n}{m+n}$.
\end{corollary}
\begin{proof}
    We will denote the permutation $p_i p_{i+1}\cdots p_{m+n-1}p_0\cdots p_{i-1}$ by $p(i)$ (where $i$ is modulo $m+n$). Let $k \leq m + n$ be smallest natural number such that $p(0) = p(k)$. Notice that $p(0) = p(sk)$ for all $s \inZ$.
	
	We now show that for any $i, j \inZ$, $p(i) = p(j)$ if and only if $i \equiv j \mod k$. It should be clear that $p(i) = p(i + sk)$ for all $s,i \inZ$, since $p(0) = p(sk)$ and shifting $p(0)$ and $p(sk)$ by $i$, we get $p(i)$ and $p(i+sk)$, respectively. Alternatively, if we assume $i \not\equiv j \mod k$ (thus $i = j + sk + r$ where $0 <r < k$) and $p(i) = p(j)$, we see that $p(i) = p(j) = p(i + sk + r) = p(i + r)$. Shifting $p(i)$ and $p(i + r)$ by $-i$, we get that $p(0) = p(r)$, a contraction to the minimality of $k$.
	
	Since $p(0) = p(m+n)$, $k$ divides $m+n$. We thus see that the set of distinct cyclic permutations of $p$ is $P = \{ p(i): 0 \leq i < k\} $ and each such permutation occurs $l = \frac{m+n}{k}$ times as a cyclic permutation of $p$. Let $q$ be the number of element of $P$ that are dominating sequences. By the Cycle Lemma, $\frac{m-n}{m+n} = \frac{lq}{lk} = \frac{q}{k}$, proving the result.
	
\end{proof}

\bigskip

\textit{Proof of Theorem \ref{count trees by all params}.} 
We count the number of trees with root degree $r$, no internal nodes and $k$ leaves using the bijection in Corollary \ref{corrNumTree}. Let $l = n+1 -k$. The beginning of all such balanced parenthesis sequences is fixed as a run of $r$ opening parenthesis followed by a closing parenthesis. So we need to count the number of valid suffixes with $n-r$ opening parenthesis, $n-1$ closing parenthesis and $l-1$ occurrences of $'(()'$ and no other occurrences of $'()'$ or symmetrically, the number of valid prefixes with $n-1$ opening parenthesis, $n-r$ closing parenthesis and $l-1$ occurrences of $'())'$ and no other occurrences of $'()'$. 

We count this via the cycle lemma. We count the number of dominating sequence with $n$ opening parenthesis, $n-r$ closing parenthesis and $k'-1$ occurrences of $'())'$ and no other occurrences of $'()'$. We first count all possible cyclic permutations of all such sequences: 

\begin{enumerate}
	\item \textit{Starting with $'('$ and ending in $')'$.} 
	
	All such sequence are achieved by partitioning the $n$ opening parenthesis into $l-1$ (ordered) parts of size 1 or more, denotes $x_1, \cdots, x_{l-1}$ and partitioning the $n-r$ closing parenthesis into $l-1$ parts of size 2 or more (since we want to guarantee that every $'()'$ is proceeded by a closing parenthesis), denotes $y_1, \cdots, y_{l-1}$. The formed sequence is then $x_1y_1 \cdots x_{l-1}y_{l-1}$. Thus the number of sequences is $\binom{n-1}{l-2} \binom{n-r-l}{l-2}$.
	
	\item \textit{Starting with $'('$ and ending in $'('$.} 
	
	Similarly to (1) with the extra step of adding another run of opening parenthesis of size at least 1 at the end of the sequence. These are sequences of the form $x_1y_1 \cdots x_{l-1}y_{l-1}x_{l}$. The number of such sequences is $\binom{n-1}{l-1} \binom{n-r-l}{l-2}$.
	
	\item \textit{Starting with $')'$ and ending in $'('$.} 
	
	We notice that we lose an occurrence of $'()'$ since it must be formed by the last opening parenthesis and the first closing parenthesis (in a cyclic shift). These sequences are similar to the sequences in (1) shifted so they are of the form $y_{l-1}x_1y_1 \cdots x_{l-1}$. Thus the number of such sequences is $\binom{n-1}{l-2} \binom{n-r-l}{l-2}$. 
	
	\item \textit{Starting with $')'$ and ending in $')'$.}
	
	 These are sequences of the form $y_{0}x_1y_1 \cdots x_{l-1}y_{l-1}$. Since, we are counting cyclic permutations of dominating sequences of the desired type, the restriction of each run of closing parenthesis sequences being of size 2 or more is loosen for the first and last runs (since in any cyclic permutation that would lead to a dominating sequences, these 2 runs will become 1 run). Thus the number of sequences is $\binom{n-1}{l-2} \binom{n-r-l+1}{l-1}$.
	
\end{enumerate}

Notice that we counted all the \emph{distinct} cyclic permutations of the desired dominating sequences. Using Corollary \ref{ratio cycle lemma}, for any sequence counted above, in the set of its distinct cyclic permutations, $\frac{r}{2n-r}$ of members of the set are valid dominating sequences. Thus the number of tree with $n$ edges, root degree of $r$ and $k$ leaves and no internal nodes is the sum of the terms derived for each of the above cases scaled by the term $\frac{r}{2n-r}$. After simplifying and setting $l = n+1-k$, we get this to be

$$\frac{r}{n}\binom{n}{k}\binom{k-r-1}{n-k-1}.$$ 

Adding internal nodes is now just a matter of deciding how many to put under each edge. Thus for each tree on $n-m$ edges with root degree of $r$ and $k$ leaves and no internal nodes, there are $\binom{n-1}{m}$ trees with $n$ edges with root degree of $r$ and $k$ leaves and $m$ internal nodes giving us the desired expression (after simplification).\\

\begin{corollary}
\label{count trees by leaves and internal} For $n - m > k $, 
$$\treecount_n(m,k) = \frac{1}{n} \binom{n}{k+m} \binom{k+m}{k}\binom{k}{n-m-k+1},$$ and for $n - m =  k $,
$$\treecount_n(m,k) = \binom{n-1}{m}.$$
\end{corollary}

\begin{proof}
	Let $p, q \inN$. We now evaluate the following sum.
	\begin{eqnarray}
	\sum_{q \geq 0}\sum_{r \geq 0}r\binom{q-r}{p} x^q &=& x\sum_{r \geq 0}rx^{r-1}\sum_{q \geq 0}\binom{q-r}{p}x^{q-r}\\
	&=&  \frac{x^{p+1}}{(1-x)^{p+3}} = \sum_{q \geq 0} \binom{q+1}{p+2} x^{q}.
	\end{eqnarray}
	
	Thus we get $$\sum_{r \geq 0}r\binom{q-r}{p} = \binom{q+1}{p+2}.$$ We now set $q = k-1$ and $p = n - m - k -1$ to achieve the desired result.\\
	
\end{proof}

\begin{corollary}
	\label{count trees by leaves and root} For $n > k > r$, 
	$$\treecount_n(k, r) = \frac{r}{n} \binom{n}{k} \binom{n-r-1}{n-k-1},$$ and for $n =  k = r$,
	$$\treecount_n(k,r) = 1.$$
\end{corollary}

\begin{proof}
	Let $p, q \inN$. We now evaluate the following sum.
	\begin{eqnarray}
	\sum_{m \geq 0}\sum_{p\geq 0}\binom{n-k}{m}\binom{q}{p-m} x^p &=& \sum_{m \geq 0}\binom{n-k}{m}x^m\sum_{p\geq m}\binom{q}{p-m} x^{p-m}\\
	&=& (1+x)^{n-k+q} = \sum_{p \geq 0} \binom{n-k+q}{p} x^{p}.
	\end{eqnarray}
	
	Thus we get $$\sum_{m \geq 0}\frac{r}{n} \binom{n}{k+m} \binom{k+m}{k}\binom{q}{p-m} = \frac{r}{n} \binom{n}{k} \sum_{m \geq 0}\binom{n-k}{m}\binom{q}{p-m} = \frac{r}{n} \binom{n}{k} \binom{n-k+q}{p} .$$ 
	We now set $q = k-r-1$ and $p = n-k-1$ to achieve the desired result.\\
	
\end{proof}





%% file: Figures/Figure_2.tex
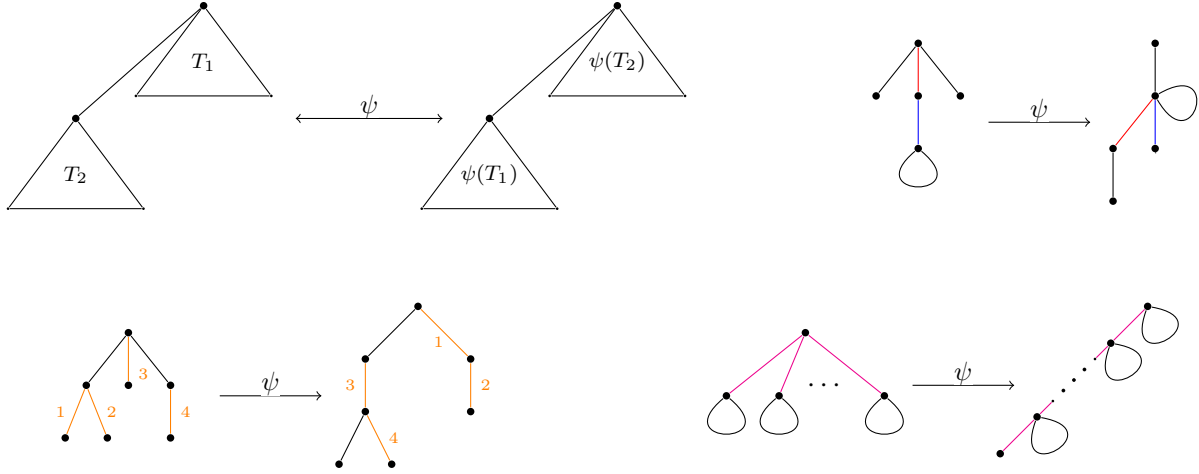
\begin{figure}
	\centering
    \begin{tikzpicture}	
	
	\begin{scope}[shift={(0,0)}]
	\begin{scope}[shift={(0,0)}]	
	
	\node (t11) at (0, 0) [circle,fill,inner sep=1pt]{};
	\node (t12) at (-0.9,-1.2)  [circle,fill,inner sep=0pt]{};
	\node (t13) at (0.9,-1.2)  [circle,fill,inner sep=0pt]{};
	
	\node (t21) at (-1.7, -1.5) [circle,fill,inner sep=1pt]{};
	\node (t22) at (-2.6,-2.7)  [circle,fill,inner sep=0pt]{};
	\node (t23) at (-0.8,-2.7)  [circle,fill,inner sep=0pt]{};

	\draw (t11) -- (t12) -- (t13) -- (t11);
	\draw (t21) -- (t22) -- (t23) -- (t21);
	\draw (t21) -- (t11);
	
	\node (l1) at (0, -0.75)  [circle,fill=white,inner sep=0.1pt]{\scriptsize $T_1$};
	\node (l1) at (-1.7, -2.25)  [circle,fill=white,inner sep=0.1pt]{\scriptsize $T_2$};
	
	\end{scope}
	\begin{scope}[shift={(5.5,0)}]	
	
	\node (t11) at (0, 0) [circle,fill,inner sep=1pt]{};
	\node (t12) at (-0.9,-1.2)  [circle,fill,inner sep=0pt]{};
	\node (t13) at (0.9,-1.2)  [circle,fill,inner sep=0pt]{};
	
	\node (t21) at (-1.7, -1.5) [circle,fill,inner sep=1pt]{};
	\node (t22) at (-2.6,-2.7)  [circle,fill,inner sep=0pt]{};
	\node (t23) at (-0.8,-2.7)  [circle,fill,inner sep=0pt]{};

	\draw (t11) -- (t12) -- (t13) -- (t11);
	\draw (t21) -- (t22) -- (t23) -- (t21);
	\draw (t21) -- (t11);
	
	\node (l1) at (0, -0.75)  [circle,fill=white,inner sep=0.1pt]{\scriptsize $\psi(T_2)$};
	\node (l1) at (-1.7, -2.25)  [circle,fill=white,inner sep=0.1pt]{\scriptsize $\psi(T_1)$};
	
	\end{scope}

	\node (A1) at (1.2, -1.5) [circle,fill=white,inner sep=0pt]{};
	\node (A2) at (3.2,-1.5)  [circle,fill=white,inner sep=0pt]{};
	\draw[<->] (A1) -- (A2) node [midway, above, fill=white,inner sep=0pt] {\small $\psi$};;
	\end{scope}
	
	\begin{scope}[shift={(9.5,-.5)}, scale=0.7]
	\begin{scope}[shift={(0,0)}]	
	
	\node (r) at (0, 0) [circle,fill,inner sep=1pt]{};
	\node (v1) at (-0.8,-1)  [circle,fill,inner sep=1pt]{};
	\node (v2) at (0,-1)  [circle,fill,inner sep=1pt]{};
	\node (v3) at (0.8,-1)  [circle,fill,inner sep=1pt]{};
	\node (v4) at (0,-2)  [circle,fill,inner sep=1pt]{};
	
	\draw (r) -- (v1)
	    (r) -- (v3);
	\draw[red] (r) -- (v2);
	\draw[blue] (v4) -- (v2);
	    
	\draw (v4) to[out=220, in=-40, distance=40] (v4);
	
	\end{scope}
	\begin{scope}[shift={(4.5,0)}]	
	
	\node (r) at (0, 0) [circle,fill,inner sep=1pt]{};
	\node (v1) at (0,-1)  [circle,fill,inner  sep=1pt]{};
	\node (v2) at (-0.8,-2)  [circle,fill,inner sep=1pt]{};
	\node (v3) at (0,-2)  [circle,fill,inner sep=1pt]{};
	\node (v4) at (-0.8,-3)  [circle,fill,inner sep=1pt]{};
	
	\draw (r) -- (v1);
	\draw[red] (v1) -- (v2);
	\draw (v3) -- (v3)
	    (v2) -- (v4);
	\draw[blue] (v1) -- (v3);
	    
	\draw (v1) to[out=300, in=30, distance=40] (v1);
	
	\end{scope}
	
	\node (A1) at (1.3, -1.5) [circle,fill=white,inner sep=0pt]{};
	\node (A2) at (3.3,-1.5)  [circle,fill=white,inner sep=0pt]{};
	\draw[->] (A1) -- (A2)node [midway, above, fill=white,inner sep=0pt] {\small $\psi$};
	\end{scope}
	
	\begin{scope}[shift={(-1, -4)}, scale=0.7]
	\begin{scope}[shift={(0,-0.5)}]	
	
	\node (r) at (0, 0) [circle,fill,inner sep=1pt]{};
	\node (v1) at (-0.8,-1)  [circle,fill,inner sep=1pt]{};
	\node (v2) at (0,-1)  [circle,fill,inner sep=1pt]{};
	\node (v3) at (0.8,-1)  [circle,fill,inner sep=1pt]{};
	\node (v4) at (-0.4,-2)  [circle,fill,inner sep=1pt]{};
	\node (v5) at (-1.2,-2)  [circle,fill,inner sep=1pt]{};
	\node (v6) at (0.8,-2)  [circle,fill,inner sep=1pt]{};
	
	\draw (r) -- (v1)
	(r) -- (v3);
	\draw[orange] (r) -- node[midway, below right] {\tiny{3}}(v2)
	(v1) -- node[midway, right] {\tiny{2}}(v4)
	(v1) -- node[midway, left] {\tiny{1}}(v5)
	(v3) -- node[midway, right] {\tiny{4}}(v6);
	
	\end{scope}
	\begin{scope}[shift={(5.5,0)}]	
	
	\node (r) at (0, 0) [circle,fill,inner sep=1pt]{};
	\node (v1) at (-1,-1)  [circle,fill,inner sep=1pt]{};
	\node (v2) at (1,-1)  [circle,fill,inner sep=1pt]{};
	\node (v3) at (-1,-2)  [circle,fill,inner sep=1pt]{};
	\node (v4) at (1,-2)  [circle,fill,inner sep=1pt]{};
	\node (v5) at (-1.5,-3)  [circle,fill,inner sep=1pt]{};
	\node (v6) at (-0.5,-3)  [circle,fill,inner sep=1pt]{};
	
	\draw (r) -- (v1)
	(v3) -- (v5);
	\draw[orange] (r) -- node[midway, left, pos=.7] {\tiny{1}} (v2)
	(v2) -- node[midway, right] {\tiny{2}} (v4)
	(v3) -- node[midway, right] {\tiny{4}}(v6)
	(v1) -- node[midway, left] {\tiny{3}}(v3);
	
	\end{scope}

	\node (A1) at (1.7, -1.7) [circle,fill=white,inner sep=0pt]{};
	\node (A2) at (3.7,-1.7)  [circle,fill=white,inner sep=0pt]{};
	\draw[->] (A1) -- (A2)node [midway, above, fill=white,inner sep=0pt] {\small $\psi$};
	\end{scope}
	
	\begin{scope}[shift={(8,-4)}, scale=0.7]
	\begin{scope}[shift={(0,-0.5)}]	
	
	\node (r) at (0, 0) [circle,fill,inner sep=1pt]{};
	\node (v1) at (-1.5,-1.2)  [circle,fill,inner sep=1pt]{};
	\node (v2) at (-0.5,-1.2)  [circle,fill,inner sep=1pt]{};
	\node (vn) at (1.5,-1.2)  [circle,fill,inner sep=1pt]{};
	
	\draw[magenta] (r) -- (v1)
	    (r) -- (v2)
	    (r) -- (vn);
	    
	\draw (v1) to[out=220, in=-40, distance=40] (v1);
	\draw (v2) to[out=220, in=-40, distance=40] (v2);
	\draw (vn) to[out=220, in=-40, distance=40] (vn);
	
	\node (l1) at (0.4,-1)  [circle,fill=white,inner sep=0.1pt]{\normalsize $\cdots$};
	
	\end{scope}
	\begin{scope}[shift={(6.5,0)}]	
	
	\node (r) at (0, 0) [circle,fill,inner sep=1pt]{};
	\node (v1) at (-.7,-.7)  [circle,fill,inner sep=1pt]{};
	\node (v-0) at (-1,-1)  [circle,fill,inner sep=0pt]{};
	\node (v-1) at (-1.8,-1.8)  [circle,fill,inner sep=0pt]{};
	\node (vn) at (-2.1,-2.1)  [circle,fill,inner sep=1pt]{};
	\node (vn1) at (-2.8,-2.8)  [circle,fill,inner sep=1pt]{};
	
	\node (k0) at (-1.2,-1.2)  [circle,fill,inner sep=0.5pt]{};
	\node (k1) at (-1.4,-1.4)  [circle,fill,inner sep=0.5pt]{};
	\node (k2) at (-1.6,-1.6)  [circle,fill,inner sep=0.5pt]{};
	
	\draw[magenta] (r) -- (v1) -- (v-0)
	    (v-1) -- (vn) -- (vn1);
	    
	\draw (r) to[out=250, in=-10, distance=40] (r);
	\draw (v1) to[out=250, in=-10, distance=40] (v1);
	\draw (vn) to[out=250, in=-10, distance=40] (vn);
	
	\end{scope}
	
	\node (A1) at (2, -1.5) [circle,fill=white,inner sep=0pt]{};
	\node (A2) at (4,-1.5)  [circle,fill=white,inner sep=0pt]{};
	\draw[->] (A1) -- (A2)node [midway, above, fill=white,inner sep=0pt] {\small $\psi$};
	\end{scope}

	\end{tikzpicture}
	\caption{Illustrations of the effects of map $\psi$ from plane trees to plane trees. The top left diagram illustrates the recurrence of $\psi$. The top right image illustrates what happens to internal nodes under $\psi$. The red and blue egdes are the parent and child edges of the internal node (and their new position under $\psi$). The bottom left image illustrates what happens to leaves under $\psi$. The numbered orange edges are leaves (and their new position under $\psi$). The bottom right image illustrates what happens to the root edges under $\psi$. The purple edges are root edges (and their new position under $\psi$).} 
	\label{fig:treebijectionimage}
\end{figure}

%% file: Discussion.tex
We have characterized the asymptotic behavior of many plane tree properties under a probability distribution where the weight of each tree depends on its number of leaves, number of internal nodes, and root degree.
We have shown that, in the case where the root degree is bounded, the distribution of any subtree additive property under parameters \((\alpha , \beta , \gamma) \) is simply constant multiple of the distribution of that same property under the parameter set \((0, 0, 0)\), where the constant depends on the property in question and the parameters \((\alpha , \beta , \gamma )\).
The probability distribution and tree properties studied were inspired by questions from molecular biology, specifically RNA secondary structure.
Of interest in the biological context, we have shown that the asymptotic distributions of total contact distance, total ladder distance, total leaf to root distance, and total internal node to root distance depend on the parameters \((\alpha , \beta , \gamma )\) in a relatively simple way. 
The fact that these asymptotic distributions can vary only up to a constant multiplier suggests limitations in the Nearest Neighbor Thermodynamic Model.
The explicit form of the scaling constant given in Theorem \ref{maintheorem3} may enable further insights about the exact role of the parameters \(\alpha \) and \(\beta \).

The results are also of independent mathematical interest, as they allow for the examination of a large set of plane tree properties under many natural probability distributions, and show that the behavior of these distributions when changing the values of \((\alpha , \beta , \gamma )\) is actually quite simple.
In particular, this means that, in order to understand the asymptotic distribution of a plane tree property where the plane trees are weighted according to a specific \((\alpha , \beta , \gamma )\), it is sufficient to understand the behavior of the property under parameters \((0, 0, 0)\).
Since the \((0, 0, 0)\) case can often be determined through combinatorial techniques (or is already known), these theorems may be useful in studying combinatorial properties of plane trees under different probability distributions.

We conclude with a few open questions.
\begin{enumerate}
\item In Theorem \ref{mainSimpleAdditveTheorem1},  we characterize the asymptotic behavior of simple subtree additive properties with bounded toll functions. What (if anything) can be proved about the asymptotic behavior simple subtree additive properties with an unbounded toll function? 
\item What is the asymptotic behavior of the distributions of subtree additive properties if we remove the restriction that the root degree of the tree must be bounded? Is a generalization of Theorem \ref{maintheorem3} possible?
\item The \emph{maximum ladder distance} of a plane tree is the length of the longest path in the tree, a.k.a. its diameter. Maximum ladder distance is not a subtree additive property, but it is of interest in the molecular biology context. What is the asymptotic distribution of maximum ladder distance? The framework constructed here does not seem to address properties of this type, and we suspect an entirely different approach would be necessary to answer this question.
\end{enumerate}

%% file: Acknowledgements.tex
The authors would like to thank Prasad Tetali for his mentorship and guidance throughout the preparation of this manuscript. The authors would like to thank Christine Heitsch for introducing them to the plane tree model of RNA secondary structure and proposing some of the questions addressed here. The authors would also like to thank Luc Devroye for consultation and for suggesting valuable references.

Kirkpatrick would like to acknowledge support from the National Science Foundation Graduate Research Fellowship Program under Grant No. DGE-1650044.
Onyeze would like to acknowledge support from the Georgia Institute of Technology President's Undergraduate Research Awards program.

%% file: Appendix.tex
\begin{proof}[Proof of Lemma \ref{lemma expansion}]
	Let $\emptyset \neq U = \{u_1, \cdots, u_q\} \subset W$. We now expand the RHS of (\ref{lemma expansion 1}) to extract the coefficient of $x_{u_1} \cdots x_{u_q}$. We need $V \subset W$ in equation (\ref{lemma expansion 1}) to contain $U$ to get such a term. Let $|V| = n$ and $V-U = \{h_1, \cdots, h_{n - q}\}$. We then expand the term of the form $(x_{u_1} - 1) \cdots (x_{u_q} - 1) \cdot  (x_{h_1} - 1) \cdots (x_{h_{n - q}} - 1)$. Expanding this term, we see that the coefficient of $x_{u_1} \cdots x_{u_q}$ is $(-1)^{n - q}$. There are $\binom{m-q}{n - q}$ ways to pick $V$ with $|V|=n$. Thus the total coefficient of $x_{u_1} \cdots x_{u_q}$ in the RHS of (\ref{lemma expansion 1}) is \begin{equation}
	\sum_{i=q}^m \binom{m-q}{i - q} (-1)^{i - q} = \sum_{i=0}^{m-q} \binom{m-q}{i} (-1)^{i} = 
	\left\{\begin{array}{ll}
	0 & 1 \leq q < m \\
	1 & q = m \\
	\end{array} \right. .
	\end{equation}
	We compute the constant coefficient similarly and get it to be 
	\begin{equation}
	\sum_{q=1}^m \binom{m}{q} (-1)^{q} = -1,
	\end{equation}	
	proving the result.
	
\end{proof}

\begin{proof}[Proof of Lemma \ref{tech lemma 1}]
	Assume $a_1 \geq a_2$. We see the lower bounds as follows. For $a_2 < -1$, 
	\begin{eqnarray}
	\sum_{\substack{n_1 + n_2 = n\\n_1, n_2 \geq 1}}n_1^{a_1} \cdot n_2^{a_2} &=& \sum_{i=1}^{n-1} i^{a_1} \cdot (n-i)^{a_2} \geq \left(n-1\right)^{a_1} \geq c_0 \cdot n^{a_1}.
	\end{eqnarray}
	For $a_2 > -1$,	
	\begin{equation}
	\sum_{\substack{n_1 + n_2 = n\\n_1, n_2 \geq 1}}n_1^{a_1} \cdot n_2^{a_2} \geq \sum_{i=\frac{n}{3}}^{\frac{n}{2}} i^{a_1} \cdot (n-i)^{a_2} \geq \sum_{i=\frac{n}{3}}^{\frac{n}{2}} b_1^{a_1} \cdot b_2^{a_2} = c_0 \cdot n^{a_1 + a_2 + 1},
	\end{equation} where $b_1 = \frac{n}{3}$ for $a_1 > 0$, $b_1 = \frac{n}{2}$ for $a_1 > 0$, $b_2 = \frac{2n}{3}$ for $a_2 > 0$ and $b_2 = \frac{n}{2}$ for $a_1 > 0$.\\
	
	To prove the upper bound, we will use the fact that $\sum_{k=1}^{\infty}\frac{1}{k^{p}}$ converges for $p > 1$. We see that
	\begin{equation*}
	\sum_{\substack{n_1 + n_2 = n\\n_1, n_2 \geq 1}}n_1^{a_1} \cdot n_2^{a_2} \leq \sum_{i=1}^{\frac{n}{2}} i^{a_2} \cdot (n-i)^{a_1} + \sum_{i=\frac{n}{2}}^{n-1} i^{a_2} \cdot (n-i)^{a_1} = \sum_{i=1}^{\frac{n}{2}} i^{a_2} \cdot (n-i)^{a_1} + \sum_{i=1}^{\frac{n}{2}} i^{a_1} \cdot (n-i)^{a_2}.
	\end{equation*}
	Notice that the ratio of 2 terms from the latter sum is $$\frac{i^{a_2} \cdot (n-i)^{a_1}}{(n-i)^{a_2} \cdot i^{a_1}} = \left(\frac{n-i}{i}\right)^{a_1-a_2} \geq 1.$$ 
	Thus, we achieve the upper bound,
	\begin{eqnarray}
	\sum_{\substack{n_1 + n_2 = n\\n_1, n_2 \geq 1}}n_1^{a_1} \cdot n_2^{a_2} &\leq& 2\sum_{i=1}^{\frac{n}{2}} i^{a_2} \cdot (n-i)^{a_1} \leq 2\sum_{i=2}^{n-1}\sum_{\frac{n}{k+1}}^{\frac{n}{k}} i^{a_2} \cdot (n-i)^{a_1}.
	\end{eqnarray}
	
	\noindent For $a_2 \geq 0$ and $a_1 \geq 0$,
	\begin{eqnarray}
	\sum_{\substack{n_1 + n_2 = n\\n_1, n_2 \geq 1}}n_1^{a_1} \cdot n_2^{a_2} &\leq& \sum_{i=1}^{n} n^{a_1 + a_2} =  n^{a_1 +a_2+1}.
	\end{eqnarray}
	
	\noindent For $-1 < a_2 < 0$ and $a_1 \geq 0$,
	\begin{eqnarray}
	\sum_{\substack{n_1 + n_2 = n\\n_1, n_2 \geq 1}}n_1^{a_1} \cdot n_2^{a_2} &\leq&  2\sum_{k=2}^{n-1}\sum_{\frac{n}{k+1}}^{\frac{n}{k}} \left(\frac{n}{k+1}\right)^{a_2} \cdot \left(n-\frac{n}{k+1}\right)^{a_1} \\
	&\leq& n^{a_1 + a_2+1}\left(2\sum_{k=2}^{n-1} \left(\frac{1}{k+1}\right)^{a_2}  \left(\frac{1}{k} - \frac{1}{k+1}\right)\right)\\
	&\leq& n^{a_1 + a_2+1}\left(2\sum_{k=1}^{\infty} \frac{1}{k^{2 + a_2}}\right) = c_0 \cdot n^{a_1 + a_2+1}.
	\end{eqnarray}
	
	\noindent For $-1 < a_2 < 0$ and $a_1 < 0$,
	\begin{eqnarray}
	\sum_{\substack{n_1 + n_2 = n\\n_1, n_2 \geq 1}}n_1^{a_1} \cdot n_2^{a_2} &\leq& 2\sum_{i=2}^{n-1}\sum_{\frac{n}{k+1}}^{\frac{n}{k}} \left(\frac{n}{k+1}\right)^{a_2} \cdot \left(n-\frac{n}{k} \right)^{a_1} \\
	&\leq& n^{a_1 + a_2+1}\left(2\cdot \left(\frac{1}{2} \right)^{a_1}\sum_{i=4}^{n-1} \left(\frac{1}{k+1}\right)^{a_2} \left(\frac{1}{k} - \frac{1}{k+1}\right)\right)\\
	&\leq& n^{a_1 + a_2+1}\left(2^{1 - a_1}\sum_{i=1}^{\infty} \frac{1}{k^{2 + a_2}}\right) = c_0 \cdot n^{a_1 + a_2+1}.
	\end{eqnarray}
	
	\noindent For $a_2 < -1$ and $a_1 \geq 0$,
	\begin{eqnarray}
	\sum_{\substack{n_1 + n_2 = n\\n_1, n_2 \geq 1}}n_1^{a_1} \cdot n_2^{a_2} &\leq& n^{a_1} \sum_{i=1}^{\infty} i^{a_2} = c_0 \cdot n^{a_1}.
	\end{eqnarray}
	
	\noindent For $a_2 < -1$ and $a_1 < 0$,
	\begin{eqnarray}
	\sum_{\substack{n_1 + n_2 = n\\n_1, n_2 \geq 1}}n_1^{a_1} \cdot n_2^{a_2} &\leq& 2\sum_{i=1}^{\frac{n}{2}}i^{a_2} \cdot \left(n-i \right)^{a_1} \leq n^{a_1}\left(2^{1 - a_1}\sum_{i=1}^{\infty}i^{a_2}\right) = c_0 \cdot n^{a_1}.
	\end{eqnarray}

\end{proof}